\documentclass[11pt,reqno]{amsart}
\usepackage[latin1]{inputenc}
\usepackage[english]{babel}
\usepackage{indentfirst}
\usepackage{paralist}
\usepackage{amssymb}
\usepackage{amsthm}
\usepackage{amsmath}
\usepackage{mathrsfs}
\usepackage{amscd}
\usepackage{graphicx}
\usepackage{hyperref}
\usepackage[margin=1in]{geometry}
\usepackage{cite}
\usepackage{lineno} 
\usepackage{enumitem} 
\usepackage{accents} 

\usepackage[notref,notcite,final]{showkeys}

\hypersetup{
	colorlinks=true,
	linkcolor=blue,
	citecolor=red,
	urlcolor=black,
	linktoc=all
}

\usepackage[dvipsnames]{xcolor}

\theoremstyle{plain} 
\newtheorem{theorem}{Theorem}[section] 
\newtheorem{corollary}[theorem]{Corollary} 
\newtheorem{lemma}[theorem]{Lemma}
\newtheorem{prop}[theorem]{Proposition}
\newtheorem{defn}[theorem]{Definition}

\theoremstyle{definition}
\newtheorem{example}[theorem]{Example}
\newtheorem{remark}[theorem]{Remark}

\numberwithin{equation}{section}

\newcommand{\eqlab}[1]{\begin{equation}  \begin{aligned}#1 \end{aligned}\end{equation}} 
\newcommand{\bgs}[1]{\begin{equation*} \begin{aligned}#1\end{aligned}\end{equation*}} 

\newcommand{\R}{\ensuremath{\mathbb{R}}}
\newcommand{\Rn}{\ensuremath{\mathbb{R}^n}}
\newcommand{\N}{\ensuremath{\mathbb{N}}}
\newcommand{\eps}{\ensuremath{\varepsilon}}

\renewcommand{\S}{\mathcal S}
\newcommand{\Ll}{\mathcal L}
\newcommand{\Co}{\mathcal C}
\newcommand{\C}{\mathcal C}

\newcommand{\I}{\mathcal I}
\newcommand{\Ha}{\mathcal H}

\newcommand{\F}{\mathcal F}
\newcommand{\G}{\mathcal G}
\newcommand{\A}{\mathcal A}
\newcommand{\Nl}{\mathcal N}
\newcommand{\Op}{\mathcal O}
\newcommand{\U}{\mathcal U}

\newcommand{\B}{\mathfrak B}
\newcommand{\W}{\mathcal W}

\DeclareMathOperator{\Per}{Per}

\DeclareMathOperator{\diam}{diam}
\DeclareMathOperator{\dist}{dist}
\DeclareMathOperator{\Tail}{Tail}
\newcommand{\PV}{\mbox{\normalfont P.V.}}
\newcommand{\h}{\mathscr{H}}
\DeclareMathOperator{\BV}{BV}
\DeclareMathOperator{\BH}{BH}
\newcommand{\loc}{{\rm loc}}
\DeclareMathOperator{\supp}{supp}

\newcommand{\kers}{|x-y|^{n-1+s}}

\newcommand{\dkers}{\frac{dx\,dy}{|x-y|^{n-1+s}}}

\newcommand{\ubar}[1]{\underaccent{\bar}{#1}}

\def\Xint#1{\mathchoice
{\XXint\displaystyle\textstyle{#1}}%
{\XXint\textstyle\scriptstyle{#1}}%
{\XXint\scriptstyle\scriptscriptstyle{#1}}%
{\XXint\scriptscriptstyle\scriptscriptstyle{#1}}%
\!\int}
\def\XXint#1#2#3{{\setbox0=\hbox{$#1{#2#3}{\int}$ }
\vcenter{\hbox{$#2#3$ }}\kern-.6\wd0}}

\def\dashint{\Xint-}

\renewcommand{\le}{\leqslant}
\renewcommand{\leq}{\leqslant}
\renewcommand{\ge}{\geqslant}
\renewcommand{\geq}{\geqslant}

\title[On nonlocal minimal graphs]{On nonlocal minimal graphs}

\author{Matteo Cozzi}
\author{Luca Lombardini}

\address{
\newline
\textit{Matteo Cozzi}
\newline
Universit\`a degli Studi di Milano, Dipartimento di Matematica, Via Saldini 50, 20133 Milan, Italy
\newline
\textit{E-mail address}: \textit{\tt matteo.cozzi@unimi.it}
}

\address{\vspace{-0.7\baselineskip}
\newline
\textit{Luca Lombardini}
\newline
University of Western Australia, Department of Mathematics and Statistics, 35 Stirling Hwy, Crawley WA 6009, Australia
\newline
\textit{E-mail address}: \textit{\tt luca.lombardini@uwa.edu.au}
}

\thanks{The first author is a member of GNAMPA (INdAM), Italy. A preliminary version of this paper appeared as part of the PhD thesis~\cite{L_phd_th} of the second author}


\makeatletter
\@namedef{subjclassname@2020}{%
  \textup{2020} Mathematics Subject Classification}
\makeatother


\begin{document}

\begin{abstract}
We develop a functional analytic approach for the study of nonlocal minimal graphs. Through this, we establish existence and uniqueness results, a priori estimates, comparison principles, rearrangement inequalities, and the equivalence of several notions of minimizers and solutions.
\end{abstract}

\maketitle

\setcounter{tocdepth}{1}
\tableofcontents

\section{Introduction}

\noindent
Given an integer~$n \ge 1$ and a real number~$s \in (0, 1)$, the \emph{fractional} or \emph{nonlocal~$s$-perimeter} of a measurable set~$E \subseteq \R^{n + 1}$ in an open set~$\Op \subseteq \R^{n + 1}$ is defined as the quantity
$$
\Per_s(E, \Op) := \Ll_s(E\cap\Op,\Co E\cap\Op)+\Ll_s(E\cap\Op,\Co E\setminus\Op)+\Ll_s(E\setminus\Op,\Co E\cap\Op),
$$
where, for two measurable and disjoint sets~$A, B \subseteq \R^{n + 1}$, we write
$$
\Ll_s(A,B):=\int_A\int_B\frac{dX\,dY}{|X-Y|^{n + 1 + s}}
$$
and~$\Co E := \R^{n + 1} \setminus E$ denotes the complement of $E$.

Nonlocal perimeters have been introduced in 2010 in the seminal paper~\cite{CRS10} of Caffarelli, Roquejoffre \& Savin. Since then, there has been a growing interest in their study and, in particular, in the understanding of the properties enjoyed by their minimizers. For more information, we refer the reader to the surveys contained in~\cite{V13},~\cite[Chapter~6]{BV16},~\cite{DV18}, and~\cite[Section~7]{CF17}.

Very recently, several articles have focused on the class of minimizers of~$\Per_s$ that can be written as entire subgraphs of measurable functions. These sets---or, better, their boundaries---are often called \emph{nonlocal $s$-minimal graphs}.
The main source of inspiration for the present paper is the work~\cite{graph}, where Dipierro, Savin \& Valdinoci showed that a set~$E$ which minimizes the~$s$-fractional perimeter in a cylinder~$\Op=\Omega\times\R$ and which is the subgraph of a continuous function~$\varphi$ outside of~$\Op$, must be a subgraph also inside~$\Op$.
The existence of such a minimizing set~$E$ was later proved in~\cite{Cyl} by the second author, while its regularity was fully established in~\cite{CaCo} by~Cabr\'e and the first author. Concerning the qualitative properties of nonlocal minimal graphs, we also mention~\cite{bdary,DSV19_1,DSV19_2} for results on their boundary behavior in low dimension and~\cite{CFL18} for Bernstein-type theorems.

The aim of this work consists in developing an appropriate functional analytic setting for studying nonlocal minimal graphs. Thanks to the introduction of the functional~$\F_s$, constituting a fractional and nonlocal version of the classical area functional, we will establish some new results, summarized here below.

\begin{enumerate}[label=(\alph*),leftmargin=*]
	\item We study the relationship between~$\F_s$ and the fractional~$s$-perimeter of subgraphs.

	\item We prove the existence and uniqueness of minimizers of the functional~$\F_s$, under very mild assumptions on the exterior data.

	\item We obtain a priori estimates for minimizers: an integral estimate in a suitable fractional Sobolev space, as well as local and global~$L^\infty$ bounds.

	\item We study weak and viscosity solutions of the nonlocal mean curvature equation~$\h_s=0$---the Euler-Lagrange equation associated to~$\F_s$.

	\item We obtain a vertical rearrangement inequality, by means of which we show that the subgraph of a minimizer of~$\F_s$ is also a minimizer of the fractional~$s$-perimeter. In particular, this leads us to conclude that graphs with vanishing~$s$-mean curvature are minimizers of the~$s$-perimeter and enables us to extend~\cite[Theorem~1]{graph} to a wider class of exterior data.

	\item We establish the equivalence of minimizers of the functional~$\F_s$, nonlocal minimal graphs, and weak, viscosity, and smooth pointwise solutions of the fractional mean curvature equation.
\end{enumerate}

We will provide the rigorous statements of these results in the remainder of the introduction. First, however, we spend a few words on the motivations that lie behind the definition of the nonlocal area functional~$\F_s$ as well as its relationship with the fractional perimeter and nonlocal minimal graphs.

\subsection{Geometric motivations}\label{Intro_motives}

Our general goal is to define a functional~$\F_s = \F_s(u, \Omega)$, associated to an open set~$\Omega \subseteq \R^n$ and acting on a measurable function~$u: \R^n \to \R$, whose minimization essentially corresponds to the minimization of the~$s$-perimeter of the subgraph
$$
\S_u := \Big\{ (x, t) \in \R^n \times \R : t < u(x) \Big\}
$$
in the cylinder~$\Omega \times \R$. We start by recalling a couple of standard notions of minimality for the nonlocal perimeter~$\Per_s$.

\begin{defn} \label{minim_sperimeter_def}
A measurable set~$E\subseteq\R^{n + 1}$ is said to be~\emph{$s$-minimal} in an open set~$\Op \subseteq \R^{n + 1}$ if~$\Per_s(E,\Op)<\infty$ and
\[
\Per_s(E,\Op)\le\Per_s(F,\Op)\quad\mbox{for every }F\subseteq\R^{n + 1}\mbox{ s.t. }F\setminus\Op=E\setminus\Op.
\]
We also say that~$E$ is~\emph{locally~$s$-minimal} in~$\Op$ if~$E$ is $s$-minimal in every open set~$\Op'\Subset\Op$.
\end{defn}

Note that, in order to be locally~$s$-minimal, a set is only required to have finite~$s$-perimeter in every bounded open set compactly contained in~$\Op$, and not in the whole~$\Op$. When~$\Op$ is bounded and Lipschitz, the two concepts are equivalent---see~\cite[Theorem~1.7]{Cyl}. Conversely, this distinction becomes especially important when dealing with unbounded sets, such as an infinite cylinder~$\Op = \Omega \times \R$. In this case and when the base set~$\Omega \subseteq \R^n$ is bounded and Lipschitz, local~$s$-minimality is equivalent to~$s$-minimality in every truncated cylinder~$\Omega\times(-M,M)$ with~$M>0$---this follows from~\cite[Remark~4.2]{Cyl}. On the other hand, it turns out that the notion of~$s$-minimality in the whole~$\Omega \times \R$ makes no sense in general, since even for a globally bounded function~$u: \R^n \to \R$ it holds
\[
\Per_s(\S_u,\Omega\times\R)= \infty,
\]
no matter how regular~$\Omega$
and~$u$ are---see~\cite[Theorem~1.14 and Corollary~4.5]{Cyl}.

This issue constitutes the main difficulty that one encounters while seeking a good definition of~$\F_s$. In order to circumvent it, it is beneficial to split the~$s$-perimeter functional into different pieces, and investigate which one is really responsible for its blow-up. Hence, following the notation introduced in~\cite{FV13} we write the~$s$-perimeter as the sum
\begin{equation} \label{Perssplit}
\Per_s(E,\Op)=\Per_s^L(E,\Op)+\Per_s^{N \! L}(E,\Op),
\end{equation}
of the local part
$$
\Per_s^L(E,\Op):=\Ll_s(E\cap\Op,\Co E\cap\Op)=\frac{1}{2} \, [\chi_E]_{W^{s,1}(\Op)},
$$
which only reads the interactions occurring inside the open set~$\Op$, and the nonlocal part
\begin{equation} \label{nonlocalPers}
\Per_s^{N \! L}(E,\Op):=\Ll_s(E\cap\Op,\Co E\setminus\Op)+\Ll_s(E\setminus\Op,\Co E\cap\Op)
=\int_\Op\int_{\Co\Op}\frac{|\chi_E(X)-\chi_E(Y)|}{|X-Y|^{n + 1 + s}} \, dX\,dY,
\end{equation}
which takes into account all the remaining interactions. It turns out that~$\Per_s^{N \! L}(\S_u,\Omega\times\R)=+\infty$, even when~$u\in C^\infty_c(\R^n)$, while~$\Per_s^L(\S_u,\Omega\times\R)$ is finite, provided~$u$ is regular enough inside~$\Omega$.

Our analysis here starts from the following observations. Let~$\Omega\subseteq\R^n$ be a bounded open set and let~$u:\Omega\to\R$ be a measurable function. Then,
\begin{equation} \label{Finite_local}
\Per_s^L \! \left( \S_u, \Omega\times\R \right) < \infty \quad \mbox{if and only if} \quad u \in W^{s, 1}(\Omega).
\end{equation}
Moreover, it holds
\begin{equation} \label{Per=F}
\Per_s^L \! \left( \S_u,\Omega\times\R \right)=\A_s(u,\Omega) +\kappa_{s,\Omega},
\end{equation}
where
\[
\A_s(u,\Omega):=\int_\Omega \int_\Omega \G_s \left(\frac{u(x)-u(y)}{|x-y|}\right) \frac{dx\,dy}{|x-y|^{n-1+s}}
\]
with
$$
\G_s(t) := \int_0^t \bigg( \int_0^\tau \frac{d\sigma}{(1 + \sigma^2)^{\frac{n + 1 + s}{2}}} \bigg) \, d\tau,
$$
and where~$\kappa_{s,\Omega}$ is an explicit constant that does not depend on~$u$---see Subsection~\ref{areageom} for the computation.

Formula~\eqref{Per=F} should be compared with the well-known identity between the standard perimeter of a subgraph inside an open infinite vertical cylinder and the area functional of its defining function---see, e.g.,~\cite{GiMa,G84}. This prompts us to interpret~$\A_s$ as a fractional version of the classical area functional. However,~$\A_s$ only accounts for interactions occurring inside~$\Omega$. In order to have a reasonable nonlocal functional it is then quite natural to add to~$\A_s$ the term
\[
\Nl_s(u,\Omega):=2\int_\Omega\int_{\Co\Omega}\G_s \! \left(\frac{u(x)-u(y)}{|x-y|}\right) \frac{dx\,dy}{|x-y|^{n-1+s}},
\]
and define the \emph{fractional $s$-area functional} as
$$
\F_s(u,\Omega):=\A_s(u,\Omega)+\Nl_s(u,\Omega)=\iint_{Q(\Omega)}\G_s \! \left(\frac{u(x)-u(y)}{|x-y|}\right) \frac{dx\,dy}{|x-y|^{n-1+s}},
$$
with~$Q(\Omega):=\R^{2n}\setminus(\Co\Omega)^2$. The functional~$\Nl_s$ gathers part of the interactions contained in~$\Per_s^{N \! L}$. It should be compared with the~$L^1(\partial \Omega)$ term that arises when computing the standard perimeter of a subgraph having jump discontinuities on~$\partial \Omega$---see, e.g.,~\cite[Chapter~14]{G03}. See also~\cite[Theorem~1.8]{Fract_Lu} (and, in particular, identity~(1.9) there), where a similar analogy is observed in the limit~$s \to 1$.

Interestingly enough, we can arrive to the functional~$\F_s$ from a different starting point, involving the fractional mean curvature. The \emph{fractional $s$-mean curvature} of a measurable set~$E\subseteq\R^{n + 1}$ at a point~$X\in\partial E$ is defined as the principal value integral
\[
H_s[E](X):=\PV\int_{\R^{n + 1}}\frac{\chi_{\Co E}(Y)-\chi_E(Y)}{|X - Y|^{n + 1 + s}} \, dY.
\]
As shown in~\cite{CRS10}, this quantity arises when taking the first variation of the~$s$-perimeter. Note that it is well-defined provided the boundary of~$E$ is regular enough around~$x$---see, e.g.,~\cite{AV14}.

In~\cite{CV13}---see also~\cite{AV14,BFV14} and~\cite[Appendix~B.1]{BLV16}---it has been observed that if~$E = \S_u$ is the subgraph of a measurable function~$u: \R^n \to \R$, then its~$s$-mean curvature can be written as an integro-differential operator acting on~$u$. Indeed, the identity
\begin{equation}\label{Curv_subgr_funct_u_id}
H_s[\S_u](x,u(x))=2\,\PV\int_{\R^n}G_s \! \left(\frac{u(x)-u(y)}{|x-y|}\right)\frac{dy}{|x-y|^{n+s}}=:\h_su(x)
\end{equation}
holds true at every point~$x \in \R^n$ around which~$u$ is sufficiently smooth, with
$$
G_s(t):= \G_s'(t) = \int_0^t\frac{d\tau}{(1+\tau^2)^\frac{n+1+s}{2}}.
$$

Now, it is easy to see that~$\F_s$ is the energy corresponding to~$\h_s$. Indeed, it holds
$$
\left. \frac{d}{d\eps} \right|_{\eps=0}\F_s(u+\eps v,\Omega)=\langle\h_su,v\rangle\quad\mbox{for every }v\in C^\infty_c(\Omega),
$$
where we indicate with
\begin{equation} \label{Hsuvpairing}
\langle\h_su,v\rangle:=\int_{\R^n}\int_{\R^n}G_s \! \left(\frac{u(x)-u(y)}{|x-y|}\right)\big(v(x)-v(y)\big)\frac{dx\,dy}{|x-y|^{n+s}} \quad \mbox{for } v \in W^{s, 1}(\R^n)
\end{equation}
the continuous linear functional induced by~$\h_s u$ on~$W^{s, 1}(\R^n)$ via the~$L^2(\R^n)$ pairing. Remarkably, the boundedness of~$G_s$ grants that definition~\eqref{Hsuvpairing} is well-posed under no assumption on the function~$u$, besides measurability.

We have therefore traced a connection between the nonlocal perimeter~$\Per_s$ and the fractional area functional~$\F_s$. From this, one can easily see that nonlocal minimal graphs are in particular minimizers of~$\F_s$. One of the main contributions of the present paper consists in verifying that the converse is also true. Note that the difficulty here lies in the fact that, to establish the~$s$-minimality of a given subgraph~$\S_u$, one has to compare its~$s$-perimeter to those of all its compact perturbations, not only the ones that come in the form of a subgraph, as is natural for~$\F_s$. We solve this issue by means of a suitable rearrangement inequality, which shows that the~$s$-perimeter decreases under vertical rearrangements---which transform non-graphical perturbations into graphical ones.

In view of these remarks, to generate nonlocal minimal graphs one can proceed to minimize the nonlocal area functional~$\F_s$. However, in the above discussion we have overlooked an important issue: the conditions on the function~$u$ needed to ensure that~$\F_s(u,\Omega) < \infty$. As we already observed in~\eqref{Finite_local}-\eqref{Per=F}, the finiteness of~$\A_s(u,\Omega)$ is equivalent to having that~$u|_\Omega\in W^{s,1}(\Omega)$. On the other hand, to have~$\Nl_s(u,\Omega)$ finite, one is led to impose some condition on the \emph{exterior datum}~$\varphi=u|_{\Co\Omega}$, such as
\begin{equation}\label{Exterior_global_cond}
\int_\Omega\left(\int_{\Co\Omega}\frac{|\varphi(y)|}{|x-y|^{n+s}} \, dy\right)dx<\infty.
\end{equation}
Note that~\eqref{Exterior_global_cond} is rather restrictive, in particular as it essentially forces~$\varphi(x)$ to grow slower than~$|x|^s$ at infinity.

Differently to what happens with homogeneous operators like the fractional Laplacian, conditions that pose limitations on the behavior at infinity are typically unnatural in the context of nonlocal perimeters. Indeed, the fractional mean curvature~$\h_s u$ solely requires local regularity assumptions on~$u$ to be well-defined, the minimization problem for~$\Per_s$ has a solution regardless of any assumption on the outside datum (see~\cite[Theorem~3.2]{CRS10} and~\cite[Corollary~1.11]{Cyl}), the gradient estimate of~\cite[Theorem~1.1]{CaCo} only involves local~$L^\infty$ norms, etc.

As we shall see in the next subsection, most of our results avoid the imposition of unnecessary global conditions such as~\eqref{Exterior_global_cond}. To do this, we will introduce a notion of minimality for~$\F_s$ that does not require the finiteness of~$\F_s$---see Definition~\ref{mindef} here below.

\subsection{Definitions and main results}\label{Intro_Defs_Results}

We begin by addressing the existence and uniqueness of minimizers of~$\F_s$ with a given outside datum. Since this result and a few others that will follow are somewhat independent of the underlying geometric structure, we will state them for a more general class of fractional area-type functionals, no longer related to nonlocal perimeters.

Let~$g: \R \to (0, 1]$ be an even continuous function such that~$\int_\R g(t) |t| \, dt$ is finite and consider its first and second antiderivatives
\begin{equation*}
G(t) := \int_0^t g(\tau) \, d\tau \quad \mbox{and} \quad \G(t) := \int_0^t G(\tau) \, d\tau = \int_0^t \left( \int_0^\tau g(\sigma) \, d\sigma \right) d\tau.
\end{equation*}
Given an open set~$\Omega\subseteq\R^n$, a real number~$s \in (0, 1)$, and any measurable function~$u: \R^n \to \R$, we define
\begin{equation} \label{Fcdef}
\F(u, \Omega) := \iint_{Q(\Omega)} \G \! \left( \frac{u(x) - u(y)}{|x - y|} \right) \frac{dx\,dy}{|x - y|^{n - 1 + s}},
\end{equation}
where~$Q(\Omega):=\R^{2n}\setminus(\Co\Omega)^2$.

Of course, by choosing~$g$ equal to
\eqlab{\label{gdef}
g_s(t) := \frac{1}{(1 + t^2)^{\frac{n + 1 + s}{2}}},
}
we recover the functions~$G_s$,~$\G_s$, and the functional~$\F_s$ introduced earlier. Throughout the paper, we will sometimes refer to this choice as the~\emph{geometric framework}---given its connection with the~$s$-perimeter.

We consider the linear space
\begin{equation} \label{Wsdef}
\W^s(\Omega):=\Big\{u:\R^n\to\R\mbox{ measurable} : u|_\Omega\in W^{s,1}(\Omega)\Big\}
\end{equation}
and, for a given measurable function~$\varphi:\Co\Omega\to\R$, its affine subset
\eqlab{\label{domain_w_data}
\W^s_\varphi(\Omega):=\Big\{v\in\W^s(\Omega) : v=\varphi\mbox{ a.e.~in }\Co\Omega\Big\}.
}

Our aim is to minimize the functional~$\F$ within~$\W^s_\varphi(\Omega)$, given
a function~$\varphi$ as exterior datum. As commented before, to avoid the imposition of restrictive assumptions on~$\varphi$---such as~\eqref{Exterior_global_cond}---, we consider the following definition of minimizer, which is well-posed thanks to the fractional Hardy inequality of Proposition~\ref{FHI}---see Lemma~\ref{tartariccio}.

\begin{defn}\label{mindef}
Let~$\Omega \subseteq \R^n$ be a bounded open set
with Lipschitz boundary. A function~$u:\R^n\to\R$ is a \emph{minimizer} of~$\F$ in~$\Omega$ if~$u\in\W^s(\Omega)$ and
$$
\iint_{Q(\Omega)} \left\{ \G \! \left( \frac{u(x) - u(y)}{|x - y|} \right) - \G \! \left( \frac{v(x) - v(y)}{|x - y|} \right) \right\} \frac{dx\,dy}{|x - y|^{n - 1 + s}} \le 0
$$
for every~$v\in\W^s(\Omega)$ such that~$v=u$ a.e.~in~$\Co\Omega$.

Given a measurable function~$\varphi:\Co\Omega\to\R$, we say that~$u:\R^n\to\R$ is a \emph{minimizer of~$\F$ within~$\W^s_\varphi(\Omega)$} if~$u$ is a minimizer of~$\F$ in~$\Omega$ which belongs to~$\W^s_\varphi(\Omega)$.
\end{defn}

When~$\varphi$ satisfies the global integrability condition~\eqref{Exterior_global_cond}, it can be checked that~$\F(u, \Omega)~<~\infty$ for every~$u \in \W^s_\varphi(\Omega)$ and therefore Definition~\ref{mindef} is equivalent to more standard ones---see Subsubsection~\ref{GlobTail_section}. Of course, Definition~\ref{mindef} allows a much higher degree of generality. The downside is that proving the existence of a minimizer becomes a more delicate issue, as one cannot blindly use the direct method of the Calculus of Variations. To overcome this problem, we exploit a ``truncation procedure'' for~$\F$ and introduce the family of functionals~$\{ \F^M \}_{M > 0}$. In the geometric framework, these correspond to the~$s$-perimeter in the truncated cylinder~$\Omega\times(-M,M)$---we refer to Subsection~\ref{areafunc} for the precise definition of the~$\F^M$'s and to Subsection~\ref{areageom} for their relationship with the $s$-perimeter.

The idea consists in proving that each functional~$\F^M$ has a unique minimizer~$u_M$ within its natural domain and to exploit the a priori estimate of Proposition~\ref{Ws1prop}---which holds under the very mild assumption~\eqref{TailinL1} on the exterior datum---to prove that the minimizers~$u_M$'s converge to a function~$u$. This limit function is then shown to be a minimizer in the sense of Definition~\ref{mindef}. We refer to Section~\ref{Linftysec} for the detailed presentation of the argument and to Remark~\ref{Comments_minim} for some further comments about this strategy.

In order to state our first result, we need to introduce some further terminology. Given a set~$\Omega\subseteq\R^n$ and~$\varrho > 0$, we introduce its exterior~$\varrho$-neighborhood 
$$
\Omega_\varrho:= \Big\{ x\in\R^n : \dist(x,\Omega)<\varrho \Big\}.
$$
Given a bounded open set~$\Omega \subseteq \R^n$ and a function~$\varphi:\Co\Omega\to\R$, we define
the tail of~$u$ restricted to a measurable set~$\Op \subseteq \Co \Omega$ and evaluated at a point~$x\in \Omega$ as
\begin{equation} \label{taildef}
\Tail_s(\varphi, \Op;x):=\int_{\Op}\frac{|\varphi(y)|}{|x-y|^{n+s}}\,dy.
\end{equation}
Similar notions of tails were introduced in~\cite{DKP14,DKP16} to study the regularity properties of solutions of homogeneous nonlinear nonlocal equations. We stress that, in contrast to those works, the tails considered here will almost always be restricted to a bounded~$\Op$ (typically, an exterior neighborhood of~$\Omega$), a reflection of the different behavior at infinity allowed by the non-homogeneous operator~$\h_s$---recall the discussion at the end of Subsection~\ref{Intro_motives}.

\begin{theorem} \label{Dirichlet}
Let~$\Omega \subseteq \R^n$ be a bounded open set with Lipschitz boundary.
Then, there exists a constant~$\Theta > 0$, depending only on~$n$,~$s$, and~$g$, such that,
given any function~$\varphi: \Co\Omega\to \R$
with
\begin{equation} \label{TailinL1}
\Tail_s(\varphi, \Omega_{\Theta \diam(\Omega)} \setminus \Omega;\,\cdot\,) \in L^1(\Omega),
\end{equation}
there exists a unique minimizer~$u$ of~$\F$ within~$\W^s_\varphi(\Omega)$. Moreover,~$u$ satisfies
\begin{equation} \label{Ws1estformin}
\| u \|_{W^{s, 1}(\Omega)} \le C \left( \left\| \Tail_s(\varphi,\Omega_{\Theta \diam(\Omega)}\setminus \Omega;\,\cdot\,) \right\|_{L^1(\Omega)} + 1 \right),
\end{equation}
for some constant~$C > 0$ depending only on~$n$,~$s$,~$g$, and~$\Omega$.
\end{theorem}

Theorem~\ref{Dirichlet} provides the existence and uniqueness of a minimizer of~$\F$ having a prescribed outside datum~$\varphi$ satisfying assumption~\eqref{TailinL1}. The uniqueness of said minimizer is particularly noteworthy, since it marks an interesting difference with the classical theory of minimal graphs. Indeed, (generalized) minimizers of the area functional need not be unique, unless one restricts oneself to continuous boundary data---see~\cite[Chapter~15]{G03} and, in particular, Example~15.12 there. This higher rigidity of nonlocal minimizers comes from the fact that~$\F$ (along with its truncations~$\{ \F^M \}$) is \emph{strictly} convex in~$\W^s_\varphi(\Omega)$, whereas the classical area functional (corresponding to the perimeter in the closed cylinder~$\overline{\Omega} \times \R$) is in general only convex.

Note that Theorem~\ref{Dirichlet} holds under hypothesis~\eqref{TailinL1} on the outside datum~$\varphi$. This requirement is much weaker than~\eqref{Exterior_global_cond}, since it poses no restriction on~$\varphi$ outside of~$\Omega_{\Theta \diam(\Omega)}$. For instance, any bounded function in~$\Omega_{\Theta \diam(\Omega)} \setminus \Omega$ satisfies it, and a mild blow-up near~$\partial \Omega$ is allowed as well---see Lemma~\ref{tail_equiv_cond_Lemma} for more information.

The existence of minimizers of~$\F$ comes with the natural energy estimate~\eqref{Ws1estformin}. In the following result, we establish instead a global~$L^\infty$ estimate. It applies in particular to the minimizers obtained in Theorem~\ref{Dirichlet}, provided the outside datum~$\varphi$ is bounded in a sufficiently large neighborhood of~$\Omega$.

\begin{theorem} \label{minareboundedthm}
Let~$\Omega \subseteq \R^n$ be a bounded open set with Lipschitz boundary. There exists a constant~$\Theta > 0$, depending only on~$n$,~$s$, and~$g$, such that if~$u\in\W^s(\Omega)$ is a minimizer of~$\F$ in~$\Omega$, bounded in~$\Omega_{\Theta \diam(\Omega)} \setminus \Omega$, then~$u$ is also bounded in~$\Omega$ and satisfies
$$
\| u \|_{L^\infty(\Omega)} \le \diam(\Omega) + \| u \|_{L^\infty(\Omega_{\Theta \diam(\Omega)} \setminus \Omega)}.
$$
\end{theorem}

Theorem~\ref{minareboundedthm} generalizes an~$L^\infty$ estimate obtained in~\cite[Section~3]{graph} for nonlocal minimal graphs.

The strategy of the proof of Theorem~\ref{minareboundedthm} consists in showing that, by truncating the minimizer~$u$ inside~$\Omega$ at height~$\pm N$, the energy decreases---provided~$N$ is big enough. Then, the conclusion follows from the uniqueness of minimizers. We refer to Proposition~\ref{truncdecreaseprop} for the precise argument.

When the exterior datum of the minimizer~$u$ is not bounded in a neighborhood of~$\Omega$,
one can still obtain that~$u$ is \emph{locally} bounded inside~$\Omega$---see Proposition~\ref{Linftylocprop}. We stress that the local boundedness of minimizers is important to show that the subgraphs of minimizers of the fractional area functional~$\F_s$ are nonlocal minimal graphs, as it allows us to make use of the rearrangement inequality of the forthcoming Theorem~\ref{Persdecreases}.

We also point out that, to obtain the global boundedness of the minimizer~$u\in\W^s_\varphi(\Omega)$ inside~$\Omega$, it is actually enough to require the function~$\varphi$ to be bounded only in a neighborhood~$\Omega_r\setminus\Omega$, with~$r>0$ arbitrarily small---see Proposition~\ref{Bdary_Bdedness_prop}. However, in this case the~$L^\infty$ bound is not as clean as the one of Theorem~\ref{minareboundedthm}.

We now focus our attention on the Euler-Lagrange operator associated with~$\F$, which is the nonlinear integro-differential operator
\bgs{
\h u(x):=2\,\PV\int_{\R^n}G\left(\frac{u(x)-u(y)}{|x-y|}\right)\frac{dy}{|x-y|^{n+s}}.
}
In order for~$\h u(x)$ to be well-defined in the pointwise sense, the function~$u$ must be regular enough (e.g.,~$C^{1,\alpha}$
for some~$\alpha>s$) in a neighborhood~$x$. Regardless, when~$u$ is merely measurable we can still understand~$\h u$ in a distributional sense, setting
\bgs{
\langle\h u,v\rangle:=\int_{\R^n}\int_{\R^n}G\left(\frac{u(x)-u(y)}{|x-y|}\right)\big(v(x)-v(y)\big)\frac{dx\,dy}{|x-y|^{n+s}} \quad \mbox{for } v\in W^{s,1}(\R^n).
}
This observation prompts us to give the following definition of weak solutions.

\begin{defn}\label{weaksoldef}
Let~$\Omega\subseteq\R^n$ be an open set and let~$f\in L^1_\loc(\Omega)$. A measurable function~$u:\R^n\to\R$
is a \emph{weak subsolution} of~$\h u=f$ in~$\Omega$ if
$$
\langle \h u,v\rangle\le\int_\Omega fv\,dx \qquad\mbox{for every }v\in C_c^\infty(\Omega)\mbox{ such that }v\ge0.
$$
We say that~$u$ is a \emph{weak supersolution} of~$\h u=f$ in~$\Omega$ if~$-u$ is a weak subsolution of~$\h(-u)=-f$ in~$\Omega$. If~$u$ is both a weak sub- and supersolution of~$\h u = f$ in~$\Omega$, we call it a \emph{weak solution}.
\end{defn}

We refer to Subsection~\ref{Prelim_EL_op} for comments on the space of the test functions~$v$.

It is immediate to verify that minimizers of~$\F$ are weak solutions of~$\h u = 0$. By the convexity of~$\F$, the converse is also true, provided~$u$ is known a priori to belong to the energy space~$\W^s(\Omega)$. As can be easily checked (again, see Subsection~\ref{Prelim_EL_op}), this requirement is not needed to make sense of the quantity~$\langle \h u,v\rangle$, and is therefore not included in Definition~\ref{weaksoldef}. We believe that it would be interesting to understand whether weak solutions of~$\h u = 0$ have necessarily (locally) finite~$W^{s, 1}$ energy or not.

Furthermore, again by exploiting the convexity of the functional~$\F$, it is easy to verify that weak sub- and supersolutions of~$\h u=0$ belonging to the energy space~$\W^s(\Omega)$ satisfy a comparison principle---see Proposition~\ref{Compari_prop}. This enables us to prove an alternative, Perron-type, existence result for minimizers of~$\F$.

In order to state it, we need some additional terminology.

\begin{defn} \label{locminforFdef}
Given an open set~$\Omega \subseteq \R^n$, we define the space
\bgs{
	\W^s_\loc(\Omega):=\left\{u:\R^n\to\R\mbox{ measurable} : u|_\Omega\in W^{s,1}_\loc(\Omega)\right\}
}
and say that a function~$u: \R^n \to \R$ is a \emph{local minimizer} of~$\F$ in~$\Omega$ if it belongs to~$\W^s_\loc(\Omega)$ and it is a minimizer of~$\F_s$ in~$\Omega'$ (as per Definition~\ref{mindef}) for every open set~$\Omega' \Subset \Omega$ with Lipschitz boundary.
\end{defn}

We then have the following result.

\begin{theorem}\label{Perron_Thm}
	Let~$\Omega\subseteq\R^n$ be an open set and~$\varphi:\Co\Omega\to\R$ be a measurable function. The following statements are equivalent:
	\begin{enumerate}[label=$(\roman*)$,leftmargin=*]
		\item \label{perronequiv1} There exists a local minimizer~$u\in\W^s_\loc(\Omega)$ of~$\F$ in~$\Omega$ such that~$u=\varphi$ a.e.~in~$\Co\Omega$;
		\item \label{perronequiv2} There exist two measurable functions~$\underline{u},\,\overline{u}\in\W^s_\loc(\Omega)$, locally bounded in $\Omega$, which are respectively a weak sub- and supersolution of~$\h u=0$ in $\Omega$ and such that
		$$
		\underline{u}\leq\overline{u} \, \textrm{ a.e.~in~}\R^n\quad\textrm{and}\quad
		\underline{u}\leq\varphi\leq\overline{u} \, \textrm{ a.e.~in~}\Co\Omega.
		$$
	\end{enumerate}		
\end{theorem}

Another way to interpret~$\h u = f$ is via the notion of viscosity solution, a concept that has proven to be very powerful in the context of integro-differential operators. We mention the highly influential papers~\cite{CS09,CS11} by Caffarelli~\& Silvestre, which developed the regularity theory for viscosity solutions of a vast class of fully nonlinear, uniformly elliptic integro-differential equations. The operators~$\h$ considered here do not belong to such class, due to their lack of homogeneity and uniform ellipticity. Nevertheless, a natural definition of viscosity solution can be formulated for them---see Definition~\ref{visc_sol_def} in Section~\ref{ViscoWeak_Sect}. The next result investigates their relationship with weak solutions.

\begin{theorem}\label{Gen_viscweak}
Let~$\Omega\subseteq\R^n$ be an open set,~$f\in C(\overline\Omega)$, and~$u$ be a viscosity subsolution of~$\h u = f$ in~$\Omega$. Then,~$u$ is a weak subsolution of the same equation and~$\max \{ u, k \} \in W^{s, 1}_\loc(\Omega)$ for every~$k \in \R$. Moreover, if~$\Omega$ is bounded with Lipschitz boundary and~$u_+ \in L^\infty(\Omega)$, then~$\max \{u, k\} \in W^{s, 1}(\Omega)$ for every~$k \in \R$.
\end{theorem}

Of course, a similar result can be stated for weak and viscosity supersolutions---and consequently for solutions. Interestingly, this shows that viscosity solutions are weak solutions that (locally) belong to the energy space~$W^{s, 1}$---and therefore, when~$f = 0$, (local) minimizers of the functional~$\F$.

We also point out that, using Theorem~\ref{Perron_Thm} in conjunction with Theorem~\ref{Gen_viscweak}, in order to obtain a local minimizer of~$\F$ in an open set~$\Omega$ it is enough to construct an ordered pair of locally bounded viscosity sub- and supersolutions of~$\h u = 0$ in~$\Omega$.

\subsubsection{Geometric framework}\label{Final_geom_intro}

We already commented in Subsection~\ref{Intro_motives} on the (essential) equivalence of the~$s$-perimeter and the functional~$\F_s$---corresponding to the choice~$g = g_s$ with~$g_s$ as in~\eqref{gdef}. In particular, we mentioned how nonlocal minimal graphs are minimizers of~$\F_s$. Here, we prove that the converse is true as well. Of course, the subgraph of a minimizer of~$\F_s$ has less~$s$-perimeter than any of its \emph{graphical} perturbations. The main difficulty in showing that it is indeed~$s$-minimal rests in verifying that this is true also for perturbations that are not themselves subgraphs.

To establish this fact, we show that, when suitably rearranging a set in the vertical direction to turn it into a subgraph, the nonlocal perimeter decreases. Given a set~$E \subseteq \R^{n + 1}$, we consider the function~$w_E : \R^n \to [- \infty, + \infty]$ defined by
\eqlab{\label{rearr_func_def}
w_E(x) := \lim_{R \rightarrow +\infty} \left( \int_{-R}^R \chi_{E}(x, t) \, dt - R \right) \quad \mbox{for all } x \in \R^n.
}
The subgraph~$\S_{w_E}$ is then the vertical rearrangement of~$E$. Roughly speaking, it is obtained by translating down each connected component (segment) of~$E \cap \left( \{ x \} \times \R \right)$, for a fixed~$x \in \R^n$, until they are joined together to form a connected set.

\begin{theorem} \label{Persdecreases}
Let~$\Omega \subseteq \R^n$ be a bounded open set. Let~$E \subseteq \R^{n + 1}$ be such that~$E \setminus (\Omega\times\R)$ is a subgraph and
\begin{equation} \label{EboundedinOmega}
\Omega \times (-\infty, -M) \subseteq E \cap (\Omega\times\R) \subseteq \Omega \times (-\infty, M),
\end{equation}
for some~$M > 0$. Then,
$$
\Per_s \! \big(\S_{w_E}, \Omega\times(-M,M)\big) \le \Per_s \! \big(E, \Omega\times(-M,M)\big).
$$
Furthermore, if~$\Per_s \! \big( E, \Omega \times (- M, M) \big)$ is finite, then the inequality is strict unless~$E=\S_{w_E}$ up to a set of measure zero.
\end{theorem}

From this result, it easily follows that minimizers of~$\F_s$ are~$s$-minimal graphs. We formalize this in the following theorem, which brings together all notions of minimizers and solutions discussed earlier. Recall that the definitions of locally~$s$-minimal sets and local minimizers of~$\F_s$ are respectively given in Definitions~\ref{minim_sperimeter_def} and~\ref{locminforFdef}.

\begin{theorem}\label{equiv_intro}
Let~$\Omega\subseteq\R^n$ be an open set and~$u: \R^n \to \R$ be a measurable function. Then, the following are equivalent:
\begin{enumerate}[label=$(\roman*)$,leftmargin=*]
\item \label{equiv:uvisc} $u$ is a viscosity solution of~$\h_s u=0$ in~$\Omega$;
\item \label{equiv:ulocweak} $u \in \W^s_\loc(\Omega)$ and~$u$ is a weak solution of~$\h_s u=0$ in~$\Omega$;
\item \label{equiv:ulocmin} $u$ is a local minimizer of~$\F_s$ in~$\Omega$;
\item \label{equiv:Sgumin} $\S_u$ is locally $s$-minimal in~$\Omega\times\R$;
\item \label{equiv:upwise} $u|_\Omega \in C^\infty(\Omega)$ and~$u$ is a pointwise solution of~$\h_s u=0$ in~$\Omega$.
\end{enumerate}
If~$\Omega$ is bounded with Lipschitz boundary and~$u|_\Omega \in L^\infty(\Omega)$, then all the above are also equivalent to:
\begin{enumerate}[label=$(\roman*)'$,leftmargin=*]
\setcounter{enumi}{1}
\item \label{equiv:uglobweak} $u \in \W^s(\Omega)$ and~$u$ is a weak solution of~$\h_s u=0$ in~$\Omega$;
\item \label{equiv:uglobmin} $u$ is a minimizer of~$\F_s$ in~$\Omega$.
\end{enumerate}
\end{theorem}

If~$u$ is continuous in~$\Omega$ and~$\S_u$ is a \emph{geometric} viscosity solution of~$H_s[\S_u] = 0$ on~$\partial \S_u \cap \left( \Omega \times \R \right)$ in the sense of~\cite[Section~5]{CRS10}, then~$u$ is a viscosity solution of~$\h_s u = 0$ in~$\Omega$. Thus, Theorem~\ref{equiv_intro} also yields the equivalence between the concepts of geometric viscosity solution and minimizer for continuous graphs. We point out that a similar result has also been established by Cabr\'e~\cite{C19} by means of a suitable notion of calibration.

Note that the first part of Theorem~\ref{equiv_intro} does not require the boundedness of~$\Omega$. In particular, it holds with~$\Omega = \R^n$ and thus provides a characterization for entire solutions of~$\h_s u = 0$.

On the other hand, by combining Theorems~\ref{equiv_intro} and~\ref{Dirichlet}, we may conclude that, if~$\Omega$ is a bounded, Lipschitz set and~$\varphi: \C \Omega \to \R$ is a sufficiently regular outside datum (satisfying hypothesis~\eqref{TailinL1} with~$\Theta$ given by Theorem~\ref{Dirichlet}), then there exists a unique locally~$s$-minimal graph~$\S_u$ in~$\Omega \times \R$ such that~$u = \varphi$ a.e.~in~$\C \Omega$. Observe that this does not rule out a priori the existence of a different set~$E$ (not a subgraph) which also locally minimizes the~$s$-perimeter in~$\Omega \times \R$ and agrees with~$\{ t < \varphi(x) \}$ outside of it. In view of Theorem~\ref{Persdecreases} (see, in particular, hypothesis~\eqref{EboundedinOmega} there), this non-graphical minimizer~$E$ could not be ``vertically bounded'' within~$\Omega \times \R$. Hence, taking advantage of the~``$L^\infty$ bound'' of~\cite[Lemma~3.3]{graph}, we can find sufficient conditions on~$\Omega$ and~$\varphi$ which exclude the existence of such a set~$E$.

\begin{theorem}\label{last_Theorem}
Let~$\Omega\subseteq\R^n$ be an open set with boundary of class~$C^2$ and such that~$\Omega \subseteq B_{R_0}$ for some~$R_0 > 0$. There exists a radius~$R > R_0$, depending only on~$n$,~$s$, and~$\Omega$, such that, if~$\varphi:\Co\Omega\to\R$ is a measurable function, bounded in~$B_R \setminus \Omega$, then there exists a unique locally~$s$-minimal set~$E$ in~$\Omega \times \R$ such that
$$
E \setminus \left( \Omega \times \R \right) = \Big\{ (x, t) \in \Co \Omega \times \R : t < \varphi(x) \Big\}.
$$
The set~$E$ is the subgraph~$\S_u$ of a measurable function~$u: \R^n \to \R$ with~$u|_\Omega \in L^\infty(\Omega) \cap C^\infty(\Omega)$. If in addition~$\varphi \in C(\Omega_r \setminus \Omega)$ for some~$r > 0$, then~$u|_{\overline{\Omega}} \in C(\overline{\Omega})$.
\end{theorem}

Theorem~\ref{last_Theorem} extends~\cite[Theorem~1.1]{graph} to a larger family of exterior data and represents one of the few uniqueness results available for~$s$-minimal surfaces.

The interior regularity of the minimizer~$u$ is a consequence of the results of~\cite{CaCo}. Its continuity up to the boundary---established in Proposition~\ref{unifcontprop}---mainly follows from the regularity theory for the obstacle problem for the~$s$-perimeter developed by Caffarelli, De Silva \& Savin~\cite{CDSS16}. We stress that~$u$ is in general \emph{not} continuous across the boundary of~$\Omega$, as so-called \emph{boundary stickiness} phenomena may occur---see the works~\cite{bdary,DSV19_2,DSV19_1} by Dipierro, Savin \& Valdinoci for examples and comments on the genericity of this circumstance. We also mention the forthcoming paper~\cite{LuCla}, where this behavior will be investigated in the case of a small fractional parameter~$s$ and in the presence of obstacles.

\subsection{Organization of the paper}

The rest of the paper is structured as follows.

In Section~\ref{prelimsec} we collect some ancillary results to be used throughout the paper. They concern the functions~$g, G, \G$, the functional~$\F$ and its first variation~$\h$, and the relationship between the fractional area functional~$\F_s$ and the fractional perimeter of subgraphs. In the final Subsection~\ref{Compa_sect} we show that weak sub- and supersolutions having finite energy satisfy a comparison principle.

Section~\ref{sect3} contains the proofs of a few a priori estimates. In Subsection~\ref{Ws1_estimates_sect} we establish the~$W^{s,1}$ estimate mentioned in Theorem~\ref{Dirichlet}---see Proposition~\ref{Ws1prop}. The other two Subsections~\ref{locboundsub} and~\ref{globboundsub} are respectively devoted to local and global~$L^\infty$ bounds---leading in particular to Theorem~\ref{minareboundedthm}.

Section~\ref{ViscoWeak_Sect} is concerned with the study of viscosity solutions of the equation~$\h u=f$. After introducing their definition and discussing their properties, we will address their relationship with weak solutions---thus proving Theorem~\ref{Gen_viscweak}.

In Section~\ref{Linftysec} we deal with the existence of minimizers of~$\F$ in~$\Omega$ with prescribed outside data~$\varphi$. In Subsection~\ref{minviadirect} we deal with the case of a bounded Lipschitz~$\Omega$ and~$\varphi$ satisfying condition~\eqref{TailinL1}. We study the truncated functionals~$\F^M$ and use the Direct Method of the Calculus of Variations to show that they have a unique minimizer~$u_M$---see Proposition~\ref{ertyui}. Thanks to the estimate of Proposition~\ref{Ws1prop}, we then have the necessary compactness to obtain a minimizer of~$\F$ and conclude the proof of Theorem~\ref{Dirichlet}. Subsubsection~\ref{GlobTail_section} briefly considers the simpler case in which the exterior datum satisfies the global integrability condition~\eqref{Exterior_global_cond}. In Subsection~\ref{Perron_proof} we instead consider a general open set~$\Omega$ and measurable function~$\varphi$. We prove Theorem~\ref{Perron_Thm}, concerning the possibility of establishing the existence of a local minimizer via the construction of an ordered pair of weak sub- and supersolutions of~$\h u = 0$.

Section~\ref{Rearrange_Sect} contains the proof of the rearrangement inequality of Theorem~\ref{Persdecreases}, carried out in Subsection~\ref{rearrangsub}. For this, we first need a one-dimensional rearrangement inequality, which we prove in a rather general setting in Subsection~\ref{Onedim_rearr_Sect}.

In the short Section~\ref{Equiv_proof_Sect} we use several results addressed earlier to establish the equivalence of minimizers and weak/viscosity/pointwise solutions, as claimed by Theorem~\ref{equiv_intro}.

Section~\ref{Unif_Cont_Sect} is devoted to the proof of Theorem~\ref{last_Theorem}.

Finally, Appendix~\ref{app} contains some known results which are used in the paper. We mention in particular the fractional Hardy-type inequality of Proposition~\ref{FHI}, which guarantees that the notion of minimality considered in Definition~\ref{mindef} is well-posed.

\section{Auxiliary results}\label{prelimsec}

\noindent
We gather here some preliminary results on the functions~$g, G, \G$, on truncated versions of the fractional area functional~$\F$, on their relationship (when~$g = g_s$) with the~$s$-perimeter, and on the associated Euler-Lagrange operator~$\h$. All these topics will be dealt with in separate subsections. To conclude the section, we will obtain a comparison principle for minimizers of the functional~$\F$.

\subsection{Elementary properties of the functions~$g$,~$G$, and~$\G$}

We begin by recalling the following definitions given in the Introduction. 
We consider a continuous function~$g: \R \to \R$ satisfying
\begin{align}
\label{geven}
g(t) = g(-t) & \quad \mbox{for every } t \in \R, \\
\label{gbounds}
0 < g \le 1 & \quad \mbox{in } \R,
\end{align}
and
\begin{equation} \label{intglelambda}
\lambda := \int_{0}^{+\infty} t g(t) \, dt < \infty.
\end{equation}
In light of these requirements, we have that
\eqlab{\label{gintegr}
\Lambda := \int_\R g(t) \, dt \le 2(\lambda+1)<\infty.
}
As remarked in the Introduction, it is easily seen that the function~$g_s$ defined in~\eqref{gdef} satisfies these assumptions. When considering~$g_s$, we will write~$\Lambda_s:=\int_\R g_s(t)\,dt$.

Associated to a general~$g$, we have the functions
\begin{equation} \label{Gdefs}
G(t) := \int_0^t g(\tau) \, d\tau, \qquad \G(t) := \int_0^t G(\tau) \, d\tau = \int_0^t \left( \int_0^\tau g(\sigma) \, d\sigma \right) d\tau,
\end{equation}
and
\begin{equation} \label{Gbardef}
\overline{G}(t) := \int_{-\infty}^t g(\tau) \, d\tau = \int_{-t}^{+\infty} g(\tau) \, d\tau,
\end{equation}
defined for every~$t \in \R$. Notice that
\begin{equation} \label{GbarG}
\overline{G}(t) = \frac{\Lambda}{2} + G(t) \quad \mbox{for every } t \in \R.
\end{equation}

It is also convenient to introduce here the following notation for cylinders, which will be consinstently used throughout the paper:
\begin{equation} \label{cylnotation}
\Omega^M:=\Omega\times(-M,M)\quad \mbox{for } M > 0\qquad\textrm{and}\qquad\Omega^\infty:=\Omega\times\R.
\end{equation}

The following lemma collects the main properties of these functions that will be used in the forthcoming sections.

\begin{lemma} \label{gsprop}
The functions~$G$ and~$\G$ are respectively of class~$C^1$ and~$C^2$. Furthermore, the following facts hold true.
\begin{enumerate}[label=$(\roman*)$,leftmargin=*]
\item \label{gsprop_G} The function~$G$ is odd, increasing, satisfies~$G(0) = 0$ and
\begin{equation} \label{Gbounds}
c_\star \min \{ 1, |t| \} \le |G(t)| \le \min \left\{ \frac{\Lambda}{2}, |t| \right\} \quad \mbox{for every } t \in \R,
\end{equation}
where
\eqlab{\label{cstardef}
c_\star=c_\star(g):=\inf_{t\in[0,1]}g(t) > 0.
}
Moreover,
\eqlab{\label{lip_G_bla}
|G(t)-G(\tau)|\le|t-\tau|\quad \mbox{for every } t,\tau \in \R.
}
\item \label{gsprop_GG} The function~$\G$ is even, increasing on~$[0,\infty)$, strictly convex, and such that~$\G(0) = 0$. It satisfies
\begin{align}
\label{GGbounds}
\frac{c_\star}{2} \min \left\{ |t|, t^2 \right\} & \le \G(t) \le \frac{t^2}{2}, \\
\label{GGbetterbound}
\frac{\Lambda}{2} |t| - \lambda & \le \G(t) \le \frac{\Lambda}{2} |t|,
\end{align}
for every~$t \in \R$, and
\eqlab{\label{Lip_Gcal}
|\G(t)-\G(\tau)|\le\frac{\Lambda}{2} \, |t-\tau|\quad \mbox{for every } t,\tau \in \R.
}
\end{enumerate}
\end{lemma}
\begin{proof}
Almost all the statements follow immediately from definitions~\eqref{Gdefs} and~\eqref{Gbardef}. The only properties that require an explicit proof are the lower bounds on~$|G|$ and~$\G$

To obtain the left-hand inequality in~\eqref{Gbounds} we assume without loss of generality that~$t \ge 0$ and compute, using~\eqref{cstardef},
$$
G(t) \ge \int_0^{\min\{ 1, t \}} g(t)\,dt \ge c_\star \min\{ 1, t \}.
$$

To get the lower bound in~\eqref{GGbounds}, we first notice that we can restrict ourselves to~$t \ge 1$, since the case~$t \in [0, 1]$ can be deduced straightaway from~\eqref{Gbounds} and the definition of $\G$. For~$t \ge 1$ we apply~\eqref{Gbounds} to compute
\bgs{
\G(t) = \int_0^1 G(\tau) \, d\tau + \int_1^t G(\tau) \, d\tau
\ge c_\star \left( \int_0^1 \tau \, d\tau + \int_1^t d\tau \right)
= \frac{c_\star}{2} \big( 1 + 2 (t - 1) \big) \ge \frac{c_\star}{2} \, t.
}

Finally, to establish the first inequality in~\eqref{GGbetterbound}, we recall definitions~\eqref{intglelambda}-\eqref{Gdefs} and compute, for~$t \ge 0$,
\begin{align*}
\G(t) - \frac{\Lambda}{2} t & = \int_0^t \left( \int_0^\tau g(\sigma) \, d\sigma \right) d\tau - \left( \int_0^{+\infty} g(\sigma) \, d\sigma \right) t = - \int_0^t \left( \int_{\tau}^{+\infty} g(\sigma) \, d\sigma \right) d\tau \\
& = - \int_0^t \left( \int_0^\sigma g(\sigma) \, d\tau \right) d\sigma - \int_t^{+\infty} \left( \int_0^t g(\sigma) \, d\tau \right) d\sigma \\
& = - \int_0^t \sigma g(\sigma) \, d\sigma - t \int_t^{+\infty} g(\sigma) \, d\sigma = - \lambda + \int_t^{+\infty} (\sigma - t) g(\sigma) \, d\sigma \ge - \lambda.
\end{align*}
Note that the third identity follows by Fubini's theorem. The proof of the lemma is thus complete.
\end{proof}

We stress that hypothesis~\eqref{intglelambda} has only been used to deduce the left-hand inequality in~\eqref{GGbetterbound}. If one drops it, the weaker lower bound
\begin{equation} \label{GGboundcons}
\G(t) \ge \frac{c_\star}{2} |t| - \frac{c_\star}{2} \quad \mbox{for every } t \in \R
\end{equation}
can still be easily deduced from~\eqref{GGbounds}. This estimate is indeed sufficient for most of the applications presented in the remainder of the paper. However, we will make crucial use of the finer bound~\eqref{GGbetterbound} at some point in the proof of Proposition~\ref{Linftylocprop}. Therefore, such result and all those that rely on it need assumption~\eqref{intglelambda} to hold.

Note that the function~$g(t) = 1/(1 + t^2)$ fulfills hypotheses~\eqref{geven},~\eqref{gbounds},~\eqref{gintegr}, but not~\eqref{intglelambda}. Also, the corresponding second antiderivative~$\G$ does not satisfy the lower bound in~\eqref{GGbetterbound} or any bound of the form~$\G(t) \ge \Lambda |t| / 2 - C$ for some constant~$C > 0$.

\subsection{Functional analytic properties of the fractional area functionals} \label{areafunc}

In this subsection we introduce the area-type functionals~$\F^M$ and determine some basic properties of the local part~$\A$ and nonlocal part~$\Nl^M$.

First of all, we observe that we can split the functional~$\F$ defined in~\eqref{Fcdef} into the two components
\bgs{
\F(u,\Omega)=\A(u,\Omega)+\Nl(u,\Omega),
}
with
\eqlab{\label{Adef}
\A(u,\Omega):=\int_\Omega\int_\Omega\G\left(\frac{u(x)-u(y)}{|x-y|}\right)\frac{dx\,dy}{|x-y|^{n-1+s}}
}
and
\bgs{
\Nl(u,\Omega):=2\int_\Omega\int_{\Co\Omega}\G\left(\frac{u(x)-u(y)}{|x-y|}\right)\frac{dx\,dy}{|x-y|^{n-1+s}}.
}
As will be shown in Lemma~\ref{Adomainlem}, in order for the local part~$\A(u,\Omega)$ to be finite, it is 
necessary and sufficient that~$u\in W^{s,1}(\Omega)$. On the other hand,
to have~$\Nl(u,\Omega)$ finite, one needs to impose some very restrictive condition
on the behavior of~$u$ in the whole~$\R^n$, such as~\eqref{Exterior_global_cond}.

For this reason, given any real number~$M \ge 0$ we define for a function~$u: \R^n \to \R$
the truncated nonlocal part
\eqlab{
\label{NMldef}
\Nl^M(u,\Omega):= \int_\Omega \left\{ \int_{\Co\Omega} \left( \int_{\frac{-M-u(y)}{|x-y|}}^{\frac{u(x)-u(y)}{|x-y|}}
\overline{G}(t) \, dt + \int^{\frac{M-u(y)}{|x-y|}}_{\frac{u(x)-u(y)}{|x-y|}}
\overline{G}(-t) \, dt \right) \frac{dy}{\kers} \right\} dx,
}
and the truncated nonlocal area-type functional
\begin{equation} \label{FMdef}
\F^M(u,\Omega) := \A(u,\Omega) + \Nl^M(u,\Omega).
\end{equation}
As we shall see shortly, this functional no longer requires extra conditions on~$u$ for its finiteness.

We will use the subscript~$\cdot_s$ to indicate the functionals corresponding to the choice~$g = g_s$---that is, we will write~$\A_s,\,\Nl_s^M$, and~$\F_s^M$. Note that, in this geometric case, the quantity~$\F_s^M(u, \Omega)$ corresponds to the~$s$-perimeter of the subgraph of~$u$ in the truncated cylinder~$\Omega^M = \Omega \times (-M, M)$. See Subsection~\ref{areageom} for more information on this.

We now proceed to analyze the functionals~$\A$,~$\Nl^M$,~$\F^M$ and investigate in particular their domains of definitions. In order to this, the following simple estimate turns out to be useful.

\begin{lemma}\label{dumb_kernel_lemma}
Let~$A, B\subseteq\Rn$ be bounded measurable sets. Then
$$
\int_A\int_B\dkers \le \frac{\Ha^{n-1}(\mathbb S^{n-1})}{1-s} \min \big\{ |A|, |B| \big\} \diam (A\cup B)^{1-s}.
$$
\end{lemma}
\begin{proof}
Suppose without loss of generality that~$|A| \le |B|$ and set~$D := \diam(A \cup B)$. Then, by changing variables conveniently we estimate
$$
\int_A\int_B\dkers \le  \int_A \Big( \int_{B_D} \frac{dz}{|z|^{n - 1 + s}} \Big) \, dx = \Ha^{n-1}(\mathbb S^{n-1}) |A| \int_0^D \frac{d\varrho}{\varrho^{s}},
$$
which directly leads to the conclusion.
\end{proof}

Thanks to this, we can easily find the natural domain of definition of the local part~$\A$.
Notice that for~$\A$ to be well-defined (albeit possibly infinite) one needs~$u$ to be defined only in~$\Omega$.

\begin{lemma} \label{Adomainlem}
Let~$\Omega \subseteq \R^n$ be a bounded open set and~$u: \Omega \rightarrow \R$
be a measurable function. Then,
\begin{equation} \label{AGags}
\frac{c_\star}{2} \Big( [u]_{W^{s,1}(\Omega)} - c_s(\Omega) \Big) \le \A(u,\Omega) \le \frac{\Lambda}{2} \, [u]_{W^{s,1}(\Omega)},
\end{equation}
where~$c_\star > 0$ is the constant defined in~\eqref{cstardef} and
\begin{equation} \label{csOmegadef}
c_s(\Omega) := \frac{\Ha^{n-1}(\mathbb S^{n-1})}{1-s}\,|\Omega| \diam (\Omega)^{1-s}.
\end{equation}
Therefore,
\bgs{
u\in W^{s,1}(\Omega) \quad \mbox{if and only if} \quad \A(u,\Omega) < \infty.
}
\end{lemma}
\begin{proof}
The upper bound in~\eqref{AGags} immediately follows by observing that~$\G(t) \le \Lambda |t| / 2$ for all~$t \in \R$, thanks to the right-hand inequality in formula~\eqref{GGbounds} of Lemma~\ref{gsprop}. To get the lower bound, we recall~\eqref{GGboundcons} and compute
$$
\A(u, \Omega) \ge \frac{c_\star}{2} \left( \int_\Omega \int_\Omega \frac{|u(x)-u(y)|}{|x-y|^{n + s}} \, dx\,dy - \int_\Omega \int_\Omega \dkers \right).
$$
The conclusion follows now from Lemma~\ref{dumb_kernel_lemma}. Finally, we observe that if~$u$ is a measurable function with~$[u]_{W^{s,1}(\Omega)}<\infty$,
then~$u\in L^1(\Omega)$---see, e.g.,~\cite[Lemma~D.1.2]{L_phd_th}.
\end{proof}

In the following result we present an equivalent representation for~$\Nl^M(u, \Omega)$, given in terms of the function~$\G$.
We also establish its finiteness when~$u$ belongs to the space~$\W^s(\Omega)$---recall~\eqref{Wsdef} for its definition.
Interestingly, no assumption on the behavior of~$u$ outside of~$\Omega$ is needed.

\begin{lemma} \label{NMdomainlem}
Let~$\Omega\subseteq\Rn$ be a bounded open set with Lipschitz boundary,~$M \ge 0$, and~$u: \R^n \to \R$ be a measurable function. Then,
\begin{equation} \label{Nlbound}
\left| \Nl^M(u,\Omega) \right| \le C\,\Lambda \big( \| u \|_{W^{s, 1}(\Omega)} + M \big),
\end{equation}
where~$\Lambda$ is the positive constant defined in~\eqref{gintegr}
and~$C > 0$ is a constant depending only on~$n$,~$s$, and~$\Omega$. Hence,
$$
\left| \Nl^M(u,\Omega) \right| < \infty \quad \mbox{if} \quad u \in \W^s(\Omega).
$$
Furthermore, we have the identity
\begin{equation} \label{nonlocal_explicit}
\begin{aligned}
\Nl^M(u,\Omega) & = \int_{\Omega} \left\{ \int_{\Co \Omega} \left\{ 2 \, \G \left( \frac{u(x)-u(y)}{|x-y|} \right) -
\G \left( \frac{M+u(y)}{|x-y|} \right) \right. \right. \\
& \quad \left. \left. - \, \G \left( \frac{M-u(y)}{|x-y|} \right) \right\} \frac{dy}{\kers} \right\} dx + M \Lambda \int_{\Omega} \int_{\Co \Omega} \frac{dx \, dy}{|x-y|^{n+s}}.
\end{aligned}
\end{equation}
\end{lemma}
\begin{proof}
We can assume that~$u|_\Omega\in W^{s,1}(\Omega)$, as otherwise~\eqref{Nlbound} is trivially satisfied.
Taking advantage of~\eqref{GbarG} and of the right-hand inequality in~\eqref{Gbounds}, we get that
\eqlab{\label{eqBlabla}
\left| \Nl^M(u,\Omega) \right| \le 2\Lambda \left\{ \int_\Omega \left( |u(x)| \int_{\Co\Omega} \frac{dy}{|x - y|^{n + s}} \right) dx + M \int_{\Omega} \int_{\Co \Omega} \frac{dx \, dy}{|x-y|^{n+s}} \right\}.
}
We remark that the last double integral in the previous formula is the $s$-fractional perimeter of~$\Omega$ in~$\R^n$,
which is finite, since~$\Omega$ is bounded and has Lipschitz boundary.
Then,~\eqref{Nlbound} follows from Corollary~\ref{FHI_corollary}.

On the other hand, identity~\eqref{nonlocal_explicit} is a simple consequence of definition~\eqref{NMldef}, formula~\eqref{GbarG}, and the symmetry properties of~$G$ and~$\G$.
\end{proof}

We stress that, in order to have~$\Nl^M(u, \Omega)$ finite, the requirement~$u|_\Omega \in W^{s, 1}(\Omega)$ is far from being optimal. In fact, as the previous proof showed, it suffices for~$u|_\Omega$ to lie in a suitable weighted~$L^1$ space over~$\Omega$, which contains for instance~$L^\infty(\Omega)$. Nevertheless, the requirement on the finiteness of~$[u]_{W^{s, 1}(\Omega)}$ does not limit our analysis, since it is needed to have~$\A(u, \Omega)$ finite, according to Lemma~\ref{Adomainlem}. We inform the interested reader that a more precise result on the natural domain of definition of~$\Nl_s^M$ i.e., for~$g = g_s$) will be provided by Lemma~\ref{per_of_subgraph_lem2} in the forthcoming Subsection~\ref{areageom}. 

Furthermore, we observe that if~$u:\R^n\to\R$ is such that~$u|_\Omega \in L^\infty(\Omega)$ and~$M\ge\|u\|_{L^\infty(\Omega)}$,
then~$\Nl^M(u,\Omega)\ge0$---this immediately follows from representation~\eqref{NMldef}.
On the other hand, in general the nonlocal part~$\Nl^M(\,\cdot\,,\Omega)$ can assume also negative values,
as showed in the following example.

\begin{example}\label{Exe_neg_nonloc}
Let~$\Omega\subseteq\Rn$ be a bounded open set with Lipschitz boundary and~$M \ge 0$.
There exists a positive
constant~$T_0=T_0(n,s,\Omega,g,M)>0$ such that, if~$u:\R^n\to\R$ is the constant function~$u\equiv T$,
for some~$T\ge T_0$, then
\bgs{
\F^M(u,\Omega)=\Nl^M(u,\Omega)<0.
}

To see this, let~$R>0$ be fixed in such a way that~$\Omega\Subset B_R$. By this, identity~\eqref{nonlocal_explicit},
and the fact that~$\G\ge0$, we have
$$
\Nl^M(u,\Omega) \le-\int_{\Omega} \int_{B_R\setminus \Omega} \G \left( \frac{M+T}{|x-y|} \right)\frac{dx \, dy}{\kers} + M \Lambda \int_{\Omega} \int_{\Co \Omega} \frac{dx \, dy}{|x-y|^{n+s}}.
$$
By exploiting~\eqref{GGboundcons}, the fact that~$\Omega$ has finite~$s$-perimeter (being bounded and Lipschitz), and Lemma~\ref{dumb_kernel_lemma}, we find that
$$
\int_{\Omega} \int_{B_R\setminus \Omega}
\G \left( \frac{M+T}{|x-y|} \right)\frac{dx \, dy}{\kers}
\ge \frac{c_\star}{2} \int_\Omega\int_{B_R\setminus \Omega}
\left( \frac{M+T}{|x-y|}- 1\right) \frac{dx \, dy}{\kers}\ge \frac{M+T}{C} - C,
$$
with~$C \ge 1$ depending only on~$n$,~$s$,~$\Omega$, and~$g$. Therefore,
\bgs{
\Nl^M(u,\Omega)\le- \frac{M+T}{C} + C + M \Lambda \int_{\Omega} \int_{\Co \Omega} \frac{dx \, dy}{|x-y|^{n+s}},
}
which is negative, provided~$T$ is large enough (in dependence of~$n$,~$s$,~$\Omega$,~$g$, and~$M$ only).
\end{example}

We collect the results of Lemmas~\ref{Adomainlem} and~\ref{NMdomainlem} in the following unifying statement.

\begin{lemma} \label{FMdomainlem}
Let~$\Omega\subseteq\Rn$ be a bounded open set with Lipschitz boundary,~$M \ge 0$, and~$u\in\W^s(\Omega)$. Then,~$\F^M(u, \Omega)$ is finite and it holds
$$
\left|\F^M(u, \Omega)\right| \le C\,\Lambda \left( \| u \|_{W^{s, 1}(\Omega)} + M \right),
$$
for some constant~$C > 0$ depending only on~$n$,~$s$, and~$\Omega$.
\end{lemma}

We conclude this subsection by specifying the convexity properties enjoyed by
the functionals~$\A$,~$\Nl^M$, and~$\F^M$. Recall~\eqref{domain_w_data} for the definition of the energy space~$\W^s_\varphi(\Omega)$ corresponding to a given outside datum~$\varphi$.

\begin{lemma}\label{conv_func}
Let~$\Omega\subseteq\Rn$ be a bounded open set with Lipschitz boundary. The following facts hold true:
\begin{enumerate}[label=$(\roman*)$,leftmargin=*]
\item The functional~$\A(\,\cdot\,,\Omega)$ is convex on~$W^{s, 1}(\Omega)$.
\item \label{convlem_ii} Given any~$M \ge 0$ and~$\varphi: \Co\Omega \to \R$,
the functionals~$\Nl^M(\,\cdot\,, \Omega)$ and~$\F^M(\,\cdot\,,\Omega)$
are strictly convex on~$\W^s_\varphi(\Omega)$.
\end{enumerate}
\end{lemma}
\begin{proof}
The convexity of the functionals is an immediate consequence of the (strict) convexity of~$\G$ warranted by Lemma~\ref{gsprop}. We point out that the convexity of~$\Nl^M(\,\cdot\,, \Omega)$ also depends on the fact that the second and third summands
appearing inside square brackets in the representation~\eqref{nonlocal_explicit} are constant on~$\W_\varphi^s(\Omega)$. Indeed, given~$u,\,v\in\W^s_\varphi(\Omega)$ and~$t\in(0,1)$, we have the identity
\eqlab{\label{conv_tarta}
& \Nl^M(tu+(1-t)v,\Omega)-t\,\Nl^M(u,\Omega)-(1-t)\Nl^M(v,\Omega)\\
& \hspace{15pt} =2\int_\Omega\bigg\{\int_{\Co\Omega}\bigg[\G\left(t\frac{u(x)-\varphi(y)}{|x-y|}+(1-t)\frac{v(x)-\varphi(y)}{|x-y|}\right)
-t\,\G\left(\frac{u(x)-\varphi(y)}{|x-y|}\right)\\
& \hspace{15pt} \quad -(1-t)\G\left(\frac{v(x)-\varphi(y)}{|x-y|}\right)\bigg]\frac{dy}{|x-y|^{n-1+s}}\bigg\}dx,
}
and the convexity of~$\G$ yields that the integrand in the double integral above is non-positive.
Furthermore, the strict convexity of~$\G$ gives that the quantity on the right-hand side of~\eqref{conv_tarta} is equal to zero if and only if
\bgs{
\frac{u(x)-\varphi(y)}{|x-y|}=\frac{v(x)-\varphi(y)}{|x-y|}\quad\mbox{for a.e.~}(x,y)\in\Omega\times\Co\Omega,
}
i.e., if and only if~$u=v$ almost everywhere in~$\Omega$---and hence in~$\R^n$.
\end{proof}

It is now worth pointing out the following result, which can be easily obtained by arguing as in the proof of
Lemma~\ref{NMdomainlem}, exploiting formula~\eqref{nonlocal_explicit} and
the global Lipschitzianity of~$\G$ given by~\eqref{Lip_Gcal}.

\begin{lemma} \label{tartariccio}
Let~$\Omega\subseteq\Rn$ be a bounded open set with Lipschitz boundary,~$M \ge 0$, and~$\varphi:\Co\Omega\to\R$.
Then, there exists a constant~$C>0$, depending only on~$n$,~$s$, and~$\Omega$, such that
\bgs{
\iint_{Q(\Omega)}\left|\G\left(\frac{u(x)-u(y)}{|x-y|}\right)
-\G\left(\frac{v(x)-v(y)}{|x-y|}\right)\right|\frac{dx\,dy}{|x-y|^{n-1+s}}
\le C\,\Lambda\|u-v\|_{W^{s,1}(\Omega)},
}
for every~$u,\,v\in\W^s_\varphi(\Omega)$, with~$\Lambda$ as defined in~\eqref{gintegr}.
Moreover, we have the identity
$$
\F^M(u,\Omega)-\F^M(v,\Omega)=
\iint_{Q(\Omega)} \! \left\{ \G\left(\frac{u(x)-u(y)}{|x-y|}\right)
-\G\left(\frac{v(x)-v(y)}{|x-y|}\right) \right\} \! \frac{dx\,dy}{|x-y|^{n-1+s}}.
$$
As a consequence, if~$u, u_k\in\W^s_\varphi(\Omega)$ are such that~$\|u-u_k\|_{W^{s,1}(\Omega)}\to0$ as~$k\to\infty$,
then
\bgs{
\lim_{k\to\infty}\F^M(u_k,\Omega)=\F^M(u,\Omega).
}
\end{lemma}

\begin{remark}\label{tartapower_Remark}
Here are some straightforward but important
consequences of Lemma~\ref{tartariccio}.
\begin{enumerate}[label=$(\roman*)$,leftmargin=*]
\item \label{tartapower_i} It guarantees that Definition~\ref{mindef} of minimizers of~$\F$ is well-posed.
\item \label{tartapower_ii} It provides an equivalent characterization for a minimizer of~$\F$ in~$\W^s_\varphi(\Omega)$
as a function~$u\in\W^s_\varphi(\Omega)$ that
minimizes~$\F^M(\,\cdot\,,\Omega)$ within~$\W^s_\varphi(\Omega)$ for some~$M \ge 0$, i.e., that satisfies
\bgs{
\F^M(u,\Omega)=\inf \Big\{ \F^M(v,\Omega) : v\in\W^s_\varphi(\Omega) \Big\}.
}
\item \label{tartapower_iii} By~\ref{tartapower_ii} and the strict convexity of~$\F^M$---see point~\ref{convlem_ii} of Lemma~\ref{conv_func}---, we obtain that a minimizer of~$\F$ in~$\W^s_\varphi(\Omega)$,
if it exists, is unique.
\item \label{tartapower_iv} As a consequence of the density of~$C^\infty_c(\Omega)$
in~$W^{s,1}(\Omega)$---see, e.g., Proposition~\ref{smooth_cpt_dense} in Appendix~\ref{app}---, Lemma~\ref{tartariccio} implies that to verify the minimality of~$u\in\W^s_\varphi(\Omega)$ we can limit ourselves to consider
competitors~$v\in\W^s_\varphi(\Omega)$ such that~$v|_\Omega\in C^\infty_c(\Omega)$.
\end{enumerate}
\end{remark}

\subsection{Geometric properties of the fractional area functionals} \label{areageom}

This subsection is devoted to the description of some key geometric properties enjoyed by~$\A_s$,~$\Nl_s^M$, and~$\F_s^M$. More specifically, we consider the case~$g=g_s$ and we show the connection existing between
the fractional perimeter~$\Per_s$ and these
functionals, which ultimately motivates their introduction.

Recall the splitting~\eqref{Perssplit}-\eqref{nonlocalPers} of the~$s$-perimeter into its local part~$\Per_s^L$ and nonlocal part~$\Per_s^{N \! L}$. We begin with a result relating~$\Per_s^L$ and~$\A_s$. Note that we indicate points in~$\R^{n + 1}$ with capital letters, writing~$X = (x, X_{n + 1})$, with~$x \in \R^n$ and~$X_{n + 1} \in \R$. Also recall the notation for cylinders introduced in~\eqref{cylnotation}.

\begin{lemma} \label{per_of_subgraph_lem1}
Let~$\Omega \subseteq \R^n$ be a bounded open set and~$u:\Omega\to\R$ be a measurable function. Then,
\begin{equation} \label{per_of_subgraph_lem1_claim1}
u \in W^{s, 1}(\Omega) \quad \mbox{if and only if} \quad \Per_s^L \! \left( \S_u, \Omega^\infty \right) < \infty.
\end{equation}
In particular, it holds
\begin{equation} \label{PerAid}
\Per_s^L \! \left( \S_u,\Omega^\infty \right)=\A_s(u,\Omega) + \Per_s^L \! \left( \{X_{n+1} < 0\},\Omega^\infty \right).
\end{equation}
\end{lemma}
\begin{proof}
Using Lebesgue's monotone convergence theorem, we write
$$
\Per_s^L \! \left( \S_u, \Omega^\infty \right) = \lim_{\delta \searrow 0} \iint_{D_\delta} dx \, dy \int_{-\infty}^{u(x)} dX_{n + 1} \int_{u(y)}^{+\infty} \frac{dY_{n + 1}}{|X - Y|^{n + 1 + s}},
$$
where~$D_\delta := \big\{ (x, y) \in \Omega \times \Omega : |u(x)| < \delta^{-1}, \, |u(y)| < \delta^{-1}, \mbox{ and } |x - y| > \delta \big\}$. Fix any small~$\delta > 0$ and let~$(x, y) \in D_\delta$. Shifting variables, we see that
$$
\int_{-\infty}^{u(x)} dX_{n + 1} \int_{u(y)}^{+\infty} \frac{dY_{n + 1}}{|X - Y|^{n + 1 + s}} = \int_{-\infty}^{u(x) - u(y)} dX_{n + 1} \int_{0}^{+\infty} \frac{dY_{n + 1}}{|X - Y|^{n + 1 + s}},
$$
so that
$$
\Per_s^L \! \left( \S_u, \Omega^\infty \right) = \lim_{\delta \searrow 0} \iint_{D_\delta} \! dx \, dy \int_{0}^{u(x) - u(y)} dX_{n + 1} \int_{0}^{+\infty} \frac{dY_{n + 1}}{|X - Y|^{n + 1 + s}} + \Per_s^L \! \left( \{ X_{n + 1} < 0 \}, \Omega^\infty \right).
$$
After a renormalization of both variables~$X_{n + 1}$ and~$Y_{n + 1}$, we have
$$
\int_{0}^{u(x) - u(y)} dX_{n + 1} \int_{0}^{+\infty} \frac{dY_{n + 1}}{|X - Y|^{n + 1 + s}} = \frac{1}{|x - y|^{n - 1 + s}} \int_{0}^{\frac{u(x) - u(y)}{|x - y|}} dt \int_{0}^{+\infty} \frac{dr}{\left[ 1 + (r - t)^2 \right]^{\frac{n + 1 + s}{2}}}.
$$
Changing coordinates once again and recalling definition~\eqref{gdef}, we obtain that
\begin{align*}
\int_{0}^{u(x) - u(y)} dX_{n + 1} \int_{0}^{+\infty} \frac{dY_{n + 1}}{|X - Y|^{n + 1 + s}} & = \frac{1}{|x - y|^{n - 1 + s}} \int_{0}^{\frac{u(x) - u(y)}{|x - y|}} \left( \int_{-t}^{+\infty} \frac{d\tau'}{\left[ 1 + (\tau')^2 \right]^{\frac{n + 1 + s}{2}}} \right) dt \\
& = \frac{1}{|x - y|^{n - 1 + s}} \int_{0}^{\frac{u(x) - u(y)}{|x - y|}} \left( \int_{-\infty}^{t} g_s(\tau) \, d\tau \right) dt.
\end{align*}
By~\eqref{Gdefs} and~\eqref{gintegr}, we get
\begin{align*}
\int_{0}^{\frac{u(x) - u(y)}{|x - y|}} \left( \int_{-\infty}^{t} g_s(\tau) \, d\tau \right) dt = \frac{\Lambda_s}{2} \frac{u(x) - u(y)}{|x - y|} + \G_s \! \left( \frac{u(x) - u(y)}{|x - y|} \right).
\end{align*}
Since, by symmetry,
$$
\iint_{D_\delta} \frac{u(x) - u(y)}{|x - y|} \frac{dx \, dy}{ |x - y|^{n - 1 + s}} = 0,
$$
we conclude that
$$
\Per_s^L \! \left( \S_u, \Omega^\infty \right) = \lim_{\delta \searrow 0} \iint_{D_\delta} \G_s \! \left( \frac{u(x) - u(y)}{|x - y|} \right) \frac{dx \, dy}{|x - y|^{n - 1 + s}} + \Per_s^L \! \left( \{ X_{n + 1} < 0 \}, \Omega^\infty \right).
$$
Claim~\eqref{PerAid} now follows by taking advantage once again of Lebesgue's monotone convergence theorem and recalling definition~\eqref{Adef}. Note that~\eqref{per_of_subgraph_lem1_claim1} immediately follows from~\eqref{PerAid} and Lemma~\ref{Adomainlem} (also recall~\cite[Lemma~D.1.2]{L_phd_th}).
\end{proof}

Next is a lemma that gives a geometric interpretation of the truncated nonlocal part~$\Nl_s^M$.

\begin{lemma} \label{per_of_subgraph_lem2}
Let~$\Omega \subseteq \R^n$ be a bounded open set with Lipschitz boundary and~$u: \R^n \to \R$ be
such that~$u|_\Omega \in L^\infty(\Omega)$. Then, for any~$M \ge \| u \|_{L^\infty(\Omega)}$, the quantity~$\Nl_s^M(u, \Omega)$ is finite and it holds
\begin{equation} \label{NMid}
\Nl_s^M(u,\Omega)=\Ll_s \! \left( \S_u\cap\Omega^M, \Co\S_u\setminus\Omega^\infty \right)
+\Ll_s \! \left( \Co\S_u\cap\Omega^M, \S_u\setminus\Omega^\infty \right).
\end{equation}
\end{lemma}
\begin{proof}
Thanks to the fact that~$M \ge \| u \|_{L^\infty(\Omega)}$, we write
\begin{align*}
\Ll_s \! \left( \S_u\cap\Omega^M, \Co\S_u\setminus\Omega^\infty \right) & = \int_{\Omega} dx \int_{\Co \Omega} dy \int_{-M}^{u(x)} dX_{n + 1} \int_{u(y)}^{+\infty} \frac{dY_{n + 1}}{|X - Y|^{n + 1 + s}}, \\
\Ll_s \! \left( \Co\S_u\cap\Omega^M, \S_u\setminus\Omega^\infty \right) & = \int_{\Omega} dx \int_{\Co \Omega} dy \int_{u(x)}^{M} dX_{n + 1} \int_{-\infty}^{u(y)} \frac{dY_{n + 1}}{|X - Y|^{n + 1 + s}}.
\end{align*}
By arguing as in the proof of Lemma~\ref{per_of_subgraph_lem1} and recalling definitions~\eqref{gdef}
and~\eqref{Gbardef}, we have
\begin{align*}
\int_{-M}^{u(x)} dX_{n + 1} \int_{u(y)}^{+\infty} \frac{dY_{n + 1}}{|X - Y|^{n + 1 + s}} & = \int_{-M - u(y)}^{u(x) - u(y)} dX_{n + 1} \int_{0}^{+\infty} \frac{dY_{n + 1}}{|X - Y|^{n + 1 + s}} \\
& = \frac{1}{|x - y|^{n - 1 + s}} \int_{\frac{-M - u(y)}{|x - y|}}^{\frac{u(x) - u(y)}{|x - y|}} \overline{G}_s(t) \, dt
\end{align*}
for every~$x \in \Omega$ and~$y \in \Co \Omega$. Hence,
$$
\Ll_s \! \left( \S_u\cap\Omega^M, \Co\S_u\setminus\Omega^\infty \right) = \int_{\Omega} dx \int_{\Co \Omega} dy \left( \frac{1}{|x - y|^{n - 1 + s}} \int_{\frac{-M - u(y)}{|x - y|}}^{\frac{u(x) - u(y)}{|x - y|}} \overline{G}_s(t) \, dt \right).
$$
Similarly,
$$
\Ll_s \! \left( \Co\S_u\cap\Omega^M, \S_u\setminus\Omega^\infty \right) = \int_{\Omega} dx \int_{\Co \Omega} dy \left( \frac{1}{|x - y|^{n + 1 + s}} \int_{\frac{u(x) - u(y)}{|x - y|}}^{\frac{M - u(y)}{|x - y|}} \overline{G}_s(-t) \, dt \right).
$$
By combining the last two identities and recalling definition~\eqref{NMldef}, we are led to~\eqref{NMid}.
\end{proof}

Putting the last two results together, we obtain a description of the~$s$-perimeter of the subgraph of~$u$ in~$\Omega^M$ in terms of the functional~$\F_s^M$.

\begin{prop} \label{per_of_subgraph_prop}
Let~$\Omega \subseteq \R^n$ be a bounded open set with Lipschitz boundary. Let~$u: \R^n \to \R$ be such that~$u|_\Omega \in L^\infty(\Omega)$ and take~$M \ge \| u \|_{L^\infty(\Omega)}$. Then,
$$
u|_\Omega \in W^{s,1}(\Omega) \quad \mbox{if and only if} \quad \Per_s \! \left( \S_u, \Omega^M \right) < \infty.
$$
In particular, it holds
\begin{equation} \label{per_of_subgraph}
\Per_s \! \left( \S_u, \Omega^M \right) = \F^M_s(u,\Omega) + \kappa_{\Omega, M},
\end{equation}
where~$\kappa_{\Omega, M}$ is the (positive) constant
$$
\kappa_{\Omega, M} := \Per_s^L \! \left( \{ X_{n + 1} < 0 \}, \Omega^\infty \right) - \Per_s^L \! \left( \{ X_{n + 1} < 0 \}, \Omega^\infty \setminus \Omega^M \right).
$$
\end{prop}
\begin{proof}
The proposition is an almost immediate consequence of Lemmas~\ref{per_of_subgraph_lem1} and~\ref{per_of_subgraph_lem2}. First, we observe that the following identities are true:
\begin{align*}
\Ll_s \! \left( \S_u \cap \Omega^M, \Omega^M \setminus \S_u \right) & = \int_{\Omega} dx \int_{\Omega} dy \int_{-M}^{u(x)} dX_{n + 1} \int_{u(y)}^M \frac{dY_{n + 1}}{|X - Y|^{n + 1 + s}}, \\
\Ll_s \! \left( \S_u \cap \Omega^M, \Co \S_u \setminus \Omega^M \right) & = \int_{\Omega} dx \int_{\Omega} dy \int_{-M}^{u(x)} dX_{n + 1} \int_{M}^{+\infty} \frac{dY_{n + 1}}{|X - Y|^{n + 1 + s}} \\
& \quad + \int_{\Omega} dx \int_{\Co \Omega} dy \int_{-M}^{u(x)} dX_{n + 1} \int_{u(y)}^{+\infty} \frac{dY_{n + 1}}{|X - Y|^{n + 1 + s}}, \\
\Ll_s \! \left( \Omega^M \setminus \S_u, \S_u \setminus \Omega^M \right) & = \int_{\Omega} dx \int_{\Omega} dy \int_{u(x)}^M dX_{n + 1} \int_{-\infty}^{-M} \frac{dY_{n + 1}}{|X - Y|^{n + 1 + s}} \\
& \quad + \int_{\Omega} dx \int_{\Co \Omega} dy \int_{u(x)}^M dX_{n + 1} \int_{-\infty}^{u(y)} \frac{dY_{n + 1}}{|X - Y|^{n + 1 + s}}.
\end{align*}
Note that we took advantage of the fact that~$M \ge \| u \|_{L^\infty(\Omega)}$ in order to obtain the above formulas. In light of this, it is not hard to see that
\begin{align*}
\Per_s \! \left( \S_u, \Omega^M \right) & = \Per_s^L \! \left( \S_u,\Omega^\infty \right) - \Per_s^L \! \left( \{ X_{n + 1} < 0 \}, \Omega^\infty \setminus \Omega^M \right) \\
& \quad + \Ll_s \! \left( \S_u \cap \Omega^M, \Co \S_u \setminus \Omega^\infty \right) + \Ll_s \! \left( \Co \S_u \cap \Omega^M, \S_u \setminus \Omega^\infty \right).
\end{align*}
Identity~\eqref{per_of_subgraph} follows by recalling definition~\eqref{FMdef} and applying~\eqref{PerAid} and~\eqref{NMid}.
\end{proof}

\subsection{Some facts about the Euler-Lagrange operator}\label{Prelim_EL_op}

We collect here some observations about the nonlocal integro-differential operator~$\h$, which is formally defined on a function~$u: \R^n \to \R$ at a point~$x \in \R^n$ by
$$
\h u(x) := 2 \, \PV \int_{\R^n} G \left( \frac{u(x)-u(y)}{|x-y|} \right) \frac{dy}{|x-y|^{n+s}}.
$$

We begin by introducing the following useful notation
\begin{equation} \label{deltagdef}
\delta_g(u,x;\xi):=G\left(\frac{u(x)-u(x+\xi)}{|\xi|}\right)-G\left(\frac{u(x-\xi)-u(x)}{|\xi|}\right),
\end{equation}
and we observe that, by symmetry, we can write
\eqlab{\label{delta_op}
\h u(x)=\PV\int_{\R^n}\frac{\delta_g(u,x;\xi)}{|\xi|^{n+s}}\,d\xi.
}
From now on, unless otherwise stated, we will understand~$\h u(x)$ as given by~\eqref{delta_op}. For~$r > 0$, we also 
define
\begin{equation} \label{h<hge}
\h^{\geq r}u(x):=\int_{\R^n\setminus B_r}\frac{\delta_g(u,x;\xi)}{|\xi|^{n+s}}\,d\xi \quad \mbox{and} \quad \h^{<r}u(x):=\int_{B_r}\frac{\delta_g(u,x;\xi)}{|\xi|^{n+s}}\,d\xi.
\end{equation}
By definition of principal value, it holds
\[\h u(x)=\lim_{r\searrow 0}\h^{\geq r}u(x).\]

\begin{remark}\label{rmk_tail}
Note that~$\h^{\geq r}u(x)$ is finite for every~$x\in\R^n$ and~$r>0$, under no assumptions on~$u$.
Indeed, by the right-hand bound in~\eqref{Gbounds} we have
\[\left|\frac{\delta_g(u,x;\xi)}{|\xi|^{n+s}}\right|\leq\frac{\Lambda}{|\xi|^{n+s}},\]
which is integrable in~$\R^n\setminus B_r$. In particular,~$|\h^{\geq r}u(x)|\leq \Lambda \Ha^{n - 1}(\mathbb{S}^{n - 1})/ (s r^s)$.
\end{remark}

One of the main advantages of writing the nonlocal operator~$\h u(x)$ as in~\eqref{delta_op}
is that the integral is well-defined in the Lebesgue sense, provided~$u$
is regular enough around~$x$.

\begin{lemma}\label{classical_form_reg_func}
Let~$u:\R^n\to\R$ be such that~$u\in C^{1,\gamma}(B_r(x))$, for some~$x\in\R^n$, $r>0$, and~$\gamma\in(s,1]$. Then,~$\h^{<\varrho}u(x)$ is finite for every~$\varrho>0$ and it holds
\eqlab{\label{formula_regular_func}
\h u(x)=\h^{<\varrho}u(x)+\h^{\geq\varrho}u(x)=\int_{\R^n}\frac{\delta_g(u,x;\xi)}{|\xi|^{n+s}}\,d\xi.
}
\end{lemma}
\begin{proof}
The lemma is an immediate consequence of the estimate
\eqlab{\label{local_C1gamma}
\left|\delta_g(u, x; \xi) \right|
\leq 2^\gamma\|u\|_{C^{1,\gamma}(B_r(x))}|\xi|^\gamma \quad \mbox{for all } \xi\in B_r\setminus\{0\}.
}
Indeed,~\eqref{local_C1gamma} easily implies that~$\h^{<\varrho} u(x)$ is well-defined (as~$\gamma > s$), whereas Remark~\ref{rmk_tail} gives the finiteness of~$\h^{\ge \varrho} u(x)$ and~\eqref{formula_regular_func} trivially follows.

To verify~\eqref{local_C1gamma}, we first notice that, by~\eqref{lip_G_bla} and~\eqref{deltagdef} we have
\begin{equation} \label{deltagest}
\left|\delta_g(u, x; \xi) \right| \le \left|\frac{u(x+\xi)+u(x-\xi)-2u(x)}{|\xi|}\right|.
\end{equation}
Then, by the mean value theorem we have
\[u(x+\xi)-u(x)=\nabla u(x+t\xi)\cdot\xi\quad\textrm{and}\quad u(x-\xi)-u(x)=\nabla u(x-\tau\xi)\cdot(-\xi),\]
for some~$t,\tau\in[0,1]$, and thus
\bgs{
\left|\frac{u(x+\xi)+u(x-\xi)-2u(x)}{|\xi|}\right| & =\left|\frac{\nabla u(x+t\xi)\cdot\xi-\nabla u(x-\tau\xi)\cdot\xi}{|\xi|}\right|\\
& \leq|\nabla u(x+t\xi)-\nabla u(x-\tau\xi)|
\leq2^\gamma\|u\|_{C^{1,\gamma}(B_r(x))}|\xi|^\gamma.
}
Estimate~\eqref{local_C1gamma} then follows at once.
\end{proof}

We stress that the right hand side of~\eqref{formula_regular_func} is defined in the standard Lebesgue sense, not as a principal value. Also notice that, thanks to Remark~\ref{rmk_tail}, we need not ask any growth condition for~$u$ at infinity.

When~$u$ is not regular enough around~$x$, the quantity~$\h u(x)$ is in general not well-defined. Nevertheless, as already observed in the Introduction,
we can understand the operator~$\h$ as defined in the following weak (distributional) sense. Given a function~$u:\R^n \to \R$, we set
\eqlab{\label{weak_opossum_curv}
\langle \h u, v\rangle:=\int_{\Rn}\int_{\Rn} G \left( \frac{u(x)-u(y)}{|x-y|} \right) \big( v(x)-v(y) \big) \, \frac{dx \, dy}{|x-y|^{n+s}}
}
for every $v\in C^\infty_c(\Rn)$. More generally, it is immediate to see that~\eqref{weak_opossum_curv} is well-defined for every~$v:\Rn \to \R$ such that~$[v]_{W^{s,1}(\Rn)}<\infty$. Indeed, taking advantage of the boundedness of~$G$, one has that
\begin{equation} \label{hsucont}
|\langle \h u ,v \rangle|\le\frac{\Lambda}{2} \, [v]_{W^{s,1}(\Rn)},
\end{equation}
with~$\Lambda$ as in~\eqref{gintegr}. Hence,~$\h u$ induces a continuous linear functional on~$W^{s, 1}(\R^n)$, that is
$$
\langle \h u, \cdot\, \rangle \in \left( W^{s, 1}(\R^n) \right)^*.
$$
Remarkably, this holds for every measurable function~$u: \R^n \to \R$, regardless of its regularity or integrability.

Estimate~\eqref{hsucont} says that the pairing~$(u, v) \mapsto \langle \h u, v \rangle$ is continuous in the second component~$v$, with respect to the~$W^{s, 1}(\R^n)$ topology. The next lemma shows that we also have continuity in~$u$ with respect to almost everywhere convergence.

\begin{lemma}\label{easylemma}
Let~$u_k,\,u:\R^n\to\R$ be such that~$u_k\to u$ a.e.~in~$\R^n$ and let~$v\in W^{s,1}(\R^n)$. Then,
\bgs{
\lim_{k\to\infty}\langle\h u_k,v\rangle=\langle\h u,v\rangle.
}
\end{lemma}

Lemma~\ref{easylemma} is a simple consequence of Lebesgue's dominated convergence theorem, thanks once again to the boundedness of~$G$.

The next result shows that the nonlocal mean curvature operator~$\h$ naturally arises when computing the Euler-Lagrange equation associated to the fractional area functional.

\begin{lemma} \label{1varlem}
Let~$\Omega \subseteq \R^n$ be a bounded open set with Lipschitz boundary,~$M \ge 0$, and~$u \in \W^s(\Omega)$. Then,
$$
\left. \frac{d}{d\eps} \right|_{\eps=0} \F^M(u+\eps v, \Omega) = \langle \h u, v \rangle \quad \mbox{for every } v \in \W^s_0(\Omega).
$$
\end{lemma}
\begin{proof}
First, notice that~$u + \varepsilon v \in \W^s(\Omega)$ for every~$\varepsilon \in \R$. Hence, by Lemma~\ref{FMdomainlem}, both~$\F^M(u, \Omega)$ and~$\F^M(u + \varepsilon v, \Omega)$ are finite. Now, by the mean value theorem, there is a function~$\tilde{\tau}_\varepsilon: \R \times \R \to [0, 1]$ such that~$\G \left( A + \varepsilon B \right) - \G \left( A \right) = \varepsilon \, G \left( A + \varepsilon \tilde{\tau}_\varepsilon(A, B) B \right) B$ for every~$A, B \in \R$. As~$v = 0$ in~$\Co \Omega$, calling
$$
\tau_\varepsilon(x, y) := \tilde{\tau}_\varepsilon \! \left( \frac{u(x)- u(y)}{|x - y|}, \frac{v(x)- v(y)}{|x - y|} \right) \quad \mbox{for every } x, y \in \R^n, 
$$
we have
\begin{align*}
& \F^M(u+\eps v,\Omega) - \F^M(u,\Omega) \\
& \hspace{20pt} = \varepsilon \int_{\R^n} \int_{\R^n} G \left( \frac{ u(x)- u(y) }{|x - y|} + \varepsilon \tau_\varepsilon(x, y) \frac{ v(x)- v(y) }{|x - y|} \right) \big( v(x)- v(y) \big) \, \frac{dx \, dy}{|x - y|^{n + s}}.
\end{align*}
Since~$G$ is bounded,~$v \in \W_0^s(\Omega)$, and~$|\tau_\varepsilon| \le 1$, we may conclude the proof using Lebesgue's dominated convergence theorem.
\end{proof}

Lemma~\ref{1varlem} says in particular that minimizers of~$\F$ solve~$\h u = 0$ in the weak sense made precise by Definition~\ref{weaksoldef}.
However, notice that Definition~\ref{weaksoldef} is a priori weaker than the conclusion of Lemma~\ref{1varlem}, as it requires the test functions~$v$ to lie in~$C^\infty_c(\Omega) \subsetneq \W^s_0(\Omega)$. Actually, if~$\Omega$ is bounded and has Lipschitz boundary, then the two notions are equivalent.

\begin{lemma}\label{crocodile_dens_rmk}
Let~$\Omega \subseteq \R^n$ be a bounded open set with Lipschitz boundary,~$f\in L^\infty(\Omega)$, and~$u:\R^n\to\R$ a measurable function. Then, the following are equivalent:
\begin{enumerate}[label=$(\roman*)$,leftmargin=*]
	\item \label{crocodilei} $u$ is a weak subsolution of~$\h u=f$ in~$\Omega$ in the sense of Definition~\ref{weaksoldef};
	\item \label{crocodileii} It holds
	\[
	\langle\h u,v\rangle\le\int_\Omega fv\,dx\quad\mbox{for every }v\in\W^s_0(\Omega)\mbox{ such that }v\ge0\mbox{ a.e.~in~}\R^n.
	\]
\end{enumerate}
\end{lemma}

\begin{proof}
The implication~\ref{crocodileii}~$\Rightarrow$~\ref{crocodilei} is obvious. As for the converse implication, let~$v\in\W^s_0(\Omega)$ such that~$v\ge0$ almost everywhere. Then, there exists a sequence~$\{ v_k \} \subseteq C^\infty_c(\Omega)$ such that~$v_k\ge0$ and~$v_k\to v$ in~$W^{s,1}(\Omega)$---see, e.g.,~Proposition~\ref{smooth_cpt_dense}. Also notice that, by Corollary~\ref{FHI_corollary},
\bgs{
\| v - v_k \|_{W^{s,1}(\R^n)}\le C \| v - v_k \|_{W^{s,1}(\Omega)}\quad\mbox{for every } k \in \N,
}
for some constant~$C > 0$ depending only on~$n$,~$s$, and~$\Omega$. Hence,~$v_k \rightarrow v$ in~$W^{s, 1}(\R^n)$, and exploiting the continuity ensured by~\eqref{hsucont} we have
\begin{equation} \label{takethelimit}
\langle\h u, v\rangle=\lim_{k\to\infty}\langle\h u, v_k\rangle\le\lim_{k\to\infty}\int_\Omega fv_k\,dx=\int_\Omega fv\,dx,
\end{equation}
concluding the proof.
\end{proof}

\begin{corollary}\label{crocodile_dens_rmk_cor}
	Let~$\Omega \subseteq \R^n$ be a bounded open set with Lipschitz boundary,~$f\in L^\infty(\Omega)$, and~$u:\R^n\to\R$ a measurable function. Then, the following are equivalent:
	\begin{enumerate}[label=$(\roman*)$,leftmargin=*]
		\item \label{alligatori} $u$ is a weak solution of~$\h u=f$ in~$\Omega$ in the sense of Definition~\ref{weaksoldef};
		\item \label{alligatorii} It holds
		\[
		\langle\h u,v\rangle=\int_\Omega fv\,dx\qquad\mbox{for every }v\in\W^s_0(\Omega).
		\]
	\end{enumerate}
\end{corollary}

\begin{proof}
	We only need to prove the implication~\ref{alligatori}~$\Rightarrow$~\ref{alligatorii}. For this, given~$v\in\W^s_0(\Omega)$, let us write~$v$ as the sum of its positive and negative parts, i.e.,~$v=v_+-v_-$. Since~$u$ is a weak subsolution and~$v_+,v_-\in\W^s_0(\Omega)$ are such that~$v_+,v_-\ge0$ almost everywhere, by Lemma~\ref{crocodile_dens_rmk} we have
	\[
	\langle\h u,v_+\rangle\le\int_\Omega fv_+\,dx\qquad\mbox{and}\qquad \langle\h u,v_-\rangle\le\int_\Omega fv_-\,dx.
	\]
	In the same way, since~$u$ is also a weak supersolution, it holds
	\[
	\langle\h u,v_+\rangle\ge\int_\Omega fv_+\,dx\qquad\mbox{and}\qquad \langle\h u,v_-\rangle\ge\int_\Omega fv_-\,dx.
	\]
	Thus, by linearity we obtain
	\[
	\langle\h u,v\rangle=\langle\h u,v_+\rangle-\langle\h u,v_-\rangle=\int_\Omega fv_+\,dx-\int_\Omega fv_-\,dx=\int_\Omega fv\,dx,
	\]
	which proves the claim.
\end{proof}

Taking advantage of the convexity of the functionals~$\F^M$, we can prove the equivalence between weak solutions having ``finite energy'' and minimizers. Actually, we can also prove that weak subsolutions are equivalent to subminimizers---in the sense defined here below.

\begin{defn}\label{submindef}
	Let~$\Omega \subseteq \R^n$ be a bounded open set
	with Lipschitz boundary. A measurable function~$u:\R^n\to\R$ is a \emph{subminimizer} of~$\F$ in~$\Omega$ if~$u\in\W^s(\Omega)$ and
	\begin{equation}\label{submin_eq}
	\iint_{Q(\Omega)} \left\{ \G \left( \frac{u(x) - u(y)}{|x - y|} \right) - \G \left( \frac{v(x) - v(y)}{|x - y|} \right) \right\} \frac{dx\,dy}{|x - y|^{n - 1 + s}} \le 0
	\end{equation}
	for every~$v\in\W^s(\Omega)$ such that~$v=u$ a.e.~in~$\Co\Omega$ and~$v\le u$ a.e.~in~$\Omega$. Symmetrically,~$u$ is a \emph{superminimizer} of~$\F$ in~$\Omega$ if~$u\in\W^s(\Omega)$ and \eqref{submin_eq} holds for every~$v\in\W^s(\Omega)$ such that~$v=u$ a.e.~in~$\Co\Omega$ and~$v\ge u$ a.e.~in~$\Omega$.
\end{defn}

\begin{remark}\label{brutto_rmk}
	In light of Lemma~\ref{tartariccio}, we can equivalently characterize a subminimizer of~$\F$ in~$\Omega$ as a function~$u\in\W^s(\Omega)$, such that
	\[
	\F^M(u,\Omega)=\inf\Big\{\F^M(v,\Omega):v\in\W^s_u(\Omega)\mbox{ s.t.~}v\le u\mbox{ a.e.~in~}\Omega\Big\},
	\]
	for some~$M\ge0$---and similarly for a superminimizer.
\end{remark}

\begin{lemma}\label{weak_subsol_implies_submin}
	Let~$\Omega \subseteq \R^n$ be a bounded open set with Lipschitz boundary and~$u\in\W^s(\Omega)$. Then,~$u$ is a subminimizer of~$\F$ in~$\Omega$ if and only if it is a weak subsolution of~$\h u=0$ in~$\Omega$.	
\end{lemma}

\begin{proof}
	Fix a non-negative number~$M\ge0$. Assume first that~$u$ is a subminimizer of~$\F$ in~$\Omega$, and let~$v\in C^\infty_c(\Omega)$ such that~$v\ge0$. Then, by Remark~\ref{brutto_rmk}, we have that
	\[
	\frac{1}{\eps}\left(\F^M(u-\eps v,\Omega)-\F^M(u,\Omega)\right)\ge0 \quad \mbox{for every } \varepsilon > 0.
	\]
	The conclusion then follows by passing to the limit~$\eps \searrow 0$ and recalling Lemma~\ref{1varlem}, as indeed
	\[
	0\le\lim_{\eps\searrow 0}\frac{1}{\eps}\left(\F^M(u-\eps v,\Omega)-\F^M(u,\Omega)\right)=-\langle\h u,v\rangle,
	\]
	for every~$v\in C^\infty_c(\Omega)$ such that~$v\ge0$.
	Conversely, suppose that~$u$ is a weak subsolution of~$\h u=0$ in~$\Omega$, and let~$v\in\W^s(\Omega)$ be such that~$v=u$ a.e.~in~$\Co\Omega$ and~$v\le u$ a.e.~in~$\Omega$. Thus,~$w:=u-v\in\W^s_0(\Omega)$ and~$w\ge0$ a.e.~in~$\R^n$. Lemma~\ref{crocodile_dens_rmk} then ensures that~$\langle\h u,w\rangle\le0$.
	Now, we observe that the convexity of~$\G$ implies that
	\bgs{
		\G(t)-\G(\tau)\ge G(\tau)(t-\tau)\quad\mbox{for every }t, \tau\in\R.
	}
	Thus, by Lemma~\ref{tartariccio} and definition~\eqref{weak_opossum_curv}, we obtain
	\bgs{
		\F^M(v,\Omega)-\F^M(u,\Omega)\ge\langle\h u, v-u\rangle=-\langle\h u,w\rangle\ge0,
	}
	which---by Remark~\ref{brutto_rmk}---implies that~$u$ is a subminimizer of~$\F$ in~$\Omega$.
\end{proof}

We then have the following comprehensive result connecting minimizers, sub-/superminimizers, and weak solutions.

\begin{corollary}\label{weak_implies_min_lemma}
Let~$\Omega \subseteq \R^n$ be a bounded open set with Lipschitz boundary and~$u \in \W^s(\Omega)$. Then, the following are equivalent:
\begin{enumerate}[label=$(\roman*)$,leftmargin=*]
	\item \label{comprehensivei} $u$ is a minimizer of~$\F$ in~$\Omega$,
	\item \label{comprehensiveii} $u$ is both a subminimizer and a superminimizer of~$\F$ in~$\Omega$,
	\item \label{comprehensiveiii} $u$ is a weak solution of~$\h u = 0$ in~$\Omega$.
\end{enumerate}
\end{corollary}

\begin{proof}
	The implication~\ref{comprehensivei}~$\Rightarrow$~\ref{comprehensiveii} simply follows from the definitions. The equivalence of~\ref{comprehensiveii} and~\ref{comprehensiveiii} is the content of Lemma~\ref{weak_subsol_implies_submin}. As for the implication~\ref{comprehensiveiii}~$\Rightarrow$~\ref{comprehensivei}, we can argue as in the proof of Lemma~\ref{weak_subsol_implies_submin}. Indeed, given~$v\in\W^s_u(\Omega)$, let~$w:=v-u\in\W^s_0(\Omega)$ and recall that by Corollary~\ref{crocodile_dens_rmk_cor} we know that~$\langle\h u,w\rangle=0$. Then, by convexity we obtain
	\bgs{
		\F^M(v,\Omega)-\F^M(u,\Omega)\ge\langle\h u,w\rangle=0.
	}
	Recalling Lemma~\ref{tartariccio}, this concludes the proof.
\end{proof}

It is worth observing that Lemma~\ref{easylemma} and Corollary~\ref{weak_implies_min_lemma} yield straightaway that
the set of minimizers of~$\F$ is closed in~$\W^s(\Omega)$, with respect to almost everywhere convergence.

\begin{prop}
Let~$\Omega \subseteq \R^n$ be a bounded open set with Lipschitz boundary and~$\{u_k\}\subseteq \W^s(\Omega)$ be such
that each~$u_k$ is a minimizer of~$\F$ in~$\Omega$.
If~$u_k\to u$ a.e.~in~$\R^n$, for some function~$u\in\W^s(\Omega)$, then~$u$ is a minimizer of~$\F$ in~$\Omega$.
\end{prop}

\subsection{Comparison principle}\label{Compa_sect}

This subsection is devoted to the proof of a comparison principle for sub- and superminimizers---and thus for weak sub- and supersolutions belonging to~$\W^s(\Omega)$, by Lemma~\ref{weak_subsol_implies_submin}. To obtain this result, we need a couple of preliminary lemmas. First, we have the following elementary result on convex functions.

\begin{lemma} \label{A+Blem}
Let~$\phi: \R \to \R$ be a convex function. Then, for every~$A, B, C, D \in \R$ satisfying~$\min \{ C, D\} \le A, B \le \max \{ C, D\}$ and~$A + B = C + D$, it holds
$$
\phi(A) + \phi(B) \le \phi(C) + \phi(D).
$$
\end{lemma}
\begin{proof}
Without loss of generality, we may suppose that~$A \le B$ and~$C \le D$. Since we have that~$C \le A \le B \le D$, there exist two values~$\lambda, \mu \in [0, 1]$ such that
$$
A = \lambda C + (1 - \lambda) D \quad \mbox{and} \quad B = \mu C + (1 - \mu) D.
$$
In view of the convexity of~$\phi$, it holds
\begin{equation} \label{A+Blemtech}
\begin{aligned}
\phi(A) + \phi(B) & = \phi(\lambda C + (1 - \lambda) D) + \phi(\mu C + (1 - \mu) D) \\
& \le \lambda \phi(C) + (1 - \lambda) \phi(D) + \mu \phi(C) + (1 - \mu) \phi(D) \\
& = \left( \lambda + \mu \right) \phi(C) + \left( 2 - \lambda - \mu \right) \phi(D).
\end{aligned}
\end{equation}
By taking advantage of the fact that~$A + B = C + D$, we now observe that
$$
\lambda C + (1 - \lambda) D + \mu C + (1 - \mu) D = C + D,
$$
or, equivalently,
$$
(1 - \lambda - \mu) (C - D) = 0.
$$
Consequently, either~$C = D$ or~$\lambda + \mu = 1$ (or both). In any case, we conclude that the right-hand side of~\eqref{A+Blemtech} is equal to~$\phi(C) + \phi(D)$, and from this the thesis follows.
\end{proof}

We use Lemma~\ref{A+Blem} to obtain the following inequality for rather general convex functionals. For our applications---both here and in Subsection~\ref{globboundsub}---, we will apply it with~$\Phi(U; x, y) := \G(U/|x - y|)$.

\begin{lemma} \label{Fminmaxlem}
Let~$\Phi: \R \times \R^n \times \R^n \to \R$ be a measurable function, convex with respect to the first variable, i.e.,~satisfying
\begin{equation} \label{Fconvex}
\Phi(\lambda U + (1 - \lambda) V; x, y) \le \lambda \Phi(U; x, y) + (1 - \lambda) \Phi(V; x, y)
\end{equation}
for every~$\lambda \in (0, 1)$,~$U, V \in \R$, and for~a.e.~$x, y \in \R^n$. Given a measurable set~$\U \subseteq \R^n \times \R^n$, consider the functional~$\mathfrak{F}$ defined by
$$
\mathfrak{F}(w) := \iint_{\U} \Phi(w(x) - w(y); x, y) \, dx \, dy
$$
for every~$w: \R^n \to \R$. Then, for every~$u, v: \R^n \to \R$, it holds
\begin{equation} \label{minmaxleuv}
\mathfrak{F}(\min\{u, v\}) + \mathfrak{F}(\max \{u, v\}) \le \mathfrak{F}(u) + \mathfrak{F}(v).
\end{equation}
\end{lemma}
\begin{proof}
We define the two functions~$m:=\min\{u,v\}$ and~$M:=\max\{u,v\}$. For fixed~$(x, y) \in \U$, we write
\begin{align*}
A & := m(x) - m(y), \quad B := M(x) - M(y), \quad C := u(x) - u(y), \quad D := v(x) - v(y),
\end{align*}
and
$$
\phi(t) = \phi_{x, y}(t) := \Phi(t; x, y) \quad \mbox{for every } t \in \R.
$$
Thanks to~\eqref{Fconvex}, the function~$\phi$ is convex. Also, we claim that
\begin{equation} \label{CleA,BleD}
\min \{ C, D \} \le A, B \le \max \{ C, D \}
\end{equation}
and
\begin{equation} \label{A+B=C+D}
A + B = C + D.
\end{equation}

Indeed, identity~\eqref{A+B=C+D} is immediate since~$m + M \equiv u + v$. The inequalities in~\eqref{CleA,BleD} are also obvious if~$u(x) \le v(x)$ and~$u(y) \le v(y)$ or if~$u(x) > v(x)$ and~$u(y) > v(y)$. On the other hand, when for example~$u(x) \le v(x)$ and~$u(y) > v(y)$, we have
$$
A = u(x) - v(y) \quad \mbox{and} \quad B = v(x) - u(y).
$$
Accordingly,
$$
C = u(x) - u(y) < u(x) - v(y) = A = u(x) - v(y) \le v(x) - v(y) = D
$$
and
$$
C = u(x) - u(y) \le v(x) - u(y) = B = v(x) - u(y) < v(x) - v(y) = D.
$$
Hence,~\eqref{CleA,BleD} is proved in this case. Arguing analogously, one can check that~\eqref{CleA,BleD} also holds when~$u(x) > v(x)$ and~$u(y) \le v(y)$.

Thanks to~\eqref{CleA,BleD} and~\eqref{A+B=C+D}, we may apply Lemma~\ref{A+Blem} and deduce that
$$
\phi(A) + \phi(B) \le \phi(C) + \phi(D).
$$
That is,
$$
\Phi(m(x) - m(y); x, y) + \Phi(M(x) - M(y); x, y) \le \Phi(u(x) - u(y); x, y) + \Phi(v(x) - v(y); x, y).
$$
Inequality~\eqref{minmaxleuv} then plainly follows by integrating the last formula in~$x$ and~$y$.
\end{proof}

We are now in position to prove our comparison principle for sub- and superminimizers. Note that it is a consequence of the strict convexity of the functional~$\F$.

\begin{prop}\label{Compari_prop}
	Let~$\Omega \subseteq \R^n$ be a bounded open set with Lipshitz boundary. Let~$\underline{u}, \overline{u}$ be respectively a sub- and a superminimizer of~$\F$ in~$\Omega$. If~$\underline{u}\le \overline{u}$ a.e.~in~$\Co\Omega$, then~$\underline{u}\le \overline{u}$ a.e.~in~$\R^n$.	
\end{prop}

\begin{proof}
	Consider the two functions~$\min\{\underline{u}, \overline{u}\}, \max\{\underline{u}, \overline{u}\}\in \W^s(\Omega)$ and let~$M\ge0$ be fixed. By Lemma~\ref{Fminmaxlem}, representation~\eqref{nonlocal_explicit} for~$\Nl^M$, and the fact that~$\underline{u} \le \overline{u}$ in~$\Co \Omega$, it is easy to see that
	\begin{equation}\label{parbleu}
	\F^M(\min \{\underline{u}, \overline{u}\},\Omega)+\F^M(\max \{\underline{u}, \overline{u}\},\Omega)\le \F^M(\underline{u},\Omega)+\F^M(\overline{u},\Omega).
	\end{equation}
	From the sub- and superminimality of~$\underline{u}$ and~$\overline{u}$ respectively, recalling Remark~\ref{brutto_rmk} we get that
	\[
	\F^M(\underline{u},\Omega)\le \F^M(\min \{\underline{u}, \overline{u}\},\Omega)\quad\mbox{and}\quad\F^M(\overline{u},\Omega)\le\F^M(\max \{\underline{u}, \overline{u}\},\Omega).
	\]
	By the second inequality and~\eqref{parbleu}, we conclude that~$\F^M(\underline{u},\Omega)= \F^M(\min \{\underline{u}, \overline{u}\},\Omega)$.

	Now, either~$\underline{u} = \min \{\underline{u}, \overline{u}\}$ a.e.~in~$\R^n$ or, by the strict convexity of~$\F^M$ (recall Lemma \ref{conv_func}),
	\[
	\F^M \! \left(\frac{\min \{\underline{u}, \overline{u}\}+\underline{u}}{2},\Omega\right)<\frac{1}{2}\,\F^M(\underline{u},\Omega)+\frac{1}{2}\, \F^M(\min \{\underline{u}, \overline{u}\},\Omega)=\F^M(\underline{u},\Omega).
	\]
	Since~$(\min \{\underline{u}, \overline{u}\}+\underline{u})/2\le\underline{u}$ a.e.~in~$\Omega$ and~$(\min \{\underline{u}, \overline{u}\}+\underline{u})/2=\underline{u}$ a.e.~in~$\Co\Omega$, the latter possibility is excluded by the subminimality of~$\underline{u}$. Thus,~$\underline{u}=\min \{\underline{u}, \overline{u}\}$ a.e.~in~$\R^n$, i.e.,~$\underline{u}\le \overline{u}$ a.e.~in~$\R^n$.
\end{proof}

\section{A priori bounds}\label{sect3}

\noindent
We collect in this section a few a priori estimates for weak solutions and minimizers.

\subsection{$W^{s, 1}$ estimates}\label{Ws1_estimates_sect}

We establish here a couple of bounds for the~$W^{s, 1}$ norm of solutions and minimizers. We begin with an estimate valid for weak subsolutions of~$\h u = f$ which are bounded from above. We will use it later in Section~\ref{ViscoWeak_Sect} to prove that viscosity solutions have finite energy. The result is stated for the positive part of the subsolution~$u$, but of course analogous estimates hold for all truncations~$(u - k)_+$ with~$k \in \R$.
 
\begin{prop} \label{localWs1prop}
Let~$\Omega \subseteq \R^n$ be a bounded open set with Lipschitz boundary,~$f: \Omega \to \R$ such that~$f_+ \in L^1(\Omega)$, and~$u$ be a weak subsolution of~$\h u = f$ in~$\Omega$. If~$u_+ \in \W^s(\Omega) \cap L^\infty(\Omega)$, then
\begin{equation} \label{localWs1est}
[u_+]_{W^{s, 1}(\Omega)} \le C \Big( 1 + \left( 1 + \| f_+ \|_{L^1(\Omega)} \right) \| u_+ \|_{L^\infty(\Omega)} \Big),
\end{equation}
for a constant~$C > 0$ depending only on~$n$,~$s$,~$g$, and on upper bounds on~$|\Omega|$,~$\diam(\Omega)$, and~$\Per_s(\Omega)$.
\end{prop}
\begin{proof}
First of all, as~$f \le f_+$, we have that
\begin{equation} \label{flef+}
\langle \h u,v\rangle\le\int_\Omega f_+ v\,dx \qquad\mbox{for every }v\in C_c^\infty(\Omega)\mbox{ such that }v\ge0.
\end{equation}
We now apply~\eqref{flef+} with~$v := \chi_{\Omega} u_+ \in \W^s_0(\Omega)$. This can be done even though~$v$ might not belong to~$C^\infty_c(\Omega)$, thanks to the same considerations made in the proof of Lemma~\ref{crocodile_dens_rmk}. Note that the convergence of the right-hand in~\eqref{takethelimit} side follows from Lebesgue's dominated convergence theorem, as the approximating sequence~$\{ v_k \}$ is now uniformly bounded by the~$L^\infty(\Omega)$ norm of~$u_+$---see Proposiiton~\ref{smooth_cpt_dense}.

Recalling definition~\eqref{weak_opossum_curv}, we get
\begin{align*}
& \int_{\Omega}\int_{\Omega} G \left( \frac{u(x)-u(y)}{|x-y|} \right) \big( u_+(x) - u_+(y) \big) \frac{dx \, dy}{|x-y|^{n+s}} \\
& \hspace{40pt} \le - 2 \int_{\Omega} \left\{ \int_{\Co \Omega} G \left( \frac{u(x)-u(y)}{|x-y|} \right) \frac{dy}{|x - y|^{n + s}} \right\} u_+(x) \, dx + \int_{\Omega} f_+(x) u_+(x) \, dx.
\end{align*}
From the monotonicity of~$G$, it is not hard to see that
$$
G \left( \frac{u(x)-u(y)}{|x-y|} \right) \big( u_+(x) - u_+(y) \big) \ge G \left( \frac{u_+(x)-u_+(y)}{|x-y|} \right) \big( u_+(x) - u_+(y) \big) \quad \mbox{for a.e.~} x, y \in \Omega.
$$
Combining this with the fact that~$G(t) t \ge c_\star(|t| - 1)$ for all~$t \in \R$---which follows from the properties of~$G$, in particular the left-hand inequality in~\eqref{Gbounds}---and Lemma~\ref{dumb_kernel_lemma}, we obtain
$$
\int_{\Omega}\int_{\Omega} G \left( \frac{u(x)-u(y)}{|x-y|} \right) \big( u_+(x) - u_+(y) \big) \frac{dx \, dy}{|x-y|^{n+s}} \ge c_\star \left( [u_+]_{W^{s, 1}(\Omega)} - c_s(\Omega) \right),
$$
with~$c_s(\Omega)$ as in~\eqref{csOmegadef}. On the other hand, using the right-hand bound in~\eqref{Gbounds}, we have that
$$
- 2 \int_{\Omega} \left\{ \int_{\Co \Omega} G \left( \frac{u(x)-u(y)}{|x-y|} \right) \frac{dy}{|x - y|^{n + s}} \right\} u_+(x) \, dx \le \Lambda \Per_s(\Omega) \, \| u_+ \|_{L^\infty(\Omega)}
$$
and
$$
\int_{\Omega} f_+(x) u_+(x) \, dx \le \| f_+ \|_{L^1(\Omega)} \| u_+ \|_{L^\infty(\Omega)}.
$$
Estimate~\eqref{localWs1est} then plainly follows.
\end{proof}

Next is a~$W^{s, 1}$ bound for minimizers of~$\F$. We will use it in Section~\ref{Linftysec} to gain the compactness needed to obtain the existence of solutions to the Dirichlet problem and prove Theorem~\ref{Dirichlet}. We stress that, although we will apply it to the family of bounded minimizers~$\{ u_M \}$, no boundedness assumption on the minimizer is required for its validity. Recall the definition~\eqref{taildef} of~$\Tail_s$.

\begin{prop} \label{Ws1prop}
There exist two positive constants~$\Theta$ and~$C$, depending only on~$n$,~$s$, and~$g$, such that the following holds true. If~$\Omega \subseteq \R^n$ is a bounded open set with Lipschitz boundary,~$\varphi:\Co\Omega\to\R$ is such that~$\Tail_s(\varphi, \Omega_{\Theta \diam(\Omega)} \setminus \Omega;\,\cdot\,) \in L^1(\Omega)$, and~$u\in\W^s_\varphi(\Omega)$ is such that
\bgs{
\F^M(u,\Omega)\le\F^M(v,\Omega)\quad\mbox{for every }v\in \W^s_\varphi(\Omega) \mbox{ s.t.~} |v| \le M \mbox{ a.e.~in } \Omega,
}
for some~$M \ge 0$, then
\bgs{
\diam(\Omega)^{- s} \| u \|_{L^1(\Omega)} + [u]_{W^{s, 1}(\Omega)}
\le C \left(  \left\| \Tail_s(\varphi, \Omega_{\Theta \diam(\Omega)} \setminus \Omega;\,\cdot\,) \right\|_{L^1(\Omega)}
+ \diam(\Omega)^{1 - s} |\Omega| \right).
}
\end{prop}

\begin{proof}
We use the function~$v := \chi_{\Co \Omega} u$ as a competitor for~$u$. We get
\begin{equation} \label{uMlev}
0 \le \F^M(v, \Omega) - \F^M(u, \Omega) = - \A(u, \Omega) + 2 \int_{\Omega} \int_{\Co \Omega} \frac{\I(x, y)}{|x - y|^{n - 1 + s}} \, dx\,dy,
\end{equation}
with
$$
\I(x, y) := \G \left( \frac{v(x) - v(y)}{|x - y|} \right) - \G \left( \frac{u(x) - u(y)}{|x - y|} \right).
$$

Write~$d := \diam(\Omega)$. On the one hand, by Lemma~\ref{Adomainlem},
\begin{equation} \label{AsuM}
\A(u, \Omega) \ge \frac{c_\star}{2} \int_\Omega \int_\Omega \frac{|u(x) - u(y)|}{|x - y|^{n + s}} \, dx \, dy
-\frac{c_\star\,\Ha^{n-1}(\mathbb S^{n-1})}{2(1-s)} |\Omega| d^{1 - s},
\end{equation}
with~$c_\star>0$ as defined in~\eqref{cstardef}. On the other hand, let~$R := \Theta d$, with~$\Theta \ge 1$ to be chosen later. Recalling the definition of~$v$ and taking advantage of point~\ref{gsprop_GG} of Lemma~\ref{gsprop}, we obtain
$$
\I(x, y) \le \frac{\Lambda}{2} \frac{|\varphi(y)|}{|x - y|} +\frac{c_\star}{2}- \frac{c_\star}{2} \frac{|u(x) - u(y)|}{|x - y|} \quad \mbox{for every } x \in \Omega, \, y \in \Omega_R \setminus \Omega
$$
and
$$
\I(x, y) \le \frac{\Lambda}{2} \frac{|u(x)|}{|x - y|} \quad \mbox{for every } x \in \Omega, \, y \in \Co \Omega_R.
$$
Hence, using Lemma~\ref{dumb_kernel_lemma}, that~$c_\star\le\Lambda$, and that~$\Co \Omega_R \subseteq \Co B_R(x)$ for every~$x \in \Omega$, we get
\begin{align*}
2\int_{\Omega} \int_{\Co \Omega} \frac{\I(x, y)}{|x - y|^{n - 1 + s}} \, dx\,dy &\le \int_{\Omega} \left( \int_{\Omega_R \setminus \Omega} \frac{\Lambda |\varphi(y)| - c_\star |u(x) - u(y)|}{|x - y|^{n + s}} \, dy \right) dx \\
&\quad+ c_\star \int_{\Omega} \int_{\Omega_R \setminus \Omega} \frac{dx\,dy}{|x - y|^{n - 1 + s}} + \Lambda \int_{\Omega} |u(x)| \left( \int_{\Co \Omega_R} \frac{dy}{|x - y|^{n + s}} \right) dx  \\
& \le \Lambda\, \bigg( \! \left\| \Tail_s(\varphi, \Omega_{\Theta d} \setminus \Omega; \,\cdot\,) \right\|_{L^1(\Omega)}
+\frac{\Ha^{n-1}(\mathbb S^{n-1})}{1-s} (1 + 2 \Theta)^{1 - s} d^{1 - s} |\Omega| \\
& \quad + \frac{\Ha^{n-1}(\mathbb S^{n-1})}{s \Theta^{s} d^{s}} \, \| u \|_{L^1(\Omega)} \bigg) - c_\star \int_{\Omega} \int_{\Omega_{\Theta d} \setminus \Omega} \frac{|u(x) - u(y)|}{|x - y|^{n + s}} \, dx\,dy.
\end{align*}
Putting together this estimate with~\eqref{uMlev} and~\eqref{AsuM}, and recalling that~$\Theta\ge1$, we find that
\begin{equation} \label{semiest}
\int_{\Omega} \int_{\Omega_{\Theta d}} \frac{|u(x) - u(y)|}{|x - y|^{n + s}} \, dx\,dy \le C_1 \! \left( \left\| \Tail_s(\varphi, \Omega_{\Theta d} \setminus \Omega; \,\cdot\,) \right\|_{L^1(\Omega)} + \Theta^{1 - s} d^{1 - s} |\Omega| + \frac{\| u \|_{L^1(\Omega)}}{\Theta^s d^s} \right)
\end{equation}
for some constant~$C_1 > 0$ depending only on~$n$,~$s$, and~$g$.

Observe now that~$\diam(\Omega_d)=3d$ and~$|\Omega_d\setminus\Omega|\ge c_nd^n$, for a dimensional constant~$c_n>0$. Indeed, the equality is an immediate consequence of the definition
of~$\Omega_d$, while the measure estimate follows from the fact that~$B_{d/2}(x_0)\subseteq\Omega_d\setminus\Omega$ for every~$x_0\in\partial\Omega_{d/2}$. Using these two facts, we estimate
$$
\int_{\Omega_d \setminus \Omega} \frac{dy}{|x - y|^{n + s}} \ge \frac{|\Omega_d \setminus \Omega|}{\diam(\Omega_d)^{n + s}} \ge \frac{c_n}{3^{n + s} d^s} \quad \mbox{for every } x \in \Omega.
$$
By this and the triangular inequality, we have
\begin{equation} \label{newFPI}
\begin{aligned}
\| u \|_{L^1(\Omega)} & \le \frac{3^{n + s} d^s}{c_n} \int_{\Omega} |u(x)| \left( \int_{\Omega_d \setminus \Omega} \frac{dy}{|x - y|^{n + s}} \right) dx \\
& \le \frac{3^{n + s} d^s}{c_n}  \left( \int_{\Omega} \int_{\Omega_{d}\setminus \Omega} \frac{|u(x) - u(y)|}{|x - y|^{n + s}} \, dx\,dy + \left\| \Tail_s(\varphi, \Omega_d \setminus \Omega;\,\cdot\,) \right\|_{L^1(\Omega)} \right).
\end{aligned}
\end{equation}
Using this estimate together with~\eqref{semiest} and recalling that~$\Theta \ge 1$, we get
$$
\| u \|_{L^1(\Omega)} \le C_2 \left( d^s \left\| \Tail_s(\varphi, \Omega_{\Theta d} \setminus \Omega; \,\cdot\,) \right\|_{L^1(\Omega)} + \Theta^{1 - s} d |\Omega| + \Theta^{-s} \| u \|_{L^1(\Omega)} \right),
$$
with~$C_2>0$ depending only on~$n$,~$s$, and~$g$.
By taking~$\Theta$ sufficiently large (in dependence of~$n$,~$s$, and~$g$ only), we can reabsorb the~$L^1$ norm of~$u$ on the left-hand side and obtain that
$$
\| u \|_{L^1(\Omega)} \le C_2 \left( d^s \left\| \Tail_s(\varphi, \Omega_{\Theta d} \setminus \Omega;\,\cdot\,) \right\|_{L^1(\Omega)} + d |\Omega| \right),
$$
for possibly a larger~$C_2$. The conclusion follows by combining this estimate with~\eqref{semiest}.
\end{proof}

\subsection{Interior $L^\infty$ estimates} \label{locboundsub}

In this subsection, we establish some interior boundedness results for minimizers of~$\F$ and~$\Per_s$. First, through a rather standard De Giorgi-type iteration, we prove the following proposition. It is stated for subminimizers---recall Definition~\ref{submindef}---and provides a bound from above. Of course, a two-sided~$L^\infty$ estimate can then be deduced for minimizers.

\begin{prop} \label{Linftylocprop}
Let~$R > 0$ and~$u$ be a subminimizer of~$\F$ in~$B_{2 R}$. Then,
\begin{equation} \label{supubound}
\sup_{B_R} u \le C \left( R + \dashint_{B_{2 R}} u_+(x) \, dx \right),
\end{equation}
for some constant~$C > 0$ depending only on~$n$,~$s$, and~$g$.
\end{prop}

\begin{proof}
Let~$0 < \varrho < \tau \le 2 R$ and~$\eta \in C_c^\infty(\R^n)$ be a
cutoff function acting between the balls~$B_\varrho$ and~$B_\tau$, i.e.,~satisfying~$0 \le \eta \le 1$ in~$\R^n$,~$\supp(\eta) \Subset B_\tau$,~$\eta = 1$
in~$B_\varrho$ and~$| \nabla \eta | \le 2 / (\tau - \varrho)$ in~$\R^n$. For~$k \ge 0$, we consider the
functions~$w = w_k := (u - k)_+$ and~$v := u - \eta w$. Clearly,~$v \le u$ in~$\R^n$ and~$v = u$ in~$\Co B_\tau$. Therefore,
\begin{equation} \label{Iineq}
\iint_{Q(B_\tau)} \frac{\I(x, y)}{|x - y|^{n - 1 + s}} \, dx\, dy \ge 0,
\end{equation}
with
$$
\I(x, y) := \G \left( \frac{v(x) - v(y)}{|x - y|} \right) - \G \left( \frac{u(x) - u(y)}{|x - y|} \right).
$$

We consider the sets~$A(k) := \{ x \in \R^n : u > k \}$ and~$A(k, t) := B_t \cap A(k)$, for~$t > 0$. First of all, we claim that
\begin{equation} \label{estinBr}
\I(x, y) \le - \frac{\Lambda}{2} \, \frac{|w(x) - w(y)|}{|x - y|} + \lambda \, \chi_{B_\varrho^2 \setminus (B_\varrho \setminus A(k, \varrho))^2}(x, y) \quad \mbox{for } x, y \in B_\varrho,
\end{equation}
with~$\lambda$ and~$\Lambda$ as defined in~\eqref{intglelambda} and~\eqref{gintegr} respectively.
Clearly,~\eqref{estinBr} holds for every~$x, y \in \Co A(k)$, since~$\I(x, y) = 0$ for these points and~$w = 0$ in~$\Co A(k)$.
Furthermore, it is also valid for~$x, y \in A(k, \varrho)$, as, indeed, by~\eqref{GGbetterbound} we have
$$
\I(x, y) = - \G \left( \frac{u(x) - u(y)}{|x - y|} \right) = - \G \left( \frac{w(x) - w(y)}{|x - y|} \right) \le - \frac{\Lambda}{2} \, \frac{|w(x) - w(y)|}{|x - y|} + \lambda.
$$
By symmetry, we are left to check~\eqref{estinBr} for~$x \in A(k, \varrho)$ and~$y \in B_\varrho \setminus A(k, \varrho)$.
In this case, using~$u(x) > k \ge u(y)$ along with~\eqref{GGbetterbound}, we get
\begin{align*}
\I(x, y) & = \G \left( \frac{k - u(y)}{|x - y|} \right) - \G \left( \frac{u(x) - u(y)}{|x - y|} \right) \le \frac{\Lambda}{2} \left( \frac{k - u(y)}{|x - y|} - \frac{u(x) - u(y)}{|x - y|} \right) + \lambda \\
& = - \frac{\Lambda}{2} \, \frac{u(x) - k}{|x - y|} + \lambda = -\frac{\Lambda}{2} \, \frac{|w(x) - w(y)|}{|x - y|} + \lambda.
\end{align*}
Hence,~\eqref{estinBr} is verified.

We now claim that
\begin{equation} \label{estoutBr}
\I(x, y) \le \Lambda \left( \chi_{B_\tau}(y) \frac{|w(x) - w(y)|}{|x - y|} + \frac{w(x)}{\max \{ \tau - \varrho, |x - y| \}} \right) \quad \mbox{for } x \in B_\tau, \, y \in \Co B_\varrho.
\end{equation}
We already observed that~$\I(x, y) = 0$ for every~$x, y \in \Co A(k)$. When~$x \in B_\tau \setminus A(k, \tau)$ and~$y \in A(k)$,
\begin{align*}
u(y) - u(x) & \ge u(y) - u(x) - \eta(y) (u(y) - k) = (1 - \eta(y)) u(y) + k \eta(y) - u(x) \\
& \ge (1 - \eta(y)) u(y) + k \eta(y) - k = (1 - \eta(y)) (u(y) - k) \ge 0
\end{align*}
and therefore
$$
\I(x, y) = \G \left( \frac{u(y) - u(x) - \eta(y) (u(y) - k)}{|x - y|} \right) - \G \left( \frac{u(y) - u(x)}{|x - y|} \right) \le 0,
$$
by the symmetry and monotonicity properties of~$\G$. We are thus left to deal with~$x \in A(k, \tau)$ and~$y \in \Co B_\varrho$. In this case, by the Lipschitz character of~$\G$ and the properties of~$\eta$,
\begin{align*}
\I(x, y) & \le \frac{\Lambda}{2} \, \frac{|\eta(x) w(x) - \eta(y) w(y)|}{|x - y|} \le \frac{\Lambda}{2} \, \frac{\eta(y) |w(x) - w(y)| + w(x) |\eta(x) - \eta(y)|}{|x - y|} \\
& \le \frac{\Lambda}{2} \left( \chi_{B_\tau}(y) \frac{|w(x) - w(y)|}{|x - y|} + \min \left\{ \frac{2}{\tau - \varrho}, \frac{1}{|x - y|} \right\} w(x) \right),
\end{align*}
and~\eqref{estoutBr} follows.

By taking advantage of estimates~\eqref{estinBr} and~\eqref{estoutBr} in~\eqref{Iineq}, by symmetry we deduce that
\begin{align*}
\iint_{B_\varrho^2} \frac{|w(x) - w(y)|}{|x - y|^{n + s}} \, dx\, dy & \le C \, \Bigg\{ \iint_{B_\tau^2 \setminus B_\varrho^2} \frac{|w(x) - w(y)|}{|x - y|^{n + s}} \, dx \,dy + \int_{A(k, \varrho)} \int_{B_\varrho} \frac{dx\, dy}{|x - y|^{n - 1 + s}} \\
& \quad + \int_{B_\tau} w(x) \left( \frac{1}{\tau - \varrho} \int_{B_{\tau - \varrho}} \frac{dz}{|z|^{n - 1 + s}} + \int_{\Co B_{\tau - \varrho}} \frac{dz}{|z|^{n + s}} \right) dx \Bigg\} \\
& \le  C \left\{\iint_{B_\tau^2 \setminus B_\varrho^2} \frac{|w(x) - w(y)|}{|x - y|^{n + s}} \, dx\, dy + |A(k, \varrho)| \varrho^{1 - s} + \frac{\| w \|_{L^1(B_\tau)}}{(\tau - \varrho)^s} \right\},
\end{align*}
for some constant~$C > 0$ depending only on~$n$,~$s$, and~$g$. Note that for the second inequality we also used Lemma~\ref{dumb_kernel_lemma}. Adding to both sides~$C$ times the left-hand side and dividing by~$1 + C$, we get that
$$
[w]_{W^{s, 1}(B_\varrho)} \le \theta \left( [w]_{W^{s, 1}(B_\tau)} + |A(k, \tau)| \tau^{1 - s} + \frac{\| w \|_{L^1(B_\tau)}}{(\tau - \varrho)^s} \right)
$$
for every~$0 < \varrho < \tau \le 2 R$ and for some constant~$\theta \in (0, 1)$ depending only on~$n$,~$s$, and~$g$. Applying, e.g,~\cite[Lemma~1.1]{GG82}, we infer that
$$
[w]_{W^{s, 1}(B_{(\varrho + \tau) / 2})} \le C \left( |A(k, \tau)| \tau^{1 - s} + \frac{\| w \|_{L^1(B_\tau)}}{(\tau - \varrho)^s} \right).
$$

Let~$\eta$ be a new cutoff acting between the balls~$B_\varrho$ and~$B_{(3 \varrho + \tau) / 4}$. Then, by the fractional Sobolev inequality (see,~e.g.,~\cite[Theorem~1]{MS02} or~\cite[Theorem~6.5]{DPV12}) and computations similar to other made previously, we have that
\begin{align*}
\| w \|_{L^{\frac{n}{n - s}}(B_\varrho)} & \le \| \eta w \|_{L^{\frac{n}{n - s}}(\R^n)} \le C \int_{\R^n} \int_{\R^n} \frac{|\eta(x) w(x) - \eta(y) w(y)|}{|x- y|^{n + s}} \, dx \, dy\\
& \le C \left( [w]_{W^{s, 1}(B_{(\varrho + \tau) / 2})} + \frac{\| w \|_{L^1(B_\tau)}}{(\tau - \varrho)^{s}} \right).
\end{align*}
Combining the last two inequalities and recalling that~$w = w_k$, we arrive at
\begin{equation} \label{wstarlew}
\| w_k \|_{L^{\frac{n}{n - s}}(B_\varrho)} \le C \left( |A(k, \tau)| \tau^{1 - s} + \frac{\| w_k \|_{L^1(B_\tau)}}{(\tau - \varrho)^s} \right)
\end{equation}
for every~$0 < \varrho < \tau \le 2 R$ and~$k \ge 0$.

Take now~$k > h \ge 0$. We have
$$
\| w_h \|_{L^1(B_\tau)} \ge \int_{A(k, \tau)} (u(x) - h) \, dx \ge (k - h) |A(k, \tau)|
$$
and
$$
\| w_h \|_{L^1(B_\tau)} \ge \int_{A(k, \tau)} (u(x) - h) \, dx \ge \int_{A(k, \tau)} (u(x) - k) \, dx = \| w_k \|_{L^1(B_\tau)}.
$$
Thanks to these relations,~\eqref{wstarlew}, and H\"older's inequality, it is easy to see that
\begin{equation} \label{preiter}
\varphi(k, \varrho) \le \frac{C}{(k - h)^{s/n}} \left( \frac{\tau^{1 - s}}{k - h} + \frac{1}{(\tau - \varrho)^s} \right) \varphi(h, \tau)^{1 + \frac{s}{n}},
\end{equation}
where we set~$\varphi(\ell, t) := \| w_\ell \|_{L^1(B_t)}$.

Consider the two sequences~$\{ k_j \}$ and~$\{ r_j \}$ defined by~$k_j := M (1 - 2^{-j})$ and~$r_j := R (1 + 2^{-j})$ for every non-negative integer~$j$ and with~$M > 0$ to be chosen later. By applying~\eqref{preiter} with~$k = k_{j + 1}$,~$\varrho = r_{j + 1}$,~$h = k_j$, and~$\tau = r_j$, setting~$\varphi_j := \varphi(k_j, r_j)$, and taking~$M \ge R$, we find
$$
\varphi_{j + 1} \le \frac{C (2^j \varphi_j)^{1 + \frac{s}{n}}}{(R \sqrt[n]{M})^s}.
$$
Applying now,~e.g.,~\cite[Lemma~7.1]{G03}, we conclude that~$\varphi_j$ converges to~$0$---i.e.,~$u \le M$ in~$B_R$---, provided we choose~$M$ in such a way that
$$
\| u_+ \|_{L^1(B_{2 R})} = \varphi_0 \le c_\sharp R^n M,
$$
for a sufficiently small constant~$c_\sharp > 0$ depending only on~$n$,~$s$, and~$g$. This concludes the proof.
\end{proof}

The counterpart of Proposition~\ref{Linftylocprop} for subgraphs that (locally) minimize~$\Per_s$ is contained in the next result, which easily follows from the density estimates of~\cite{CRS10}.

\begin{prop} \label{LinftyforPers}
Let~$\Omega\subseteq\R^n$ be an open set and~$E \subseteq \R^{n + 1}$ be a locally~$s$-minimal set in~$\Omega \times \R$ satisfying
\begin{equation*}
E \cap \left( \Omega \times \R \right) = \Big\{ (x,t) \in \Omega \times \R : t < u(x) \Big\},
\end{equation*}
for some measurable function~$u: \Omega \to \R$. Then,~$u\in L^\infty_\loc(\Omega)$.
\end{prop}

\begin{proof}
	The claim is equivalent to proving that, for every ball~$B := B_r(x_0)$ with~$4 B := B_{4 r}(x_0) \subseteq \Omega$, there exists a constant~$M>0$ such that, up to sets of measure zero,
	\[
	B \times(-\infty,-M)\subseteq E \cap \left( B \times \R \right)\subseteq B \times(-\infty,M).
	\]
	We only establish the rightmost inclusion, the other one being analogous.

	We argue by contradiction and suppose that
	\[
	\left|E\cap\left(B\times(k,+\infty)\right)\right|>0,
	\]
	for every~$k\in\N$. Notice that, for every~$k\in\N$, it holds
	\[
	\left|\left(B\times(k,+\infty)\right)\setminus E\right|>0,
	\]
	as otherwise we would have~$u\equiv+\infty$ in~$B$. These two facts imply that there exists a sequence of points~$X_k=(x_k,t_k)\in\partial E\cap\left(B\times\R\right)$ such that~$t_k>k$. As~$E$ is $s$-minimal in the $(n+1)$-dimensional ball~$\mathcal{B}_r(X_k) := \left\{ Y \in \R^{n + 1} : |Y - X_k| < r \right\} \Subset \Omega \times \R$, the clean ball condition~\cite[Corollary~$4.3$]{CRS10} ensures the existence of a point~$Z_k=(z_k,\tau_k)$ such that
	\[
	\mathcal{B}_{2 c r}(Z_k)\subseteq E\cap\mathcal{B}_r(X_k),
	\]
	for some constant~$c \in (0, 1/2)$ depending only on~$n$ and~$s$. This yields that~$u\geq \tau_k > k - r$ in~$B_{c r}(z_k)$.

	Now, up to passing to a subsequence, the points~$z_k$ converge to some~$z_0 \in 3 B := B_{3 r}(x_0)$. Hence,~$B_{c r}(z_0)\subseteq B_{2 c r}(z_k)$ for every large enough~$k$. As a result,~$u\geq k-d$ in~$B_{c r}(z_0)$ for every large~$k$. Since this is a contradiction, the proof of the proposition is complete.
\end{proof}

Proposition~\ref{LinftyforPers} establishes that the function~$u$ defining a locally~$s$-minimal subgraph is necessarily locally bounded. We point out that, by suitably modifying the proof we just presented, one can also obtain a quantitative~$L^\infty$ bound in terms of the~$L^1$ norm of~$u$---such as~\eqref{supubound}. For our aims, the statement of Proposition~\ref{LinftyforPers} is however more powerful, since it does not require a priori~$u$ to be integrable, but merely measurable.

\subsection{Global $L^\infty$ estimates} \label{globboundsub}

When the outside datum is bounded in a suitably large neighborhood of~$\Omega$, we can establish the boundedness of minimizers of~$\F$ up to the boundary of~$\Omega$---as claimed in Theorem~\ref{minareboundedthm}. We obtain this result by showing that, if~$u: \R^n \to \R$ is bounded in~$\Omega_R \setminus \Omega$ for some large~$R > 0$, the value of~$\F^M$ decreases when~$u$ is truncated at a high enough level---namely considering
\begin{equation} \label{uNdef}
u^{(N)} := \begin{cases}
\min \{ u, N \} & \quad \mbox{in } \Omega,\\
u & \quad \mbox{in } \Co \Omega,
\end{cases}
\end{equation}
for~$N \ge 0$. The precise formulation of this fact is as follows.

\begin{prop} \label{truncdecreaseprop}
There exists a positive constant~$\Theta$, depending only on~$n$,~$s$, and~$g$, such that the following holds true. If~$\Omega \subseteq \R^n$ is a bounded open set with Lipschitz boundary,~$M \ge 0$, and~$u: \R^n \to \R$ is a function bounded from above in~$\Omega_{\Theta \diam(\Omega)} \setminus \Omega$, then it holds
\begin{equation} \label{AandNdecrease}
\A(u^{(N)}, \Omega) \le \A(u, \Omega) \quad \mbox{and} \quad \Nl^M(u^{(N)}, \Omega) \le \Nl^M(u, \Omega)
\end{equation}
for every
\begin{equation} \label{NgeR0+sup}
N \ge \diam(\Omega) + \sup_{\Omega_{\Theta \diam(\Omega)} \setminus \Omega} u.
\end{equation}
In particular,
$$
\F^M(u^{(N)}, \Omega) \le \F^M(u, \Omega)
$$
for every~$N$ satisfying~\eqref{NgeR0+sup}.
\end{prop}

Observe that Proposition~\ref{truncdecreaseprop}
directly implies Theorem~\ref{minareboundedthm}, thanks to the uniqueness of the minimizer of~$\F$ in~$\Omega$ with respect to its own outside datum---see point~\ref{tartapower_iii} of Remark~\ref{tartapower_Remark}.

\begin{proof}[Proof of Proposition~\ref{truncdecreaseprop}]
Write~$v := u^{(N)}$,~$R_0 := \diam(\Omega)$, and~$R := \Theta R_0$, with~$\Theta \ge 1$ to be chosen later sufficiently large,
in dependence of~$n$,~$s$, and~$g$ only.
From Lemma~\ref{Fminmaxlem}, it clearly follows that~$\A(v, \Omega) \le \A(u, \Omega)$.
Hence, we can focus on the inequality for the nonlocal part~$\Nl^M$.

Thanks to representation~\eqref{nonlocal_explicit}, we have
$$
\Nl^M(v, \Omega) - \Nl^M(u, \Omega) = 2 \int_{\Omega} \int_{\Co \Omega} \left\{ \G \left( \frac{v(x) - v(y)}{|x - y|} \right) - \G \left( \frac{u(x) - u(y)}{|x - y|} \right) \right\} \frac{dx\,dy}{|x - y|^{n - 1 + s}}.
$$
Setting~$\Omega_+ := \{ x \in \Omega : u(x) > N \}$ and writing~$\Co \Omega = A_1 \cup A_2$, with~$A_1 := \Omega_R \setminus \Omega$ and~$A_2 := \C \Omega_R$, we infer from the above identity that the second inequality in~\eqref{AandNdecrease} is equivalent to
\begin{equation} \label{Nlorderequiv}
\alpha_1 + \alpha_2 \le 0,
\end{equation}
where we set
\begin{align*}
\alpha_i := \int_{\Omega_+} \left\{ \int_{A_i} \left\{ \G \left( \frac{N - u(y)}{|x - y|} \right) - \G \left( \frac{u(x) - u(y)}{|x - y|} \right) \right\} \frac{dy}{|x - y|^{n - 1 + s}} \right\} dx \quad \mbox{for } i = 1, 2.
\end{align*}
 
First, we establish a (negative) upper bound for~$\alpha_1$. Let~$x \in \Omega_+$ and~$y \in A_1$. Since, by hypothesis~\eqref{NgeR0+sup},~$u(y) \le N < u(x)$ and~$G$ is increasing, we have
\begin{equation} \label{Gincruyux}
\begin{aligned}
\G \left( \frac{N - u(y)}{|x - y|} \right) - \G \left( \frac{u(x) - u(y)}{|x - y|} \right) & = - \int_{\frac{N - u(y)}{|x - y|}}^{\frac{u(x) - u(y)}{|x - y|}} G(t) \, dt \\
& \le - G\left( \frac{N - u(y)}{|x - y|} \right) \frac{u(x) - N}{|x - y|},
\end{aligned}
\end{equation}
and consequently
$$
\alpha_1 \le - \int_{\Omega} \left( u(x) - N \right)_+ \left\{ \int_{A_1} G \left( \frac{N - u(y)}{|x - y|} \right) \frac{dy}{|x - y|^{n + s}} \right\} dx.
$$
Let~$D_1 \subseteq \Omega_{R_0} \setminus \Omega \subseteq A_1$ be a measurable set with~$|D_1| \ge c_\star R_0^n$ for some dimensional constant~$c_\star > 0$. To obtain such a set, we can select a point~$x_0 \in \partial \Omega$ for which~$\Omega$ lies on one side of a supporting hyperplane for~$\partial \Omega$ at~$x_0$---the existence of such a point follows from the boundedness of~$\Omega$. Calling~$H_+$ the half-space bounded by this hyperplane which does not contain~$\Omega$, we can simply set~$D_1 := H_+ \cap B_{R_0}(x_0)$. Using~\eqref{NgeR0+sup} and the monotonicity of~$G$, we estimate
$$
\int_{A_1} G \left( \frac{N - u(y)}{|x - y|} \right) \frac{dy}{|x - y|^{n + s}} \ge \int_{D_1} G \left( \frac{R_0}{R_0 + R_0} \right) \frac{dy}{(R_0 + R_0)^{n + s}} = \frac{G(1/2) |D_1|}{(2 R_0)^{n + s}} \ge \frac{c_1}{R_0^s}
$$
for every~$x \in \Omega$ and for some constant~$c_1 > 0$ depending only on~$n$,~$s$, and~$g$. Accordingly,
\begin{equation} \label{negcontr}
\alpha_1 \le - \frac{c_1}{R_0^s} \int_{\Omega} \left( u(x) - N \right)_+ dx.
\end{equation}

On the other hand, to control~$\alpha_2$ we simply use that~$\G$ is a globally Lipschitz function, by~\eqref{Lip_Gcal}, and compute
\begin{align*}
\alpha_2 & \le\frac{\Lambda}{2} \int_{\Omega} \left( u(x) - N \right)_+ \left( \int_{\C \Omega_R} \frac{dy}{|x - y|^{n + s}} \right) dx \\
& \le \frac{\Lambda}{2}\int_{\Omega} \left( u(x) - N \right)_+ \left( \int_{\C B_R} \frac{dz}{|z|^{n + s}} \right) dx \le \frac{C_2}{\Theta^s R_0^s} \int_{\Omega} \left( u(x) - N \right)_+ dx,
\end{align*}
for some constant~$C_2 > 0$ depending only on~$n$,~$s$, and~$g$. Notice that to get the second inequality we changed variables and took advantage of the inclusion~$B_R(x) \subseteq \Omega_R$, which holds for all~$x \in \Omega$. Combining this last estimate with~\eqref{negcontr}, we obtain
$$
\alpha_1 + \alpha_2 \le - \frac{1}{R_0^{s}} \left( c_1 - \frac{C_2}{\Theta^s} \right) \int_{\Omega} \left( u(x) - N \right)_+ dx,
$$
and~\eqref{Nlorderequiv} follows provided we take~$\Theta \ge (C_2 / c_1)^{1/s}$.
\end{proof}

Through an appropriate modification of the proof
of Proposition~\ref{truncdecreaseprop} and the interior~$L^\infty$ bound of Proposition~\ref{Linftylocprop}, we can actually improve Theorem~\ref{minareboundedthm}, establishing the global boundedness of minimizers of~$\F$ which have exterior data bounded only in an \emph{arbitrarily small} neighborhood of~$\Omega$.

We indicate with~$\bar{d}_\Omega$ the signed distance function from~$\partial\Omega$, negative inside~$\Omega$, and write
\begin{equation} \label{Omegarhodef}
\Omega_\varrho := \Big\{ x\in\R^n : \bar{d}_\Omega(x) <\varrho\Big \} \quad \mbox{for } \varrho \in \R.
\end{equation}

\begin{prop}\label{Bdary_Bdedness_prop}
Let~$\Omega\subseteq\R^n$ be a bounded open set with Lipschitz boundary.
If~$u$ is a minimizer of~$\F$ in~$\Omega$ and~$u\in L^\infty(\Omega_d\setminus\Omega)$
for some~$d > 0$,
then~$u\in L^\infty(\Omega)$ and it holds
$$
\|u\|_{L^\infty(\Omega\setminus\Omega_{-\theta d})}
\le d+\max\left\{\|u\|_{L^\infty(\Omega_{-\theta d})},\|u\|_{L^\infty(\Omega_d\setminus\Omega)}\right\},
$$
for some constant~$\theta\in(0,1)$ depending only on~$n$,~$s$,~$\Omega$, and~$g$.
\end{prop}

We stress that the~$L^\infty$ norm of~$u$ in~$\Omega_{- \theta d} \Subset \Omega$ is finite thanks to the interior~$L^\infty$ estimate of Proposition~\ref{Linftylocprop} and a simple covering argument. Also, note that if we further assume that~$\Tail_s(u, \Omega_{\Theta \diam(\Omega)} \setminus \Omega;\,\cdot\,) \in L^1(\Omega)$ for~$\Theta > 0$ sufficiently large,
then the combination of the~$L^1$ estimate of Proposition~\ref{Ws1prop},
the~$L^\infty$ one of Proposition~\ref{Linftylocprop}, and the bound of Proposition~\ref{Bdary_Bdedness_prop} leads to an estimate on~$\|u\|_{L^\infty(\Omega)}$ purely in terms of the
exterior data and of the geometry of~$\Omega$.

\begin{proof}[Proof of Proposition~\ref{Bdary_Bdedness_prop}]
Without loss of generality, we only prove the bound from above.

Pick any
\eqlab{\label{sonny_rollins}
N\ge d+\max\left\{\sup_{\Omega_{-\theta d}} u,\sup_{\Omega_d\setminus\Omega} u\right\},
}
with~$\theta \in (0, 1/4]$ to be chosen suitably small later, and let~$u^{(N)}$ be defined by~\eqref{uNdef}. By taking~$u^{(N)}$ as a competitor in Definition~\ref{mindef} and using Lemma~\ref{Fminmaxlem} as in the proof of Proposition~\ref{truncdecreaseprop}, we easily deduce that
\begin{equation} \label{beta1+beta2ge0}
\beta_1 + \beta_2 \ge 0,
\end{equation}
where, setting~$\Omega_+:=\{ x\in\Omega : u(x)>N\}$,
\begin{align*}
\beta_1 & := \int_{\Omega_+}\left\{\int_{\Co B_d(x)\setminus\Omega} \left\{ \G\left(\frac{N-u(y)}{|x-y|}\right)-\G\left(\frac{u(x)-u(y)}{|x-y|}\right) \right\}
\frac{dy}{|x-y|^{n-1+s}}\right\}dx, \\
\beta_2 & := \int_{\Omega_+}\left\{\int_{B_d(x)\setminus\Omega} \left\{ \G\left(\frac{N-u(y)}{|x-y|}\right)-\G\left(\frac{u(x)-u(y)}{|x-y|}\right) \right\}
\frac{dy}{|x-y|^{n-1+s}}\right\}dx.
\end{align*}

On the one hand, by the Lipschitz character of~$\G$---see~\eqref{Lip_Gcal}---we compute
\begin{equation} \label{beta1est}
\beta_1 \le \frac{\Lambda}{2} \int_{\Omega_+} \left( u(x) - N \right) \left( \int_{\Co B_d(x)} \frac{dy}{|x-y|^{n + s}}\right) dx \le \frac{C_1}{d^s} \int_{\Omega} \left( u(x) - N \right)_+ dx,
\end{equation}
for some constant~$C_1 > 0$ depending only on~$n$,~$s$, and~$g$.

We now address~$\beta_2$. First, notice that, in view of hypothesis~\eqref{sonny_rollins}, we have that~$u(y) \le N < u(x)$ for every~$x \in \Omega_+$ and~$y \in B_d(x) \setminus \Omega$. Hence, applying inequality~\eqref{Gincruyux} we get
$$
\beta_2 \le-\int_{\Omega_+}\left(u(x)-N\right) \left\{ \int_{B_d(x)\setminus\Omega}G\left(\frac{N-u(y)}{|x-y|}\right)
\frac{dy}{|x-y|^{n+s}}\right\}dx.
$$
As~$\Omega$ is bounded and Lipschitz, there exists a small constant~$\delta > 0$, depending only on~$n$ and~$\Omega$, such that~$|B_r(p) \setminus \Omega| \ge \delta r^n$ for every~$p \in \partial \Omega$ and~$r > 0$. Given~$x \in \Omega \setminus \Omega_{- \theta d}$, let~$p_x$ be any point on~$\partial \Omega$ such that~$|x - p_x| = \dist(x, \partial \Omega)$. Then,~$x \in \overline{B_{\theta d}(p_x)} \subseteq B_d(x)$ and~$|B_{\theta d}(p_x) \setminus \Omega| \ge \delta \theta^n d^n$. Since, by~\eqref{sonny_rollins},~$\Omega_+\subseteq\Omega\setminus\Omega_{-\theta d}$, this holds in particular for every~$x \in \Omega_+$. By virtue of this and the fact that, by the monotonicity of~$G$ and again~\eqref{sonny_rollins},~$G\left(\frac{N-u(y)}{|x-y|}\right) \ge G(1) > 0$ for every~$x \in \Omega_+$ and~$y \in B_d(x) \setminus \Omega$, we further estimate
$$
\beta_2 \le-G(1)\int_{\Omega_+}\left(u(x)-N\right) \left( \int_{B_{\theta d}(p_x)\setminus\Omega} \frac{dy}{|x-y|^{n+s}}\right) dx \le - \frac{c_2}{\theta^s d^s} \int_{\Omega}\left(u(x)-N\right)_+ dx,
$$
for some constant~$c_2 > 0$ depending only on~$n$,~$s$,~$\Omega$, and~$g$.

By combining this estimate,~\eqref{beta1est}, and~\eqref{beta1+beta2ge0}, we find that
$$
0 \le - \frac{1}{d^s} \left( \frac{c_2}{\theta^s} - C_1 \right) \int_{\Omega}\left(u(x)-N\right)_+ dx.
$$
If~$\theta < (c_2 / C_1)^{1/s}$, this is a contradiction, unless~$u \le N$ a.e.~in~$\Omega$. The proof is thus complete.
\end{proof}

\section{Viscosity solutions of fractional mean curvature-type equations}\label{ViscoWeak_Sect}

\noindent
In Subsection~\ref{Prelim_EL_op} we have seen how weak solutions of~$\h u = 0$ naturally arise when dealing with the fractional area-type functional~$\F^M$. Here, we focus instead on a different notion of solution, based on the viscosity approach developed by Caffarelli \& Silvestre~\cite{CS09,CS11} for a different class of integro-differential operators. In the geometric case~$g = g_s$, this notion is strongly related to the one considered in~\cite{CRS10}. The final aim of the section will be to show that viscosity (sub)solutions are also weak distributional (sub)solutions, thus proving Theorem~\ref{Gen_viscweak}.

\subsection{Viscosity (sub)solutions}

The starting point of our analysis is the following simple remark about an important monotonicity property enjoyed by the operator~$\h$.

\begin{remark}\label{mah}
Let~$x_0 \in \R^n$ and~$u,\,v:\R^n\to\R$ be such that
\[u(x_0)=v(x_0)\quad\textrm{and}\quad u(x)\leq v(x)\quad \mbox{for all } x\in\R^n.\]
Then, recalling definition~\eqref{deltagdef} and the monotonicity of the function~$G$, we infer that
\[\delta_g(u,x_0;\xi)\geq\delta_g(v,x_0;\xi) \quad \mbox{for all } \xi\in\R^n.\]
In particular, it follows that
\[\h u(x_0)\geq\h v(x_0),\]
provided the quantities~$\h u(x_0)$ and~$\h v(x_0)$ are well-defined.
\end{remark}

In light of the above remark, it is reasonable to consider the following definition.

\begin{defn}\label{visc_sol_def}
	Let~$\Omega\subseteq\R^n$ be an open set and~$f\in C(\overline{\Omega})$.
	We say that a function~$u:\R^n\to\R$ is a \emph{viscosity subsolution} of~$\h u=f$ in~$\Omega$, and we write
	\[
	\h u\leq f\quad\textrm{in }\Omega,
	\]
	if~$u$ is upper semicontinuous in~$\Omega$ and the following happens: if
	\begin{enumerate}[label=$(\roman*)$,leftmargin=*]
		\item \label{touched1} $x_0\in\Omega$,
		\item \label{touched2} $v\in C^{1,1}(B_r(x_0))$ for some~$r<\dist(x_0,\partial\Omega)$,
		\item \label{touched3} $v(x_0)=u(x_0)$ and~$v(y)\geq u(y)$ for every~$y\in B_r(x_0)$,
	\end{enumerate}
	then the function
	\begin{equation}\label{vtildedef}
	\tilde v(x):=
	\begin{cases}
	v(x) & \quad \textrm{if }x\in B_r(x_0),\\
	u(x) & \quad \textrm{if }x\in\R^n\setminus B_r(x_0),
	\end{cases}
	\end{equation}
	satisfies
	\[
	\h\tilde v(x_0)\leq f(x_0).
	\]
	A function~$u$ is a \emph{viscosity supersolution} of~$\h u=f$ in~$\Omega$ if~$-u$ is a viscosity supersolution of~$\h(-u)\le-f$ in~$\Omega$. A \emph{viscosity solution} of~$\h u = f$ in~$\Omega$ is a function~$u:\R^n\to\R$ which is continuous in~$\Omega$ and is both a viscosity subsolution and supersolution of~$\h u = f$ in~$\Omega$.
\end{defn}

From now on, we will mostly focus on subsolutions, as corresponding statements for supersolutions can be easily obtained by symmetry arguments. Also, unless otherwise specified,~$\Omega$ always denotes a bounded open subset of~$\R^n$ and~$f$ a continuous function in~$\overline{\Omega}$.

A first, useful observation on viscosity subsolutions is contained in the following result.

\begin{prop}\label{pwise_prop}
Let~$u$ be a viscosity subsolution of~$\h u = f$ in~$\Omega$. Assume that~$u$ is touched from above at a point~$x_0 \in \Omega$ by a~$C^{1,1}$ function~$v$, namely that points~\ref{touched2} and~\ref{touched3} of Definition~\ref{visc_sol_def} hold true. Then,~$\h u(x_0)$ is well-defined in the Lebesgue sense, is finite, and satisfies
\[\h u(x_0)\leq f(x_0).\]
\end{prop}

\begin{proof}
We begin by showing that~$\delta_g(u,x_0;\xi)|\xi|^{-n-s}$ is integrable in~$\R^n$, so that~$\h u(x_0)$ is well-defined as a Lebesgue integral. Our argument follows that of~\cite[Proposition 1]{Lin16}.

For every~$\varrho\in(0,r]$, we consider the functions
\[ v_\varrho(y):=
\begin{cases}
v(y) & \quad \textrm{if }y\in B_\varrho(x_0),\\
u(y) & \quad \textrm{if }y\in\R^n\setminus B_\varrho(x_0).
\end{cases}
\]
As~$ v\in C^{1,1}(B_r(x_0))$, the function $\xi \mapsto \delta_g( v_\varrho,x_0;\xi)|\xi|^{-n-s}$ is integrable in $\R^n$, that is
\[\int_{\R^n}\frac{\delta_g^+( v_\varrho,x_0;\xi)+\delta_g^-( v_\varrho,x_0;\xi)}{|\xi|^{n+s}}\,d\xi=
\int_{\R^n}\frac{|\delta_g( v_\varrho,x_0;\xi)|}{|\xi|^{n+s}}\,d\xi<+\infty,\]
where, for a general function~$\psi$, we write
\[\delta_g^+( \psi,x_0;\xi):=\max\{\delta_g( \psi,x_0;\xi),0\}\quad\textrm{and}\quad
\delta_g^-( \psi,x_0;\xi):=\max\{-\delta_g( \psi,x_0;\xi),0\}.\]
Moreover, by Remark~\ref{mah},
\begin{equation} \label{deltagvrhos}
\delta_g(u,x_0;\xi)\geq\delta_g( v_{\varrho_1},x_0;\xi)\geq\delta_g( v_{\varrho_2},x_0;\xi) \quad\textrm{for every }
0<\varrho_1\leq\varrho_2\leq r,
\end{equation}
and hence
\eqlab{\label{delta-}
\int_{\R^n}\frac{\delta_g^-(u,x_0;\xi)}{|\xi|^{n+s}}\,d\xi
\le\int_{\R^n}\frac{|\delta_g(v_r,x_0;\xi)|}{|\xi|^{n+s}}\,d\xi<+\infty.}
Also,~$\delta_g^+( v_\varrho,x_0;\xi)\nearrow\delta_g^+(u,x_0;\xi)$ as~$\varrho \searrow 0$, and therefore the monotone convergence theorem gives that
\[
\int_{\R^n}\frac{\delta_g^+(u,x_0;\xi)}{|\xi|^{n+s}}\,d\xi = \lim_{\varrho\to0^+}\int_{\R^n}\frac{\delta_g^+( v_\varrho,x_0;\xi)}{|\xi|^{n+s}}\,d\xi.
\]
Now, since~$u$ is a subsolution, we have
\[
\int_{\R^n}\frac{\delta_g^+( v_{\varrho_1},x_0;\xi)}{|\xi|^{n+s}}\,d\xi - \int_{\R^n}\frac{\delta_g^-( v_{\varrho_1},x_0;\xi)}{|\xi|^{n+s}}\,d\xi = \int_{\R^n}\frac{\delta_g( v_{\varrho_1},x_0;\xi)}{|\xi|^{n+s}}\,d\xi\leq f(x_0),
\]
and thus, recalling~\eqref{deltagvrhos},
$$
\int_{\R^n}\frac{\delta_g^+( v_{\varrho_1},x_0;\xi)}{|\xi|^{n+s}}\,d\xi
\leq\int_{\R^n}\frac{\delta_g^-( v_{\varrho_1},x_0;\xi)}{|\xi|^{n+s}}\,d\xi+f(x_0) \leq\int_{\R^n}\frac{\delta_g^-( v_{\varrho_2},x_0;\xi)}{|\xi|^{n+s}}\,d\xi+f(x_0),
$$
for every $0<\varrho_1\leq \varrho_2\leq r$. Letting~$\varrho_1 \rightarrow 0$ and taking~$\varrho_2 = \varrho$, we obtain that
\begin{equation}\label{pwise_def}
\int_{\R^n}\frac{\delta_g^+(u,x_0;\xi)}{|\xi|^{n+s}}\,d\xi\leq
\int_{\R^n}\frac{\delta_g^-( v_\varrho,x_0;\xi)}{|\xi|^{n+s}}\,d\xi+f(x_0)<+\infty,
\end{equation}
for every~$\varrho\in(0,r]$. Combining~\eqref{delta-} and~\eqref{pwise_def}, we conclude that~$\delta_g(u,x_0;\xi)|\xi|^{-n-s}$ is integrable in~$\R^n$ and hence~$\h u(x_0)$ is well-defined and finite.

To check that~$\h u(x_0) \le f(x_0)$, notice that, by~\eqref{deltagvrhos},
\[
\frac{\delta_g^-( v_\varrho,x_0;\xi)}{|\xi|^{n+s}}\leq\frac{\delta_g^-( v_r,x_0;\xi)}{|\xi|^{n+s}} \quad \mbox{for every } \varrho \in (0, r].
\]
As the function on the right-hand side is integrable in $\R^n$, by Lebesgue's dominated convergence theorem we can let~$\varrho\to0$ in~\eqref{pwise_def}, obtaining
\[\int_{\R^n}\frac{\delta_g^+(u,x_0;\xi)}{|\xi|^{n+s}}\,d\xi\leq
\int_{\R^n}\frac{\delta_g^-(u,x_0;\xi)}{|\xi|^{n+s}}\,d\xi+f(x_0),\]
which is the claim.
\end{proof}

For later use, it is convenient to introduce the following definition.
\begin{defn}\label{defC11}
A function~$u: \R^n \to \R$ is said to be~$C^{1,1}$ at a point~$x_0 \in \R^n$, and we write~$u\in C^{1,1}(x_0)$, if there exist $\ell\in\R^n$ and $M, r>0$ such that
\eqlab{\label{localC11}
|u(x_0+\xi)-u(x_0)-\ell\cdot\xi|\le M|\xi|^2 \quad \mbox{for all } \xi\in B_r.
}
\end{defn}

Clearly, if~$u\in C^{1,1}(B_R(x_0))$ for some~$R > 0$, then~$u\in C^{1,1}(x_0)$. Geometrically,~$u$ is~$C^{1,1}$ at~$x_0$ if there exist both an interior and an exterior tangent paraboloid to the subgraph of~$u$ at the point~$(x_0,u(x_0))$.

As a consequence of Proposition \ref{pwise_prop}, we obtain the following corollary.

\begin{corollary}\label{pwise_cor}
Let~$u$ be a viscosity subsolution of~$\h u = f$ in~$\Omega$ and assume that~$u$ is~$C^{1,1}$ at some point~$x_0 \in \Omega$. Then,~$\h u(x_0)$ is well-defined in the Lebesgue sense, is finite, and satisfies
\bgs{
\h u(x_0)\le f(x_0).
}
\end{corollary}

\begin{proof}
Consider the paraboloid
\bgs{
q(x):=u(x_0)+\ell\cdot(x-x_0)+M|x-x_0|^2 \quad \mbox{for all } x\in B_r(x_0),
}
with $\ell,\,M$ and $r$ as in Definition~\ref{defC11}. Obviously,~$q\in C^{1,1}(B_r(x_0))$. Also, by~\eqref{localC11} we know that~$q$ touches~$u$ from above at~$x_0$. The conclusion then follows from Proposition~\ref{pwise_prop}.
\end{proof}

\subsection{Sup-convolutions}
Here, we introduce and study the so-called sup-convolutions~$u^\eps$ of a viscosity subsolution~$u$---namely, a sequence of subsolutions of approximating equations which converge to~$u$ and enjoy nice regularity properties. For simplicity, we will consider only subsolutions which are bounded in the whole of~$\R^n$.

\begin{defn}\label{supconvolution}
Let~$u: \R^n \to \R$ be a bounded function and~$\eps > 0$. We define the \emph{sup-convolution}~$u^\eps$ of~$u$ as
$$
u^\eps(x):=\sup_{y\in\R^n}\left\{u(y)-\frac{1}{\eps}|y-x|^2\right\} \quad \mbox{for } x \in \R^n.
$$
\end{defn}

Now we point out some easy properties of sup-convolutions. From the definition, it immediately follows that
\begin{equation} \label{uepsgeu}
u^\eps\ge u\quad\mbox{in }\R^n.
\end{equation}
Moreover, setting~$M := \sup_{\R^n} |u| <+\infty$, we see that~$u^\varepsilon$ can be equivalently written as
\eqlab{\label{silence}
u^\eps(x)=\sup_{|y-x|\le\sqrt{2M\eps}}\left\{u(y)-\frac{1}{\eps}|y-x|^2\right\}.
}
Indeed, if $|y-x|>\sqrt{2M\eps}$, then~$u(y)-|y-x|^2/\eps<-M\le u(x)$, and~\eqref{silence} follows from~\eqref{uepsgeu}.

\begin{remark}\label{attained_conv}
Given an open set $\Omega\subseteq\R^n$, we define
\eqlab{\label{open_visc}
\Omega^\eps:=\left\{x\in\Omega : d(x,\partial\Omega)>2\sqrt{2M\eps}\right\}.
}
If~$u$ is upper semicontinuous in~$\Omega$, then for every~$x\in\Omega^\eps$ there exists~$y_0\in B_{\sqrt{2 M \varepsilon}}(x) \subseteq \Omega$ such that
\bgs{
u^\eps(x)=u(y_0)-\frac{1}{\eps}|y_0-x|^2=\max_{|y-x|\le\sqrt{2M\eps}}\left\{u(y)-\frac{1}{\eps}|y-x|^2\right\}.
}
This is a straightforward consequence of~\eqref{silence} and the upper semicontinuity of~$u$.
\end{remark}

In the next theorem we collect some important properties of sup-convolutions,
whose proofs can be found, for instance, in~\cite{Ambrosio}. First, we recall the definition of semiconvex functions.

\begin{defn}
Let $\Omega\subseteq\R^n$ be an open set and let $u:\Omega\to\R$. We say that $u$ is \emph{semiconvex} in $\Omega$ if there
exists a constant $c\ge0$ such that
\bgs{
x\mapsto u(x)+\frac{c}{2}|x|^2
}
is convex in every ball~$B\subseteq\Omega$. The smallest constant~$c\ge0$ for which this happens is called the \emph{semiconvexity constant} of~$u$ and is denoted by~$sc(u,\Omega)$.
\end{defn}

\begin{prop}\label{supconvo_teo}
Let~$u: \R^n \to \R$ be a bounded function and~$\varepsilon > 0$. Then,~$u^\eps$ is semiconvex in~$\R^n$ and~$sc(u^\eps,\R^n)\le \frac{2}{\varepsilon}$. Also,~$u^\eps\in W^{1,\infty}_\loc(\R^n)$,~$\nabla u^\eps\in \BV_\loc(\R^n,\R^n)$, and~$u^\eps\in C^{1,1}(x)$ for a.e.~$x\in\R^n$.

Furthermore, if~$u$ is upper semicontinuous in an open set~$\Omega\subseteq\R^n$, then, for every~$x \in \Omega$,
\bgs{
u^\eps(x)\searrow u(x) \quad \mbox{as} \quad \eps \searrow 0.
}
The convergence is locally uniform if~$u$ is continuous in~$\Omega$.
\end{prop}

\begin{proof}
The semiconvexity of~$u^\eps$ follows from~\cite[Proposition 4$(i)$]{Ambrosio}. By~\cite[Theorems~15 and~16]{Ambrosio}, this gives in turn that~$u^\eps\in W^{1,\infty}_{\loc}(\R^n)$ and~$\nabla u^\eps\in \BV_{\loc}(\R^n,\R^n)$.
That~$u^\eps\in C^{1,1}(x)$ for a.e.~$x\in\R^n$ follows from the Taylor expansion of~\cite[Theorem 16$(ii)$]{Ambrosio}. Finally, the convergence of~$u^\eps$ to~$u$ can be obtained by arguing as in the proof of~\cite[Proposition 4$(ii)$]{Ambrosio}.
\end{proof}

Sup-convolutions are particularly useful since they preserve the subsolution property, eventually up to a small error. We deal with this in the next result.

\begin{prop}\label{convo_visc}
Let~$\Omega\subseteq\R^n$ be a bounded open set,~$f\in C(\overline{\Omega})$, and~$u$ be a viscosity subsolution of~$\h u = f$ in~$\Omega$, bounded in~$\R^n$. Then,
\bgs{
\h u^\eps(x) \le f(x) + c_\eps \quad\mbox{for a.e.~} x \in \Omega^\eps,
}
where~$\Omega^\eps$ is defined by~\eqref{open_visc}, with~$M := \sup_{\R^n} |u|$, and
\begin{equation} \label{cepsdef}
c_\eps:=\sup_{\substack{x, y\in\overline\Omega \\ |x-y|\le\sqrt{2M\eps}}}|f(x)-f(y)|.
\end{equation}
Notice in particular that
\eqlab{\label{errorproperty}
c_\eps\searrow0 \mbox{ as }\eps\searrow0\quad\mbox{and}\quad c_\eps=0 \mbox{ if } f \mbox{ is constant.}
}
\end{prop}

\begin{proof}
Observe that, in view of Proposition~\ref{supconvo_teo},~$u^\eps$ is~$C^{1, 1}$ at a.e.~point of~$\Omega^\eps$. Thanks to this and Corollary~\ref{pwise_cor}, it then suffices to show that~$u^\eps$ is a viscosity subsolution of~$\h u^\eps = f + c_\eps$ in~$\Omega^\eps$.

Let~$x_0\in\Omega^\eps$ and suppose that there exist~$r \in (0, \dist(x_0, \partial \Omega^\eps))$ and~$v\in C^{1,1}(B_r(x_0))$ such that
\bgs{
v(x_0)=u^\eps(x_0)\quad\mbox{and}\quad v(x)\ge u^\eps(x) \quad \mbox{for all } x \in  B_r(x_0).
}
Defining~$\tilde{v}$ as in~\eqref{vtildedef}, we need to show that
\begin{equation} \label{Hvtildele}
\h\tilde v(x_0)\le f(x_0)+c_\eps,
\end{equation}
with~$c_\eps$ as in~\eqref{cepsdef}.

Thanks to Remark~\ref{attained_conv}, we can find~$y_0\in\Omega$ in such a way that~$|y_0-x_0|\le\sqrt{2M\eps}$ and
\bgs{
u^\eps(x_0)=u(y_0)-\frac{1}{\eps}|y_0-x_0|^2.
}
Then, we define
\bgs{
\psi(x):=v(x+x_0-y_0)+\frac{1}{\eps}|y_0-x_0|^2 \quad \mbox{for all } x\in B_r(y_0).
}
Clearly,~$\psi\in C^{1,1}(B_r(y_0))$. Moreover,
\bgs{
\psi(y_0)=v(x_0)+\frac{1}{\eps}|y_0-x_0|^2
=u^\eps(x_0)+\frac{1}{\eps}|y_0-x_0|^2=u(y_0).
}
As~$v\ge u^\eps$ in~$B_r(x_0)$, by definition of~$u^\eps$ we also have that
\bgs{
u(y)-\frac{1}{\eps}|y-x|^2\le u^\eps(x)\le v(x) \quad \mbox{for all } y\in\R^n \mbox{ and } x\in B_r(x_0).
}
Taking $y\in B_r(y_0)$ and $x:=y+x_0-y_0$, we get
\bgs{
u(y)\le \psi(y) \quad \mbox{for all } y\in B_r(y_0).
}
Thus,~$\psi$ touches~$u$ from above at~$y_0$ and hence
\begin{equation} \label{psitildeeq}
\h\tilde\psi(y_0)\le f(y_0),
\end{equation}
where, again,~$\tilde{\psi}$ is defined as in~\eqref{vtildedef} around the point~$y_0$, starting from~$\psi$.

Recalling definition~\eqref{h<hge} and changing variables appropriately, we compute
\begin{equation} \label{h<rpsitilde}
\begin{aligned}
\h^{<r}\tilde\psi(y_0)&
=2 \, \PV \int_{B_r(y_0)}G \left( \frac{\psi(y_0)-\psi(y)}{|y-y_0|} \right) \frac{dy}{|y-y_0|^{n+s}}\\
&
=2 \, \PV \int_{B_r(x_0)}G \left( \frac{v(x_0)-v(x)}{|x-x_0|} \right) \frac{dx}{|x-x_0|^{n+s}} =\h^{<r}\tilde v(x_0).
\end{aligned}
\end{equation}
On the other hand,
\bgs{
\h^{\ge r}\tilde\psi(y_0)&
= 2 \int_{\R^n\setminus B_r(y_0)} G\left( \frac{u^\eps(x_0)+\eps^{-1}|y_0-x_0|^2-u(y)} {|y-y_0|} \right) \frac{dy}{|y-y_0|^{n+s}}\\
&
= 2 \int_{\R^n\setminus B_r(x_0)}G \left( \frac{u^\eps(x_0)+ \eps^{-1} |y_0-x_0|^2-u(x+y_0-x_0)}{|x-x_0|} \right) \frac{dx}{|x-x_0|^{n+s}}.
}
Plugging~$y:=x+y_0-x_0$ in the definition of~$u^\eps(x)$ yields
\bgs{
\eps^{-1} |y_0-x_0|^2-u(x+y_0-x_0)\ge-u^\eps(x) \quad \mbox{for all } x\in\R^n.
}
Hence, by the monotonicity of~$G$,
\bgs{
\h^{\ge r}\tilde\psi(y_0)\ge 2 \int_{\R^n\setminus B_r(x_0)} G \left( \frac{u^\eps(x_0)-u^\eps(x)}{|x-x_0|} \right) \frac{dx}{|x-x_0|^{n+s}}=\h^{\ge r}\tilde v(x_0).
}
By combining this with~\eqref{h<rpsitilde} and~\eqref{psitildeeq}, we obtain that
\bgs{
\h\tilde v(x_0)\le\h\tilde\psi(y_0)\le f(y_0)\le f(x_0)+c_\eps,
}
which is~\eqref{Hvtildele}. This concludes the proof.
\end{proof}

\subsection{Relationship with weak (sub)solutions}

In this subsection, we explore the connection existing between viscosity and weak subsolutions---recall Definition~\ref{weaksoldef}. In particular, we will show the validity of Theorem~\ref{Gen_viscweak}.

To do this, we use a perturbative approach based on the sup-convolutions introduced in the previous subsection. Observe that, by Proposition~\ref{convo_visc}, we already know that sup-convolutions are pointwise a.e.~subsolutions of approximating equations. To improve this result to an inequality holding in the weak sense of Definition~\ref{weaksoldef}, it is convenient to consider the space of functions with bounded Hessian in~$\Omega$, defined as
\bgs{
\BH(\Omega) := & \,\, \Big\{u\in W^{1,1}(\Omega) : \nabla u\in \BV(\Omega,\R^n)\Big\}\\
= & \,\, \Big\{u\in W^{1,1}(\Omega) : \partial_j u\in \BV(\Omega) \mbox{ for every } j=1,\dots,n\Big\}
}
and endowed with the norm
\bgs{
\|u\|_{\BH(\Omega)}:=\|u\|_{W^{1,1}(\Omega)}+|D^2u|(\Omega),
}
where~$|D^2u|(\Omega)$ indicates the total variation of~$D^2 u$ in~$\Omega$, i.e., the~$\BV$ seminorm of~$\nabla u$ in~$\Omega$. 

For the properties of the space~$\BH(\Omega)$, we refer the interested reader to~\cite{Demengel}. We only recall the following useful density property---see~\cite[Proposition 1.4]{Demengel} for a proof.

\begin{prop}\label{bh_approx}Let~$\Op\subseteq\R^n$ be a bounded open set with $C^2$ boundary and let~$u\in \BH(\Op)$. Then, there exist a sequence of functions~$\{ u_k \} \subseteq C^2(\Op)\cap W^{2,1}(\Op)$ such that
\bgs{
\lim_{k\to\infty}\left\{\|u-u_k\|_{W^{1,1}(\Op)}+\big||D^2u|(\Op)-|D^2u_k|(\Op)\big|\right\}=0.
}
\end{prop}

Exploiting this density property, we can prove the following result.

\begin{lemma}\label{bv2_prop}
Let~$\Omega' \Subset \Omega\subseteq\R^n$ be bounded open sets and~$u\in \BH(\Omega)$. Then,
\eqlab{\label{bv2esti}
\int_{\Omega'}|u(x+\xi)+u(x-\xi)-2u(x)|\,dx\le 2|\xi|^2|D^2u|(\Omega) \quad \mbox{for every } \xi\in B_d.
}
where we set~$d:=\dist(\Omega',\partial\Omega)/2$.
\end{lemma}

\begin{proof}
Let~$\Op\subseteq\Omega$ be a bounded open set with~$C^2$ boundary such that
\eqlab{\label{op_appro_BH}
\Omega'\Subset\Op \quad \mbox{and} \quad d(\Omega',\partial\Op)>d.
}
By Proposition~\ref{bh_approx}, we can find a sequence~$\{ u_k \} \subseteq C^2(\Op)\cap W^{2,1}(\Op)$ such that
\eqlab{\label{appro_BH_conv}
\lim_{k\to\infty}\left\{\|u-u_k\|_{W^{1,1}(\Op)}+\big||D^2u|(\Op)-|D^2u_k|(\Op)\big|\right\}=0.
}
Now, let~$\xi \in B_d$ be fixed and notice that
\bgs{
|u_k(x+\xi)+u_k(x-\xi)-2u_k(x)| & \le |u_k(x+\xi)-u_k(x)-\nabla u_k(x)\cdot\xi|\\
& \quad +
|u_k(x-\xi)-u_k(x)-\nabla u_k(x)\cdot(-\xi)|.
}
By Taylor's theorem with integral remainder, we have
\bgs{
|u_k(x \pm \xi)-u_k(x)-\nabla u_k(x)\cdot (\pm \xi)|\le|\xi|^2\int_0^1|D^2u_k(x \pm t\xi)|\,dt,
}
so that, integrating as~$x$ ranges over~$\Omega'$ and using Fubini's theorem, we get
\bgs{
\int_{\Omega'}|u_k(x+\xi)+u_k(x-\xi)-2u_k(x)|\,dx & \le
|\xi|^2\int_{-1}^1 \left( \int_{\Omega'}|D^2u_k(x+t\xi)|\,dx \right) \! dt \\
& \le 2|\xi|^2|D^2u_k|(\Op),
}
since $|\xi|<d$ and $\Op$ satisfies \eqref{op_appro_BH}. Then, Fatou's Lemma and \eqref{appro_BH_conv} yield
\bgs{
\int_{\Omega'}|u(x+\xi)+u(x-\xi)-2u(x)|\,dx&\le 2|\xi|^2\lim_{k\to\infty}|D^2u_k|(\Op)=2|\xi|^2|D^2u|(\Op) \le2|\xi|^2|D^2u|(\Omega),
}
which is the desired bound~\eqref{bv2esti}.
\end{proof}

Thanks to the previous lemma, we can prove the following crucial result.

\begin{prop}\label{BHprop_curvature}
Let~$\Omega\subseteq\R^n$ be a bounded open set and~$u\in \BH(\Omega)$. Then,~$\h u\in L^1_{\loc}(\Omega)$ and
\begin{equation} \label{weakhuasinnerinL2}
\langle\h u,v\rangle=\int_{\Omega}\h u(x)v(x)\,dx \quad \mbox{for every } v\in C^\infty_c(\Omega).
\end{equation}
\end{prop}

\begin{proof}
Let~$\Omega'\Subset\Omega$ and~$d:= \dist(\Omega',\partial\Omega) / 2$. Taking advantage of Remark~\ref{rmk_tail},~\eqref{deltagest}, and~\eqref{bv2esti}, we estimate
\begin{align*}
\int_{\Omega'} |\h u(x)|\,dx & \le \int_{\Omega'} \left( \int_{\R^n}\frac{|\delta_g(u,x;\xi)|}{|\xi|^{n+s}}\,d\xi \right) \! dx\\
& \hspace{0pt} \le \frac{\Lambda \Ha^{n - 1}(\mathbb{S}^{n - 1})}{s} \frac{|\Omega'|}{d^s}
+\int_{B_d} \left( \int_{\Omega'}|u(x+\xi)+u(x-\xi)-2u(x)|\,dx \right) \! \frac{d\xi}{|\xi|^{n+1+s}}\\
& \hspace{0pt} \le \frac{\Lambda \Ha^{n - 1}(\mathbb{S}^{n - 1})}{s} \frac{|\Omega'|}{d^s} + 2 |D^2u|(\Omega) \frac{\Ha^{n - 1}(\mathbb{S}^{n - 1})}{1-s} \, d^{1-s}<+\infty.
\end{align*}
This proves that~$\h u\in L^1_{\loc}(\Omega)$.

We now head to the proof of~\eqref{weakhuasinnerinL2}. Notice that
\bgs{
|\h^{\ge\varrho}u(x)|\le \int_{\R^n}\frac{|\delta_g(u,x;\xi)|}{|\xi|^{n+s}}\,d\xi \quad \mbox{for every } \varrho>0
}
and that the right-hand side of the above formula is locally integrable in~$\Omega$ as a function of~$x$, thanks to the previous computation. Therefore, given~$v\in C^\infty_c(\Omega)$ we can apply Lebesgue's dominated convergence theorem to obtain that
\bgs{
\lim_{\varrho\searrow 0}\int_{\R^n} \h^{\ge\varrho}u(x)v(x)\,dx=\int_{\R^n} \h u(x)v(x)\,dx.
}
Now notice that, by symmetry,
\bgs{
\int_{\R^n} \h^{\ge\varrho}u(x)v(x)\,dx=\int_{\R^n}\int_{\R^n}G \left( \frac{u(x)-u(y)}{|x-y|} \right) \! \big(v(x)-v(y)\big)
\chi_{\C B_\varrho}(x-y) \, \frac{dx\,dy}{|x-y|^{n+s}}.
}
Hence, since~$v\in C^\infty_c(\Omega)\subseteq W^{s,1}(\R^n)$ and~$G$ is bounded, Lebesgue's dominated convergence theorem can be used once again to deduce that
\bgs{
\lim_{\varrho\searrow 0}\int_{\R^n}\int_{\R^n}G \left( \frac{u(x)-u(y)}{|x-y|} \right) \! \big(v(x)-v(y)\big)
\chi_{\C B_\varrho}(x-y) \, \frac{dx\,dy}{|x-y|^{n+s}} =\langle\h u,v\rangle.
}
The combination of the last three identities leads us to~\eqref{weakhuasinnerinL2}.
\end{proof}

By putting together Propositions~\ref{convo_visc},~\ref{supconvo_teo}, and~\ref{BHprop_curvature}, we immediately obtain the following result.

\begin{corollary}\label{convo_weak}
Let~$\Omega\subseteq\R^n$ be a bounded open set,~$f\in C(\overline\Omega)$, and~$u$ be a viscosity subsolution of~$\h u = f$ in~$\Omega$, bounded in~$\R^n$. Then,~$u^\varepsilon$ is a weak subsolution of~$\h u^\varepsilon = f + c_\eps$ in~$\Omega^\eps$, with~$c_\eps$ as in~\eqref{cepsdef}.
\end{corollary}

An easy consequence of the previous corollary is the next result, which already provides a proof of Theorem~\ref{Gen_viscweak} in the case of bounded, (semi)continuous outside data. Indeed, in the following statement we require the subsolution~$u$ to be upper semicontinuous outside of a set~$S \subseteq \C \Omega$ having vanishing Lebesgue measure. Note that, when~$\Omega$ is well-behaved, one may take~$S$ to contain~$\partial \Omega$, thus allowing~$u$ to be discontinuous across~$\partial \Omega$.

\begin{prop}\label{ViscWeak_reg}
Let~$\Omega\subseteq\R^n$ be a bounded open set with Lipschitz boundary,~$f\in C(\overline\Omega)$, and~$u$ be  a viscosity subsolution of~$\h u = f$ in~$\Omega$, bounded in~$\R^n$. Assume that there exists a closed set~$S\subseteq \R^n \setminus \Omega$
such that~$|S|=0$ and that~$u$ is upper semicontinuous in~$\R^n\setminus S$.

Then,~$u$ is a weak subsolution of~$\h u = f$ in~$\Omega$. In addition,~$u_+ \in W^{s, 1}(\Omega)$ and it holds
\begin{equation} \label{u+inWs1}
[u_+]_{W^{s, 1}(\Omega)} \le C,
\end{equation}
for some constant~$C > 0$ depending only on~$n$,~$s$,~$g$,~$\Omega$,~$\| u_+ \|_{L^\infty(\Omega)}$, and~$\| f_+ \|_{L^1(\Omega)}$.
\end{prop}

\begin{proof}
The hypotheses on~$u$ and Proposition~\ref{supconvo_teo} give that, as~$\eps \to 0$, the sup-convolutions~$\{ u^\eps \}$ converge to~$u$ pointwise outside of~$S$ and hence a.e.~in $\R^n$. Let~$v\in C^\infty_c(\Omega)$ be a non-negative function. Note that~$\supp(v)\subseteq\Omega^\eps$ provided~$\eps$ is small enough. Thus, using Lemma~\ref{easylemma}, Corollary~\ref{convo_weak}, and property~\eqref{errorproperty},
we obtain
\bgs{
\langle\h u,v\rangle=\lim_{\eps \searrow 0}\langle\h u^\eps,v\rangle\le
\lim_{\eps\searrow 0}\int_\Omega(f+c_\eps)v\,dx=\int_\Omega fv\,dx.
}
Hence,~$u$ is a weak solution of~$\h u = f$ in~$\Omega$.

To check the validity of~\eqref{u+inWs1}, let~$\Omega' \Subset \Omega$ be any open set with Lipschitz boundary such that~$\Per_s(\Omega') \le \Per_s(\Omega) + 1$. For all~$\varepsilon$ sufficiently small,~$u^\varepsilon$ is a weak subsolution of~$\h u^\varepsilon = f + c_\varepsilon$ in~$\Omega'$ lying in~$\W^s_\loc(\R^n) \cap L^\infty(\R^n)$. Thus, we may apply to it Proposition~\ref{localWs1prop} and deduce that
\begin{align*}
[u^\varepsilon_+]_{W^{s, 1}(\Omega')} & \le C \Big( 1 + \left( 1 + \| f_+ \|_{L^1(\Omega')} + c_\varepsilon |\Omega'| \right) \| u^\varepsilon_+ \|_{L^\infty(\Omega')} \Big) \\
& \le C \Big( 1 + \left( 2 + \| f_+ \|_{L^1(\Omega)} \right) \| u_+ \|_{L^\infty(\Omega)} \Big)
\end{align*}
for some constant~$C > 0$ depending only on~$n$,~$s$,~$g$, and~$\Omega$. For the second inequality, we also took advantage of expression~\eqref{silence} for~$u^\varepsilon$, property~\eqref{errorproperty} for~$c_\varepsilon$, and assumed~$\varepsilon$ to be suitably small. Letting~$\varepsilon \searrow 0$ in the above inequality, by Fatou's lemma one gets a~$W^{s, 1}(\Omega')$ bound for~$u_+$ with constant independent of~$\Omega'$. Estimate~\eqref{u+inWs1} then follows by the arbitrariness of~$\Omega' \Subset \Omega$, using again Fatou's lemma.
\end{proof}

In order to extend Proposition~\ref{ViscWeak_reg} to the case of general exterior data, and thus prove Theorem~\ref{Gen_viscweak} in its full generality, we will use a particular approximation procedure. The crucial point is represented by the following observation, which follows essentially from the fact that~$\h^{\ge d}u(x)$ can be bounded independently of both~$u$ and~$x$---see Remark~\ref{rmk_tail}.

\begin{lemma}\label{approtrick}
Let~$\Omega' \Subset \Omega\subseteq\R^n$ be bounded open sets,~$f\in C(\overline\Omega)$, and~$u$ be a viscosity subsolution of~$\h u = f$ in~$\Omega$, locally integrable in~$\R^n$. Let~$\{ u_k \} \subseteq L^1_\loc(\R^n)$ be a sequence of functions converging to~$u$ in $L^1_{\loc}(\R^n)$. Define
\begin{equation*}
\bar u_k(x):=
\begin{cases}
u(x) & \quad \mbox{if }x\in\Omega,\\
u_k(x) & \quad \mbox{if }x\in\R^n\setminus\Omega,
\end{cases}
\end{equation*}
Then,~$\bar{u}_k$ is a viscosity subsolution of~$\h \bar{u}_k = f + e_k$ in~$\Omega'$, for some non-negative constant~$e_k$ such that~$e_k \rightarrow 0$ as~$k \rightarrow +\infty$.
\end{lemma}

\begin{proof}
We write~$d:=\dist(\Omega',\partial\Omega)>0$ and observe that, for every~$x\in\Omega'$,
\eqlab{\label{approdel2}
\delta_g(\bar u_k,x;\xi)=\delta_g(u,x;\xi) \quad \mbox{for all } \xi\in B_d.
}
On the other hand, set
\bgs{
\omega_k(x):=\h^{\ge d}\bar u_k(x)-\h^{\ge d}u(x) \quad \mbox{for every } x\in\Omega'
}
and let~$R_0>0$ be such that~$\Omega\subseteq B_{R_0}$. Then, for every~$x\in\Omega'$ and~$R\ge d$ we have
\bgs{
|\omega_k(x)|&\le2\int_{\R^n\setminus B_d(x)} \left| G \left( \frac{u(x)-\bar u_k(y)}{|x-y|} \right) - G \left( \frac{u(x)-u(y)}{|x-y|} \right) \right|
\frac{dy}{|x-y|^{n+s}}\\
&
\le2\int_{B_R(x)\setminus B_d(x)}\frac{|\bar u_k(y)-u(y)|}{|x-y|^{n+1+s}}\,dy
+2\Lambda\int_{\R^n\setminus B_R(x)}\frac{dy}{|x-y|^{n+s}}\\
&
\le\frac{2}{d^{n+1+s}}\|u_k-u\|_{L^1(B_{R+R_0})}+\frac{2\Lambda \Ha^{n - 1}(\mathbb{S}^{n - 1})}{s R^s},
}
where for the second inequality we took advantage of~\eqref{lip_G_bla} and~\eqref{Gbounds}. Hence, setting
$$
e_k:=\inf_{R\ge d}\left(\frac{2}{d^{n+1+s}}\|u_k-u\|_{L^1(B_{R+R_0})}+\frac{2\Lambda \Ha^{n - 1}(\mathbb{S}^{n - 1})}{s R^s}\right),
$$
we obtained that
\eqlab{\label{unif_est}
\sup_{\Omega'} |\omega_k| \le e_k \quad \mbox{for every } k\in\N.
}
Notice that, as~$u_k\to u$ in $L^1_{\loc}(\R^n)$ by assumption, it holds
\bgs{
\limsup_{k\to\infty}e_k\le \frac{2}{d^{n+1+s}} \lim_{k\to\infty}\|u_k-u\|_{L^1(B_{R+R_0})} + \frac{2\Lambda \Ha^{n - 1}(\mathbb{S}^{n - 1})}{s R^s}
=\frac{2\Lambda \Ha^{n - 1}(\mathbb{S}^{n - 1})}{s R^s},
}
for every~$R\ge d$. Since~$R$ can be taken arbitrarily large, it follows that~$e_k \to 0$.

Let now~$x_0\in\Omega'$ and suppose that there exist~$r \in (0, \dist(x_0,\partial \Omega'))$ and~$v\in C^{1,1}(B_r(x_0))$ such that
\bgs{
v(x_0)=\bar u_k(x_0)=u(x_0)\quad\mbox{and}\quad v(x)\ge \bar u_k(x)=u(x)\quad \mbox{for all } x\in B_r(x_0).
}
By Proposition \ref{pwise_prop}, we then infer that~$\h u(x_0)$ is well-defined and satisfies~$\h u(x_0)\le f(x_0)$. Hence,~$\h \bar{u}_k(x_0)$ is well-defined too and, taking into account~\eqref{approdel2} and~\eqref{unif_est}, we get
$$
\h \bar u_k(x_0) =\h^{<d}u(x_0)+\h^{\ge d}\bar u_k(x_0)=\h u(x_0)+\omega_k(x_0) \le f(x_0)+e_k.
$$
The conclusion of the theorem is then an immediate consequence of Remark~\ref{mah}.
\end{proof}

With this approximation tool at hand, we are ready to tackle Theorem~\ref{Gen_viscweak} in its full generality.

\begin{proof}[Proof of Theorem~\ref{Gen_viscweak}]
We only prove the second (global) statement, as the first (local) one is then an immediate consequence---to this aim, notice that a viscosity subsolution~$u$ is always locally bounded from above in~$\Omega$, thanks to its upper semicontinuity.

Thus, we assume~$\Omega \subseteq \R^n$ to be open, bounded, and Lipschitz, and~$u$ to be a viscosity subsolution of~$\h u = f$ in~$\Omega$, bounded from above in the whole~$\Omega$. Recalling Definition~\ref{weaksoldef}, we only need to prove that, for every open set~$\Omega' \Subset \Omega$ with Lipschitz boundary,
\begin{equation} \label{viscisweakclaim}
u \mbox{ is a weak subsolution of } \h u = f \mbox{ in } \Omega' \quad \mbox{ and } \quad u_+ \in \W^s(\Omega').
\end{equation}
Notice that this would also yield that~$\max \{ u, k\} \in \W^s(\Omega)$ for every~$k \in \R$, thanks to Proposition~\ref{localWs1prop}, Fatou's lemma, and the fact that~$u$ is a subsolution bounded from above if and only if~$u - k$ is.

We prove~\eqref{viscisweakclaim} in three steps, of increasing degree of generality.

\begin{enumerate}[label=\emph{Step~$\arabic*$},leftmargin=*,wide]

\itemsep2pt

\item \hspace{-5pt}. \label{Step1} We first establish~\eqref{viscisweakclaim} under the additional assumption that
\begin{equation} \label{uinL1loc}
u \in L^\infty(\R^n).
\end{equation}

As~$u$ satisfies~\eqref{uinL1loc}, there exists a sequence~$\{ u_k \} \subseteq C(\R^n)\cap L^\infty(\R^n)$ converging to~$u$ in~$L^1_\loc(\R^n)$ and a.e.~in~$\R^n$. Let now~$\bar u_k$ and~$e_k$ be as in the statement of Lemma~\ref{approtrick}. Notice that the functions~$\bar u_k$ are globally bounded. Moreover, since $u$ is upper semicontinuous in $\Omega$ and $u_k$ is continuous in $\R^n$,
the functions $\bar u_k$ are upper semicontinuous in $\R^n\setminus\partial\Omega$.

In light of this and of the fact that~$\partial \Omega$ has zero Lebesgue measure (being Lipschitz), we can apply Lemma~\ref{approtrick} and Proposition~\ref{ViscWeak_reg}, deducing that~$\bar{u}_k$ is a weak subsolution of~$\h \bar{u}_k = f + e_k$ in~$\Omega'$ and that~$u_+|_{\Omega'} = (\bar{u}_k)_+|_{\Omega'} \in W^{s, 1}(\Omega')$. Then, since~$\bar u_k\to u$ a.e.~in~$\R^n$ and~$e_k\to 0$, by Lemma~\ref{easylemma} we easily deduce that~$u$ is a weak subsolution of~$\h u = f$ in~$\Omega'$.

This concludes the proof of claim~\eqref{viscisweakclaim} under the extra hypothesis~\eqref{uinL1loc}.

\item \hspace{-5pt}. \label{Step2} We now prove~\eqref{viscisweakclaim} assuming only that
\begin{equation} \label{ulocbounded}
u \in L^\infty(\Omega).
\end{equation}

Take~$k > \| u \|_{L^\infty(\Omega)}$ and set
$$
u_k := \begin{cases}
k & \quad \mbox{in } \{ u \ge k \},\\
u & \quad \mbox{in } \{ - k < u < k \},\\
-k & \quad \mbox{in } \{ u \le - k \}.
\end{cases}
$$
Notice that~$u_k$ coincides with~$u$ in~$\Omega$ and is thus upper semicontinuous in~$\Omega$. Let~$x_0 \in \Omega'$ and suppose that~$u_k$ is touched from above by a~$C^{1, 1}$ function at~$x_0$. Then, this function also touches~$u$ from above at~$x_0$ and, by Proposition~\ref{pwise_prop}, the quantity~$\h u(x_0)$ is well-defined and satisfies~$\h u(x_0) \le f(x_0)$. Clearly,~$\h u_k(x_0)$ is also well-defined and we have
\begin{align*}
\h u_k(x_0) & \le f(x_0) + \int_{\{ |u| > k \}} \left\{ G \left( \frac{u(x_0) - u_k(y)}{|x_0 - y|} \right) - G \left( \frac{u(x_0) - u(y)}{|x_0 - y|} \right) \right\} \frac{dy}{|x_0 - y|^{n + s}} \\
& \le f(x_0) + \Lambda \int_{\{ |u| > k \}} \frac{dy}{|x_0 - y|^{n + s}},
\end{align*}
where the second inequality follows from~\eqref{Gbounds}. From the boundedness of~$\Omega'$ it is easy to deduce the existence of a constant~$C > 0$, depending only on~$n$,~$s$,~$\diam(\Omega')$, and~$\dist(\Omega', \partial \Omega)$, such that
$$
|x_0 - y|^{- n - s} \le C (1 + |y|)^{- n - s} \quad \mbox{for all } y \in \Co \Omega.
$$
Observing that~$\{ |u| > k \} \subseteq \Co \Omega$, we get
$$
\h u_k(x_0) \le f(x_0) + C \Lambda \int_{\R^n} \frac{\chi_{\{ |u| > k \}}(y)}{(1 + |y|)^{n + s}} \, dy =: f(x_0) + \delta_k.
$$

Recalling Remark~\ref{mah},~$u_k$ is a viscosity subsolution of~$\h u_k = f + \delta_k$ in~$\Omega'$. By dominated convergence and the fact that~$\chi_{\{ |u| > k \}} \rightarrow 0$ a.e.~in~$\R^n$, we have that~$\delta_k \searrow 0$ as~$k \rightarrow +\infty$. As~$u_k \in L^\infty(\R^n)$ and~$f + \delta_k \in C(\overline{\Omega})$, we may apply what we proved in~\ref{Step1}, obtaining that~$u_k$ is a weak subsolution of~$\h u_k = f + \delta_k$ in~$\Omega'$ such that~$u_+|_{\Omega'} = (u_k)_+|_{\Omega'} \in \W^s(\Omega')$. By Lemma~\ref{easylemma} and the fact that~$u_k \rightarrow u$ a.e.~in~$\R^n$, we immediately conclude the validity of~\eqref{viscisweakclaim} under assumption~\eqref{ulocbounded}.

\item \hspace{-5pt}. \label{Step3} We now show that~\eqref{viscisweakclaim} holds true for a general viscosity subsolution~$u$ bounded from above in~$\Omega$. Note that, up to a partition of unity argument, it suffices to prove~\eqref{viscisweakclaim} for~$\Omega'$ equal to any ball~$B = B_\varrho(\bar{x}) \Subset \Omega$ of radius~$\varrho > 0$ arbitrarily small.

For~$k > 0$, consider the function~$\varphi_k := - k - \chi_{B_{2 \varrho}(\bar{x})}$. If~$\varrho \in (0, 1)$ is small enough,~$\varphi_k$ satisfies
\begin{equation} \label{phibarrier}
\h \varphi_k \le f \quad \mbox{in } B,
\end{equation}
Indeed, changing variables appropriately and taking advantage of the oddness and monotonicity of~$G$, for every~$x \in B$ we have
\begin{align*}
\h \varphi_k(x)
& = - \int_{\Co B_{2 \varrho}(\bar{x})} G \left( \frac{1}{|y - x|} \right) \frac{dy}{|y - x|^{n + s}} \\
& \le - \varrho^{-s} \int_{B_3 \setminus B_2} G \left( \frac{1}{|\varrho z - (x - \bar{x})|} \right) \frac{dz}{|z - (x - \bar{x})/\varrho|^{n + s}} \le - \frac{|B_3 \setminus B_2| \, G(1/4)}{4^{n + s}} \, \varrho^{-s},
\end{align*}
which is smaller than~$-\| f_- \|_{L^\infty(\Omega)}$, provided~$\varrho$ is sufficiently small. Hence,~\eqref{phibarrier} holds true.

Let now~$\hat{u}_k := \max \{ u, \varphi_k \}$. Clearly,~$\hat{u}_k$ is bounded and upper semicontinuous in~$B$. Suppose now that~$\hat{u}_k$ is touched from above by a~$C^{1, 1}$ function~$v$ at some point~$x \in B$. Then,~$v$ also touches either~$u$ or~$\varphi_k$ from above at~$x$. In both cases, using~\eqref{phibarrier}, the monotonicity of~$G$, and Proposition~\ref{pwise_prop}, we easily deduce that~$\h \hat{u}_k(x) \le f(x)$. Accordingly,~$\hat{u}_k$ is a viscosity subsolution of~$\h \hat{u}_k = f$ in~$B$, bounded in~$B$. By~\ref{Step2},~$\hat{u}_k$ is then also a weak subsolution of the same equation and its positive part belongs to~$\W^s(B)$. As~$\hat{u}_k \rightarrow u$ a.e.~in~$\R^n$ as~$k \rightarrow + \infty$, by Lemma~\ref{easylemma} we conclude that~$u$ is a weak subsolution as well. Using Proposition~\ref{localWs1prop} and Fatou's lemma, we also obtain that~$u_+ \in \W^s(B)$.
\end{enumerate}

\vspace{-3pt}

The proof of Theorem~\ref{Gen_viscweak} is thus complete.
\end{proof}

\section{Existence of minimizers. Proofs of Theorems~\ref{Dirichlet} and~\ref{Perron_Thm}} \label{Linftysec}

\noindent
In this section, we deal with the existence of minimizers and local minimizers of~$\F$ in an open set, with respect to a given outside datum.

\subsection{Existence of minimizers via the Direct Method} \label{minviadirect}

We begin by establishing the existence of minimizers of~$\F$ in a bounded Lipschitz set~$\Omega$ among all functions which agree with a given function~$\varphi$ outside of~$\Omega$. That is, we prove Theorem~\ref{Dirichlet}.

As anticipated in the Introduction, we prove the existence of minimizers through an approximation procedure that makes use of the truncated functionals~$\F^M$ and of their minimizers within an appropriate family of spaces, that we define as follows.

Given a bounded open set~$\Omega\subseteq\R^n$ and~$M\ge0$, we consider the spaces
$$
\B\W^s(\Omega):= \Big\{ u \in\W^s(\Omega) : u|_\Omega\in L^\infty(\Omega) \Big\}
$$
and
$$
\B_M\W^s(\Omega):= \Big\{ u\in\B\W^s(\Omega) : \|u\|_{L^\infty(\Omega)}\le M \Big\}.
$$
Moreover, given a function~$\varphi:\Co\Omega\to\R$, we define
$$
\B\W^s_\varphi(\Omega) := \Big\{ u\in\B\W^s(\Omega) : u=\varphi\textrm{ a.e.~in }\Co\Omega \Big\}
$$
and
$$
\B_M\W^s_\varphi(\Omega) := \Big\{ u\in\B_M\W^s(\Omega) : u=\varphi\textrm{ a.e.~in }\Co\Omega \Big\}.
$$

Our notion of minimality for the Dirichlet problem having~$\varphi$ as outside datum is essentially that of Definition~\ref{mindef}. Namely, a function~$u\in\W^s_\varphi$ is a minimizer of~$\F$ in~$\W^s_\varphi(\Omega)$ if
$$
\iint_{Q(\Omega)}\left\{\G\left(\frac{u(x)-u(y)}{|x-y|}\right)-\G\left(\frac{v(x)-v(y)}{|x-y|}\right)\right\}
\frac{dx\,dy}{|x-y|^{n-1+s}}\le0,
$$
for every~$v\in\W^s_\varphi(\Omega)$. To obtain Theorem~\ref{Dirichlet}, we first solve the Dirichlet problem in~$\B_M \W^s_\varphi(\Omega)$, for a fixed~$M > 0$. This is achieved easily with the aid of the following two results.

First, we observe that~$\F^M$ is lower semicontinuous in~$\B_M\W^s(\Omega)$ with respect to pointwise convergence almost everywhere.

\begin{lemma} \label{semicontlem}
Let~$\Omega \subseteq \R^n$ be an open set and~$M > 0$. Let~$\{ u_k \} \subseteq\B_M\W^s(\Omega)$ be a sequence of functions converging to some~$u: \R^n \to \R$ a.e.~in~$\R^n$.
Then,
\bgs{
\F^M(u,\Omega)\le\liminf_{k\to\infty}\F^M(u_k,\Omega).
}
\end{lemma}

\begin{proof}
The proof is a consequence of Fatou's lemma, applied separately to the functionals~$\A$ and~$\Nl^M$.
Notice that, in order to use this result with~$\Nl^M$, 
the uniform bound~$\|u_k\|_{L^\infty(\Omega)}\le M$ is important to guarantee that the quantity inside square brackets in~\eqref{NMldef}
is non-negative---recall that~$\overline{G} \ge 0$ by definition~\eqref{Gbardef}.
\end{proof}

Next is a compactness result for sequences uniformly bounded with respect to~$\A$. 

\begin{lemma} \label{complem}
Let~$\Omega \subseteq \R^n$ be a bounded open set with Lipschitz boundary. Let~$\{ u_k \}$ be a sequence of functions~$u_k: \Omega \to \R$ satisfying
$$
\sup_{k \in \N} \Big( \|u_k\|_{L^1(\Omega)}+\A(u_k, \Omega) \Big) < \infty.
$$
Then, up to a subsequence,~$\{ u_k \}$ converges to a function~$u \in W^{s, 1}(\Omega)$ in $L^1(\Omega)$ and a.e.~in~$\Omega$.
\end{lemma}

Lemma~\ref{complem} follows at once from the coercivity of~$\A$ with respect to the~$W^{s, 1}(\Omega)$ seminorm observed in Lemma~\ref{Adomainlem} and the compact embedding~$W^{s, 1}(\Omega) \hookrightarrow \hookrightarrow L^1(\Omega)$---see,~e.g.,~\cite[Theorem~7.1]{DPV12}.

By combining these two results, we easily obtain the existence of a (unique) minimizer~$u_M$ of~$\F^M$ among all functions in~$\B_M \W^s(\Omega)$ with fixed values outside of~$\Omega$.

\begin{prop}\label{ertyui}
Let~$\Omega \subseteq \R^n$ be a bounded open set with Lipschitz boundary and~$\varphi: \Co \Omega \to \R$ be a given function. For every~$M > 0$, there exists a unique minimizer~$u_M$ of~$\F^M(\,\cdot\,, \Omega)$ in~$\B_M\W^s_\varphi(\Omega)$, i.e., a unique~$u_M \in \B_M \W^s_\varphi(\Omega)$ for which
\begin{equation} \label{uMmin}
\F^M(u_M,\Omega) = \inf \Big\{ \F^M(v, \Omega) : v \in \B_M\W_\varphi^s(\Omega) \Big\}.
\end{equation}
\end{prop}
\begin{proof}
Since~$\B_M\W^s_\varphi(\Omega)$ is a convex subset of~$\W^s_\varphi(\Omega)$,
the uniqueness of the minimizer of~$\F^M(\,\cdot\,, \Omega)$ within~$\B_M\W^s_\varphi(\Omega)$
is a consequence of the strict convexity of~$\F^M(\,\cdot\,, \Omega)$---see point~\ref{convlem_ii} of Lemma~\ref{conv_func}.
Therefore, we are only left to establish its existence.

Let~$\{ u^{(k)} \} \subseteq \B_M\W^s_\varphi(\Omega)$ be a minimizing sequence, that is
$$
\lim_{k\to\infty}\F^M(u^{(k)},\Omega)=\inf \Big\{ \F^M(v, \Omega) : v \in \B_M\W_\varphi^s(\Omega) \Big\} =: m.
$$
Observe that both~$\A$ and~$\Nl^M$ are non-negative in~$\B_M\W_\varphi^s(\Omega)$---recall definitions~\eqref{Adef} and~\eqref{NMldef}. Hence,~$m \ge 0$ and~$\A(u^{(k)}) \le \F^M(u^{(k)},\Omega) \le m + 1$ for~$k$ large enough. In light of Lemma~\ref{complem}, we then deduce that~$\{ u^{(k)} \}$ converges (up to a subsequence) to a function~$u_M \in \B_M\W^s_\varphi(\Omega)$ a.e.~in~$\R^n$. Identity~\eqref{uMmin} follows by applying Lemma~\ref{semicontlem}.
\end{proof}

Proposition~\ref{ertyui} shows that, for each~$M > 0$, there exists a unique minimizer~$u_M$ within the space~$\B_M\W_\varphi^s(\Omega)$. To establish the existence of a minimizer of~$\F$ in the whole~$\W^s_\varphi(\Omega)$---and thus prove Theorem~\ref{Dirichlet}---, we need~$u_M$ to stabilize as~$M\to\infty$. This is achieved through the uniform~$W^{s, 1}$ estimate of Proposition~\ref{Ws1prop}, at the price of assuming some (weighted) integrability on the exterior datum in a sufficiently large neighborhood of~$\Omega$.

\begin{proof}[Proof of Theorem \ref{Dirichlet}]
Let~$\Theta > 0$ be the constant given by Proposition~\ref{Ws1prop}. For any~$M > 0$, the minimizer~$u_M$ satisfies the hypotheses of Proposition~\ref{Ws1prop}. Therefore,
\begin{equation} \label{Ws1estforuM}
\| u_M \|_{W^{s, 1}(\Omega)} \le C \left( \left\| \Tail_s(\varphi, \Omega_{\Theta \diam(\Omega)} \setminus \Omega;\,\cdot\,) \right\|_{L^1(\Omega)} + 1 \right),
\end{equation}
for some constant~$C > 0$ depending only on~$n$,~$s$,~$g$, and~$\Omega$---in particular,~$C$ is independent of~$M$.

By the compact fractional Sobolev embedding (see, e.g.,~\cite[Theorem~7.1]{DPV12}),
we conclude that there exists a function~$u \in \W^s_\varphi(\Omega)$ to which~$\{ u_{M_j} \}$
converges in~$L^1(\Omega)$ and~a.e.~in~$\Omega$, for some diverging sequence~$\{ M_j \}_{j \in \N}$.
Letting~$M = M_j \rightarrow +\infty$ in~\eqref{Ws1estforuM}, by Fatou's Lemma
we see that~$u$ satisfies~\eqref{Ws1estformin}.
We are therefore left to show that~$u$ is a minimizer for~$\F$ in~$\W^s_\varphi(\Omega)$.

Take~$v \in \B\W^s_\varphi(\Omega)$. Then, for~$j$ large enough we have $M_j\ge\| v \|_{L^\infty(\Omega)}$,
and hence, by the minimality of~$u_{M_j}$ we get~$\F^{M_j}(u_{M_j}, \Omega) \le \F^{M_j}(v, \Omega)$. Equivalently,
\begin{equation} \label{equivuMjmin}
\begin{aligned}
0 & \ge \A(u_{M_j}) + 2 \int_{\Omega} \left\{ \int_{\Omega_R \setminus \Omega} \G \left( \frac{u_{M_j}(x) - \varphi(y)}{|x-y|} \right) \frac{dy}{|x-y|^{n - 1 + s}} \right\} dx \\
& \quad - \A(v) - 2 \int_{\Omega} \left\{ \int_{\Omega_R \setminus \Omega} \G \left( \frac{v(x) - \varphi(y)}{|x-y|} \right) \frac{dy}{|x-y|^{n - 1+ s}} \right\} dx \\
& \quad + 2 \int_{\Omega} \left\{ \int_{\Co \Omega_R} \left\{ \G \left( \frac{u_{M_j}(x) - \varphi(y)}{|x-y|} \right) - \G \left( \frac{v(x) - \varphi(y)}{|x-y|} \right) \right\} \frac{dy}{|x-y|^{n - 1 + s}} \right\} dx,
\end{aligned}
\end{equation}
for any fixed~$R \in (0, \Theta \diam(\Omega)]$. Note that such a choice for~$R$ guarantees the finiteness of all the quantities appearing in~\eqref{equivuMjmin}, taking advantage of the properties of~$\G$ and of hypothesis~\eqref{TailinL1}.

We now claim that letting~$j \rightarrow +\infty$ in~\eqref{equivuMjmin},
we obtain the same inequality with~$u_{M_j}$ replaced by~$u$. Indeed, the quantities on the first line can be dealt with by using Fatou's lemma.
Moreover, the Lipschitz character of~$\G$---see~\eqref{Lip_Gcal}---and the fact that~$u_{M_j} \rightarrow u$ in~$L^1(\Omega)$ ensure that
\bgs{
& \lim_{j \rightarrow +\infty} \int_{\Omega} \left\{ \int_{\Co \Omega_R} \left| \G \left( \frac{u_{M_j}(x) - \varphi(y)}{|x-y|} \right) - \G \left( \frac{u(x) - \varphi(y)}{|x-y|} \right) \right| \frac{dy}{|x-y|^{n - 1 + s}} \right\} dx\\
& \hspace{130pt} \le \frac{\Lambda}{2} \lim_{j \rightarrow +\infty} \int_\Omega|u_{M_j}(x)-u(x)| \left( \int_{\Co B_R(x)}\frac{dy}{|x-y|^{n+s}}\right) dx \\
& \hspace{130pt} \le \frac{\Lambda \Ha^{n - 1}(\S^{n - 1})}{2s R^{s}} \lim_{j \rightarrow +\infty} \|u_{M_j}-u\|_{L^1(\Omega)} = 0.
}
Hence, the third line passes to the limit as well.
All in all, we have proved that~$u$ minimizes~$\F$ in~$\B\W_\varphi^s(\Omega)$. The minimality of~$u$
within the larger class~$\W^s_\varphi(\Omega)$ follows from density arguments, using,~e.g.,~Proposition~\ref{smooth_cpt_dense} and Lemma~\ref{tartariccio}. See also point~\ref{tartapower_iv} of Remark~\ref{tartapower_Remark}.

Finally, the uniqueness of the minimizer follows by point~\ref{tartapower_iii} of Remark~\ref{tartapower_Remark}.
\end{proof}

\begin{remark}\label{Comments_minim}
Here are some mostly technical observations on Theorem~\ref{Dirichlet} and its proof.
\begin{enumerate}[label=$(\roman*)$,leftmargin=*]
\itemsep2pt
\item The strategy just displayed is inspired by the one employed in~\cite{Cyl} by the second author to obtain the existence of~$s$-minimal surfaces in general open sets---thus extending~\cite[Theorem~3.2]{CRS10} to the case of unbounded or irregular~$\Omega$. However, there is a striking difference between Theorem~\ref{Dirichlet} here and, say,~\cite[Corollary~1.11]{Cyl}: there, the existence of a (locally) $s$-minimal set is obtained under no restriction on the outside datum, whereas here we need to limit ourselves to data satisfying~\eqref{TailinL1}.

We believe that it would be interesting to understand whether Theorem~\ref{Dirichlet} could be proved under weaker or even no assumptions on~$\varphi$ (obtaining perhaps only a local minimizer of~$\F$), or whether, in the geometric case~$g = g_s$, the local minimizers constructed in~\cite{Cyl} are necessarily subgraphs inside~$\Omega^\infty$.

\item In light of point~\ref{tartapower_ii} of Remark~\ref{tartapower_Remark}, we could have proceeded to directly minimize the functional~$\F^M$ in~$\W^s_\varphi(\Omega)$, for some fixed~$M \ge 0$, instead of considering a family of approximating problems. This approach works as well, but brings in its own difficulties, first and foremost the fact that the functional~$\F^M$ may assume negative values in~$\W^s(\Omega) \setminus \B_M\W^s(\Omega)$, as shown by Example~\ref{Exe_neg_nonloc}. In addition, we preferred the use of several~$\F^M$'s in order to maintain an analogy with the argument of~\cite{Cyl} and keep a connection with the underlying geometry, as motivated by the results of Subsection~\ref{areageom}.

\item A different strategy to obtain Theorem~\ref{Dirichlet}---similar to the one employed in~\cite[Section~12]{G84}---is to show that the~$L^\infty(\Omega)$ norm, and not the~$W^{s, 1}(\Omega)$ norm, of~$u_M$ stabilizes for large~$M$. This can be done, depending on the exterior data, using the~$L^\infty$ estimates of Subsections~\ref{locboundsub} and~\ref{globboundsub} in place of Proposition~\ref{Ws1prop}. This approach will be exploited in order to establish the existence of a solution to the obstacle problem in~\cite{LuCla}.

\item When the exterior datum~$\varphi$ satisfies the global integrability condition~\eqref{Exterior_global_cond}, the proof of Theorem~\ref{Dirichlet} can be simplified considerably. Indeed, in this case the functional~\eqref{Fcdef} is well defined in~$\W^s_\varphi(\Omega)$ and one can minimize it directly, with no need to consider the approximate minimizers~$u_M$ or the truncated functionals~$\F^M$---see Section~\ref{GlobTail_section} for the rigorous arguments.

However, although natural to deal with functional~\eqref{Fcdef}, condition~\eqref{Exterior_global_cond} is quite restrictive on the behavior of~$\varphi$ at infinity and does not play any role in the well-posedness of the operator~$\h$, which corresponds to the first variation of~$\F$ or~$\F^M$---recall the end of Section~\ref{Intro_motives} for a more detailed discussion.
\end{enumerate}
\end{remark}

As shown in the following Lemma, the integrability of the restricted tail prescribed by~\eqref{TailinL1} is equivalent to plain integrability of the datum in the exterior neighborhood~$\Omega_{\Theta \diam(\Omega)} \setminus \Omega$ plus weighted integrability arbitrarily close to the boundary of~$\Omega$. 

\begin{lemma}\label{tail_equiv_cond_Lemma}
Let~$\Omega \Subset\Op\subseteq \R^n$ be two bounded open sets, with~$\partial \Omega$ Lipschitz. Let~$r \in (0, \dist(\Omega, \partial \Op))$ and~$\varphi:\Co\Omega\to\R$.

Then,~$\Tail_s(\varphi, \Op \setminus \Omega;\,\cdot\,) \in L^1(\Omega)$
if and only if~$\varphi\in L^1(\Op\setminus\Omega)$
and~$\Tail_s(\varphi, \Omega_r \setminus \Omega;\,\cdot\,) \in L^1(\Omega\setminus\Omega_{-r})$. Moreover,
\begin{enumerate}[label=$(\roman*)$,leftmargin=*]
\item \label{tail_equiv_i} if~$\varphi\in L^1(\Op\setminus\Omega) \cap W^{s,1}(\Omega_r\setminus\Omega)$,
then~$\Tail_s(\varphi, \Op \setminus \Omega;\,\cdot\,) \in L^1(\Omega)$;
\item \label{tail_equiv_ii} if~$\varphi\in L^1(\Op\setminus\Omega) \cap L^\infty(\Omega_r\setminus\Omega)$,
then~$\Tail_\sigma(\varphi, \Op \setminus \Omega;\,\cdot\,) \in L^1(\Omega)$
for every~$\sigma\in(0,1)$.
\end{enumerate}
\end{lemma}

\begin{proof}
To begin, assume that~$\Tail_s(\varphi, \Op \setminus \Omega;\,\cdot\,) \in L^1(\Omega)$. Note that~$|x-y| \le \diam(\Op)$ for every~$(x, y)\in\Omega \times \Op\setminus\Omega$. Hence,
\bgs{
\|\varphi\|_{L^1(\Op\setminus\Omega)}\le \frac{\diam(\Op)^{n+s}}{|\Omega|}\left\|\Tail_s(\varphi, \Op \setminus \Omega;\,\cdot\,)\right\|_{L^1(\Omega)}.
}
Moreover, we clearly have
\bgs{
\left\|\Tail_s(\varphi, \Omega_r \setminus \Omega;\,\cdot\,)\right\|_{L^1(\Omega\setminus\Omega_{-r})}\le
\left\|\Tail_s(\varphi, \Op \setminus \Omega;\,\cdot\,)\right\|_{L^1(\Omega)}
}
for every~$r \in (0, \dist(\Omega, \partial \Op))$. Accordingly,~$\varphi\in L^1(\Op\setminus\Omega)$ and~$\Tail_s(\varphi, \Omega_r \setminus \Omega;\,\cdot\,) \in L^1(\Omega\setminus\Omega_{-r})$.

Next, let~$\varphi\in L^1(\Op\setminus\Omega)$ be such that~$\Tail_s(\varphi, \Omega_r \setminus \Omega;\,\cdot\,) \in L^1(\Omega\setminus\Omega_{-r})$ for some~$r \in (0, \dist(\Omega, \partial \Op))$. We verify that~$\Tail_s(\varphi, \Op \setminus \Omega;\,\cdot\,) \in L^1(\Omega)$. Indeed, since~$|x-y|\ge r$ for every~$(x, y) \in\Omega \times \Op\setminus\Omega_r$ and~$(x, y) \in \Omega_{- r} \times \Omega_r \setminus \Omega$, we have
\eqlab{\label{alatriste_e_le_maledette_code}
\left\|\Tail_s(\varphi, \Op \setminus \Omega_r;\,\cdot\,)\right\|_{L^1(\Omega)}
\le\frac{|\Omega|}{r^{n+s}}\|\varphi\|_{L^1(\Op\setminus\Omega_r)}
}
and
\bgs{
\left\|\Tail_s(\varphi, \Omega_r \setminus \Omega;\,\cdot\,)\right\|_{L^1(\Omega_{-r})}
\le\frac{|\Omega_{-r}|}{r^{n+s}}\|\varphi\|_{L^1(\Omega_r\setminus\Omega)}
\le\frac{|\Omega|}{r^{n+s}}\|\varphi\|_{L^1(\Omega_r\setminus\Omega)}.
}
Therefore,
\bgs{
\left\|\Tail_s(\varphi, \Op \setminus \Omega;\,\cdot\,)\right\|_{L^1(\Omega)}&
=\left\|\Tail_s(\varphi, \Op \setminus \Omega_r;\,\cdot\,)\right\|_{L^1(\Omega)}
+\left\|\Tail_s(\varphi, \Omega_r \setminus \Omega;\,\cdot\,)\right\|_{L^1(\Omega_{-r})}\\
&
\quad+\left\|\Tail_s(\varphi, \Omega_r\setminus\Omega;\,\cdot\,)\right\|_{L^1(\Omega\setminus\Omega_{-r})}\\
&
\le\frac{|\Omega|}{r^{n+s}}\|\varphi\|_{L^1(\Op\setminus\Omega)}
+\left\|\Tail_s(\varphi, \Omega_r\setminus\Omega;\,\cdot\,)\right\|_{L^1(\Omega\setminus\Omega_{-r})} < \infty.
}

We now address point~\ref{tail_equiv_i}. Without loss of generality, we assume~$r$ to be small enough for~$\partial \Omega_r$ to be Lipschitz. As~$\varphi\in W^{s,1}(\Omega_r\setminus\Omega)$, Corollary~\ref{FHI_corollary} (applied in the Lipschitz set~$\Omega_r \setminus \Omega$) yields
\bgs{
\left\|\Tail_s(\varphi, \Omega_r \setminus \Omega;\,\cdot\,)\right\|_{L^1(\Omega)}
\le C \|\varphi\|_{W^{s,1}(\Omega_r\setminus\Omega)},
}
for some constant~$C > 0$. The fact that~$\Tail_s(\varphi, \Op \setminus \Omega;\,\cdot\,) \in L^1(\Omega)$ follows then from this and~\eqref{alatriste_e_le_maledette_code}.

Finally, we deal with point~\ref{tail_equiv_ii}. If~$\varphi\in L^\infty(\Omega_r\setminus\Omega)$, then
\bgs{
\left\|\Tail_\sigma(\varphi, \Omega_r \setminus \Omega;\,\cdot\,)\right\|_{L^1(\Omega)}
\le\|\varphi\|_{L^\infty(\Omega_r\setminus\Omega)}\Per_\sigma(\Omega),
}
for every~$\sigma\in(0,1)$. Thus, we obtain~\ref{tail_equiv_ii} by using again~\eqref{alatriste_e_le_maledette_code}. This concludes the proof.
\end{proof}

\subsubsection{Integrable global tail}\label{GlobTail_section}

We briefly present here an alternative approach to the existence of minimizers of~$\F$, valid when the exterior datum~$\varphi$ satisfies the global summability condition~\eqref{Exterior_global_cond}. We begin by showing that in this situation the functional~$\F$ is well-defined on~$\W^s_\varphi(\Omega)$.

\begin{lemma}\label{quocca_style}
	Let~$\Omega \subseteq \R^n$ be a bounded open set with Lipschitz boundary and~$\varphi:\Co\Omega\to\R$ be a measurable function satisfying~\eqref{Exterior_global_cond}. Then,~$\F(u) \in [0, +\infty)$ for every~$u \in \W^s_\varphi(\Omega)$.
\end{lemma}

\begin{proof}
	In view of Lemma~\ref{Adomainlem}, we only need to prove that the nonlocal part~$\Nl$ is finite. For this, by~\eqref{GGbetterbound}, the triangle inequality, and Corollary~\ref{FHI_corollary}, given any function~$u\in\W^s_\varphi(\Omega)$ we have
	\bgs{
		\Nl(u,\Omega)&=2\int_\Omega\int_{\Co\Omega}\G\left(\frac{u(x)-u(y)}{|x-y|}\right)\frac{dx\,dy}{|x-y|^{n-1+s}}\le
		\Lambda\int_\Omega\int_{\Co\Omega}\frac{|u(x)-u(y)|}{|x-y|^{n+s}}dx\,dy\\
		&
		\le C\left(\|u\|_{W^{s,1}(\Omega)}+\left\|\Tail_s(\varphi,\Co\Omega;\,\cdot\,)\right\|_{L^1(\Omega)}\right)< + \infty.
	}
	The non-negativity of~$\F$ is an immediate consequence of its definition.
\end{proof}

The existence of a unique minimizer can then be obtained via the Direct Method of the Calculus of Variations.

\begin{prop}
	Let~$\Omega \subseteq \R^n$ a bounded open set with Lipschitz boundary and~$\varphi:\Co\Omega\to\R$ be a measurable function satisfying~\eqref{Exterior_global_cond}. Then, there exists a unique function~$u\in\W^s_\varphi(\Omega)$ such that
	\begin{equation}\label{Quocca_uber_alles}
	\F(u,\Omega)=\inf\Big\{\F(v,\Omega) : v\in\W^s_\varphi(\Omega)\Big\}.
	\end{equation}
\end{prop}

\begin{proof}
	First of all, we observe that, by Lemma~\ref{quocca_style}, the infimum of~$\F$ in~$\W^s_\varphi(\Omega)$ is finite and non-negative. Now, consider a minimizing sequence~$u_k\in\W^s_\varphi(\Omega)$, i.e.,
	\[
	\lim_{k\to\infty}\F(u_k,\Omega)=\inf\Big\{\F(v,\Omega) : v\in\W^s_\varphi(\Omega)\Big\}=:m.
	\]
	Since the nonlocal part~$\Nl$ is non-negative, by~\eqref{AGags} we have the uniform estimate
	\begin{equation}\label{Heidegger_puzza1}
	[u_k]_{W^{s,1}(\Omega)}\le \frac{2}{c_\star} \, \A(u_k,\Omega)+c_s(\Omega)\le \frac{2}{c_\star} \, \F(u_k,\Omega)+c_s(\Omega)\le \frac{2}{c_\star}(m+1)+c_s(\Omega),
	\end{equation}
	for every~$k$ large enough. Moreover, arguing as for~\eqref{newFPI} we easily obtain that
	\begin{equation}\label{Heidegger_puzza2}
		\|u_k\|_{L^1(\Omega)} \le C_1 \left( \int_\Omega\int_{\Omega_1\setminus\Omega}\frac{|u_k(x)-u_k(y)|}{|x-y|^{n+s}} \, dx\,dy+\left\| \Tail_s(\varphi,\Co\Omega;\,\cdot\,)\right\|_{L^1(\Omega)}\right),
	\end{equation}
	with~$C_1 := \diam(\Omega_1)^{n + s} / |\Omega_1 \setminus \Omega|$. Now, we observe that, by~\eqref{GGbetterbound} and Lemma~\ref{dumb_kernel_lemma}, it holds
	\begin{align*}
		\int_\Omega\int_{\Omega_1\setminus\Omega}\frac{|u_k(x)-u_k(y)|}{|x-y|^{n+s}} \, dx\,dy&
		\le\frac{2}{\Lambda}\int_\Omega\int_{\Omega_1\setminus\Omega}\left\{ \G\left(\frac{|u_k(x)-u_k(y)|}{|x-y|}\right)+\lambda\right\} \frac{dx\,dy}{|x-y|^{n-1+s}}\\
		&
		\le\frac{2}{\Lambda}\left(\Nl(u_k,\Omega)+\lambda\frac{\Ha^{n-1}(\mathbb S^{n-1})}{1-s}|\Omega|\diam(\Omega_1)^{1-s}\right) \le C_2 (m + 1),
	\end{align*}
	for some constant~$C_2 > 0$ depending only on~$n$,~$s$,~$g$, and~$\Omega$. Note that, for the third inequality we also took advantage of the non-negativity of~$\A$. Adding together~\eqref{Heidegger_puzza1},~\eqref{Heidegger_puzza2}, and the last estimate, we obtain that there exists a constant~$C_3>0$, depending only on~$n$,~$s$,~$g$, and~$\Omega$, such that
	\[
	\|u_k\|_{W^{s,1}(\Omega)}\le C_3 \left(m+1+\left\|\Tail_s(\varphi,\Co\Omega;\,\cdot\,)\right\|_{L^1(\Omega)}\right),
	\]
	for every~$k$ large enough. By the compactness of the embedding of~$W^{s,1}(\Omega)$ into~$L^1(\Omega)$, we find that, up to a subsequence,~$u_k$ converges to some function~$\tilde{u} \in W^{s, 1}(\Omega)$ a.e.~in~$\Omega$ and in~$L^1(\Omega)$. Then, if we define the function~$u\in\W^s_\varphi(\Omega)$ by setting~$u:=\tilde{u}$ in~$\Omega$ and~$u:=\varphi$ in~$\Co\Omega$, by Fatou's Lemma we conclude that
	\[
	m\le\F(u,\Omega)\le\liminf_{k\to\infty}\F(u_k,\Omega)=m.
	\]
	The uniqueness of the minimizer~$u$ follows from the strict convexity of the functional~$\F$ on~$\W^s_\varphi(\Omega)$---which can be proved by arguing as in Lemma~\ref{conv_func}.
\end{proof}

We stress that a function~$u\in\W^s_\varphi(\Omega)$ which minimizes~$\F$ in the sense of~\eqref{Quocca_uber_alles} is clearly also a minimizer in the sense of Definition~\ref{mindef}. Hence, all the results satisfied by minimizers apply also to minimizers in the sense of~\eqref{Quocca_uber_alles}. This is true in particular for the a priori estimates of Section~\ref{sect3} and for the results on the relationship existing in the geometric framework with nonlocal minimal graphs, which will be fully explored in Section~\ref{Rearrange_Sect}.

\subsection{Existence of local minimizers via a Perron-type result} \label{Perron_proof}

We present here a proof of Theorem~\ref{Perron_Thm}, which claims that the existence of local minimizers of~$\F$ is equivalent to the existence of an ordered pair of locally bounded weak sub- and supersolutions of~$\h u = 0$.

	\begin{proof}[Proof of Theorem~\ref{Perron_Thm}]
	Implication~\ref{perronequiv1}~$\Rightarrow$~\ref{perronequiv2} is an immediate consequence of Proposition~\ref{Linftylocprop} and Corollary~\ref{weak_implies_min_lemma}. By these results,~$u$ is locally bounded in~$\Omega$ and weakly solves~$\h u=0$ in~$\Omega$. Hence, we can consider~$\underline{u}=\overline{u}=u$.
	
	We now show that~\ref{perronequiv2}~$\Rightarrow$~\ref{perronequiv1}. Let~$u_0\in\W^s_\loc(\Omega)$ be the function defined by setting~$u_0|_\Omega:=\frac{\underline{u}+\overline{u}}{2}$ and $u_0|_{\Co\Omega}:=\varphi$. Consider a regular exhaustion of~$\Omega$, i.e.,~a sequence~$\{\Omega_h\}$ of bounded open sets with Lipschitz boundaries such that
	\[
	\Omega_h\Subset\Omega_{h+1}\Subset\Omega\quad\mbox{and}\quad\bigcup_{h=1}^\infty\Omega_h=\Omega.
	\]
	
	We first solve an auxiliary minimization problem in each~$\Omega_h$. Set
	\[
	M_h:=\max\Big\{\|\underline{u}\|_{L^\infty(\Omega_h)},\|\overline{u}\|_{L^\infty(\Omega_h)}\Big\}.
	\]
	By appropriately modifying the proof of Proposition~\ref{ertyui}, we easily find that there exists a unique function~$u_h\in\B_{M_h}\W^s_{u_0}(\Omega_h)$ such that
	\[
	\F^{M_h}(u_h,\Omega_h)=\inf\Big\{\F^{M_h}(v,\Omega_h) : v\in\W^s_{u_0}(\Omega_h)\mbox{ and~}\underline{u}\leq v\leq\overline{u}\mbox{ a.e.~in~}\Omega_h\Big\}.
	\]
	
	We then claim that, up to a subsequence, the functions~$u_h$'s converge in $L^1_\loc(\Omega)$ and a.e.~in~$\Omega$ to a function~$u\in\W^s_\loc(\Omega)$ such that~$u=\varphi$ a.e.~in~$\Co\Omega$. To this aim, we estimate
	\begin{equation*}
	[u_h]_{W^{s,1}(\Omega_k)}\leq C\big(\A(u_h,\Omega_k)+1\big)
	\leq C\big(\F^{M_k}(u_h,\Omega_k)+1\big),
	\end{equation*}
	for every integer~$h\geq k$ and for some constant~$C>0$ depending only on~$n$,~$s$,~$g$, and~$\Omega_k$. Note that the first inequality follows from Lemma~\ref{Adomainlem}, while the second one from the non-negativity of~$\Nl^{M_k}(u_h, \Omega_k)$---this is an immediate consequence of definition~\eqref{NMldef} and the fact that~$|u_h| \le M_k$ a.e.~in~$\Omega_k$. Now let~$\tilde{u}_h\in\W^s_{u_h}(\Omega_k)$ be defined by~$\tilde{u}_h|_{\Omega_k}:=u_0$ and~$\tilde{u}_h|_{\Co\Omega_k}:=u_h$. By the minimality of~$u_h$ and Lemma~\ref{tartariccio}, we have that
	\begin{equation*}
	\F^{M_k}(u_h,\Omega_k)-\F^{M_k}(\tilde{u}_h,\Omega_k)=\F^{M_h}(u_h,\Omega_h)-\F^{M_h}(\tilde{u}_h,\Omega_h)\leq0,
	\end{equation*}
	for every $h\geq k$. Next, using~\eqref{eqBlabla} to estimate the term~$\Nl^{M_k}(\tilde{u}_h,\Omega_k)$, we have the bound
	\begin{equation*}
	\F^{M_k}(\tilde{u}_h,\Omega_k)=\A(u_0,\Omega_k)+\Nl^{M_k}(\tilde{u}_h,\Omega_k)\leq
	\A(u_0,\Omega_k)+4\Lambda M_k\Per_s(\Omega_k,\R^n)=:c_k.
	\end{equation*}			
	Combining these three estimates, we obtain
	\[
	[u_h]_{W^{s,1}(\Omega_k)}\leq C(c_k+1),
	\]
	for every~$h\geq k$. Since we also have that~$\|u_h\|_{L^\infty(\Omega_k)}\leq M_k$, by the compact embedding of~$W^{s,1}(\Omega_k)$ into~$L^1(\Omega_k)$ and a diagonal argument we conclude that the claim holds true.
	
	We are left to show the local minimality of~$u$. First, we prove that, given any open set~$\Op \Subset \Omega$ with Lipschitz boundary,~$u$ is the unique minimizer of~$\F$ among all functions~$v$ in~$\W^s_u(\Op)$ such that~$\underline{u}\leq v\leq\overline{u}$ a.e.~in~$\Op$. To see this, let~$\ell$ be the smallest integer for which~$\Op\Subset\Omega_\ell$. For~$h\geq \ell$, define then~$v_h$ by setting~$v_h|_\Op:=v$ and~$v_h|_{\Co\Op}:=u_h$. Observe that~$v_h$ is a competitor for~$u_h$. Hence,~$\F^{M_\ell}(u_h,\Op) \le \F^{M_\ell}(v_h,\Op)$ and, by Lemma~\ref{semicontlem},
	\[
	\F^{M_\ell}(u,\Op)\leq\liminf_{h\to\infty}\F^{M_\ell}(u_h,\Op)\leq\liminf_{h\to\infty}\F^{M_\ell}(v_h,\Op)=\A(v,\Op)+\liminf_{h\to\infty}\Nl^{M_\ell}(v_h,\Op).
	\]
	Note that, by the global Lipschitzianity of~$\G$, for every~$(x,y)\in\Op\times\Co\Op$ we have
	\bgs{
		\left| 2 \, \G \! \left( \frac{v(x)-u_h(y)}{|x-y|} \right) -
		\G \! \left( \frac{M_\ell+u_h(y)}{|x-y|} \right) - \G \! \left( \frac{M_\ell-u_h(y)}{|x-y|} \right) \right| \frac{1}{\kers} & \leq \Lambda\frac{|v(x)|+M_\ell}{|x-y|^{n+s}} \\
		& \leq 2\Lambda\frac{M_\ell}{|x-y|^{n+s}}.
	}
	As the function on the last line is integrable in~$\Op\times\Co\Op$, we may apply Lebesgue's dominated convergence theorem and deduce, using representation~\eqref{nonlocal_explicit}, that
	\[
	\lim_{h\to\infty}\Nl^{M_\ell}(v_h,\Op)=\Nl^{M_\ell}(v,\Op).
	\]
	Accordingly,~$\F^{M_\ell}(u,\Op) \le \F^{M_\ell}(v,\Op)$. This shows that~$u$ has the desired minimality properties. Its uniqueness is an immediate consequence of the strict convexity of~$\F^{M_\ell}$.
	
	Finally, we prove that~$u$ is a true local minimizer of~$\F$---i.e., not only with respect to competitors constrained between~$\underline{u}$ and~$\overline{u}$. Let~$\Theta=\Theta(n,s)$ be the positive constant of Theorem~\ref{Dirichlet}. For~$x\in\Omega$ fixed, let~$\varrho_x:=\frac{\dist(x,\partial\Omega)}{2\Theta+2}$. Since~$u\in\W^s_\loc(\Omega)$, by Lemma~\ref{tail_equiv_cond_Lemma} we have that~$\Tail_s(u,B_{(1+2\Theta)\varrho_x}(x)\setminus B_{\varrho_x}(x);\,\cdot\,)\in L^1(B_{\varrho_x}(x))$. Hence, Theorem~\ref{Dirichlet} ensures the existence of a unique minimizer~$\tilde{u}$ of~$\F$ within $\W^s_u(B_{\varrho_x}(x))$. Since, by the comparison principle of Proposition~\ref{Compari_prop}, we have that~$\underline{u}\leq \tilde{u}\leq\overline{u}$ a.e.~in~$\R^n$, by the uniqueness of~$u$ we conclude that~$u=\tilde{u}$ a.e.~in~$\R^n$. In particular,~$u$ is a weak solution of~$\h u=0$ in~$B_{\varrho_x}(x)$. By the arbitrariness of~$x\in\Omega$ and a partition of unity argument, we find that~$u$ is a weak solution of~$\h u=0$ in the whole~$\Omega$. In light of Corollary~\ref{weak_implies_min_lemma},~$u$ is then a local minimizer of~$\F$ in~$\Omega$.
	\end{proof}

\section{Minimizers of~$\F_s$ versus minimizers of~$\Per_s$. Proof of Theorem~\ref{Persdecreases}} \label{Rearrange_Sect}

\noindent
We continue here the analysis, started in Subsection~\ref{areageom}, of the geometric properties enjoyed by the functional~$\F_s^M$ and of its relation with the~$s$-perimeter. In particular, we prove Theorem~\ref{Persdecreases}, i.e., we show that the nonlocal perimeter decreases under the vertical rearrangement~\eqref{rearr_func_def}. This fact will be a consequence of a rearrangement inequality for a rather general class of~$1$-dimensional integral set functions, that we establish in the next subsection.

\subsection{A one-dimensional rearrangement inequality}\label{Onedim_rearr_Sect}

Let~$K: \R \to \R$ be a non-negative function. Given two measurable sets~$A, B \subseteq \R$, we define
\begin{equation} \label{IKdef}
\I_K(A, B) := \int_A \int_B d\mu, \quad \mbox{where } \, d\mu = d\mu_K(x, y) := K(x - y) \, dx \, dy,
\end{equation}
whenever this quantity is finite.

Fix two real numbers~$\alpha, \beta$ and consider two sets~$A, B \subseteq \R$ satisfying
$$
(-\infty, \alpha) \subseteq A \quad \mbox{and} \quad (\beta, +\infty) \subseteq B.
$$
We define the~\emph{decreasing rearrangement}~$A_*$ of~$A$ as
\begin{equation} \label{decrrearr}
A_* := (-\infty, a_* ), \quad \mbox{with } \, a_* := \lim_{R \rightarrow +\infty} \left( \int_{-R}^{R} \chi_A(t) \, dt - R \right).
\end{equation}
Similarly, the~\emph{increasing rearrangement}~$B^*$ of~$B$ is given by
\begin{equation} \label{incrrearr}
B^* := (b^*, +\infty), \quad \mbox{with } \, b^* := \lim_{R \rightarrow +\infty} \left( R - \int_{-R}^R \chi_B(t) \, dt \right).
\end{equation}
Notice that, up to a set of vanishing measure---actually, a point---it holds
\begin{equation} \label{decrincr}
B^* = \Co (\Co B)_*.
\end{equation}

The next result shows that the value of~$\I_K$ decreases when its arguments are appropriately rearranged.

\begin{prop} \label{iotadecreasesprop}
Let~$A, B \subseteq \R$ be two measurable sets satisfying
$$
(-\infty, \ubar{\alpha}] \subseteq A^\circ \subseteq \overline{A} \subseteq (-\infty, \bar{\alpha}) \quad \mbox{and} \quad [\bar{\beta}, +\infty) \subseteq B^\circ \subseteq \overline{B} \subseteq (\ubar{\beta}, +\infty),
$$
for some real numbers~$\ubar{\alpha} < \bar{\alpha}$ and~$\ubar{\beta} < \bar{\beta}$. Let~$K: \R \to \R$ be a measurable, non-negative function and suppose that
\begin{equation} \label{Kerint}
\I_K \! \left( (-\infty, \bar{\alpha}), (\ubar{\beta}, +\infty) \right) < \infty.
\end{equation}
Then,
\begin{equation} \label{iotadecreases}
\I_K(A_*, B^*) \le \I_K(A, B).
\end{equation}
In addition, if~$K$ is locally bounded from below by positive constants and~$A = A_*$ (or~$B = B^*$) up to sets of measure zero, then the inequality in~\eqref{iotadecreases} is strict unless also~$B = B^*$ (or~$A = A_*$) up to sets of measure zero.
\end{prop}

We strongly believe that, more generally, the strict inequality is valid in~\eqref{iotadecreases} for all couples of sets~$A$ and~$B$ for which~\emph{at least} one of the two does not coincide with its rearrangement. However, we will not investigate the validity of this stronger statement, as it would not play any role for our applications to the~$s$-perimeter.

\begin{proof}[Proof of Proposition~\ref{iotadecreasesprop}]
First of all, we observe that we can restrict ourselves to assume that~$A$ and~$B$ are both open sets. Indeed, if~$A$ and~$B$ are merely measurable, by the outer regularity of the Lebesgue measure there exist two sequences of open sets~$\{ A_k \}, \, \{ B_k \}$ with~$A \subseteq A_k \subseteq (-\infty, \bar{\alpha})$ and~$B \subseteq B_k \subseteq (\ubar{\beta}, +\infty)$ for every~$k \in \N$, and such that~$|A_k \setminus A|, |B_k \setminus B| \rightarrow 0$ as~$k \rightarrow +\infty$. Suppose now that~\eqref{iotadecreases} holds with~$A_k$ and~$B_k$ respectively in place of~$A$ and~$B$. By this and the fact that, by definitions~\eqref{decrrearr}-\eqref{incrrearr}, it clearly holds~$A_* \subseteq (A_k)_*$ and~$B^* \subseteq (B_k)^*$ for any~$k$, we deduce that
$$
\I_K(A_*, B^*) \le \lim_{k \rightarrow +\infty} \I_K((A_k)_*, (B_k)^*) \le \lim_{k \rightarrow +\infty} \I_K(A_k, B_k) = \I_K(A, B).
$$
The last identity follows from Lebesgue's dominated convergence theorem, which can be used thanks to~\eqref{Kerint}. In light of this, it suffices to prove~\eqref{iotadecreases} when~$A$ and~$B$ are open sets.

Next, we recall that each open subset of the real line can be written as the union of countably many disjoint open intervals. In our setting, we have
$$
A = \bigcup_{k = 0}^{+\infty} A^{(k)}, \mbox{ with } \, A^{(k)} := \bigcup_{i = 0}^k A_i,
$$
and
$$
B = \bigcup_{k = 0}^{+\infty} B^{(k)}, \mbox{ with } \, B^{(k)} := \bigcup_{j = 0}^k B_j,
$$
for two sequences~$\{ A_i \}$,~$\{ B_j \}$ of open intervals satisfying~$A_{i_1} \cap A_{i_2} = \varnothing$ for every~$i_1 \ne i_2$ and~$B_{j_1} \cap B_{j_2} = \varnothing$ for every~$j_1 \ne j_2$, and such that~$(-\infty, \ubar{\alpha}) \subseteq A_0$ and~$(\bar{\beta}, +\infty) \subseteq B_0$, Suppose now that~\eqref{iotadecreases} holds when~$A$ and~$B$ are the unions of finitely many disjoint open intervals. In particular,~\eqref{iotadecreases} is true with~$A^{(k)}$ and~$B^{(k)}$ in place of~$A$ and~$B$, respectively. Hence,
\begin{equation} \label{finiteimpliesnum}
\I_K((A^{(k)})_*, (B^{(k)})^*) \le \I_K(A^{(k)}, B^{(k)}) \le \I_K(A, B)
\end{equation}
for every~$k \in \N$. On the other hand, it is easy to see that
$$
(-\infty, \ubar{\alpha}) \subseteq (A^{(k - 1)})_* \subseteq (A^{(k)})_* \subseteq A_* \quad \mbox{and} \quad (\bar{\beta}, +\infty) \subseteq (B^{(k - 1)})^* \subseteq (B^{(k)})^* \subseteq B^*
$$
for every~$k \in \N$. Since both~$|A_* \setminus (A^{(k)})_*|$ and~$|B^* \setminus (B^{(k)})^*|$ go to~$0$ as~$k \rightarrow +\infty$, Lebesgue's monotone convergence theorem yields that
$$
\I_K(A_*, B^*) = \lim_{k \rightarrow +\infty} \I_K((A^{(k)})_*, (B^{(k)})^*). 
$$
The combination of this and~\eqref{finiteimpliesnum} gives~\eqref{iotadecreases}.

In light of the above considerations, we are left to prove~\eqref{iotadecreases} when~$A$ and~$B$ are unions of finitely many disjoint open intervals. Thus, we fix~$M, N \in \N \cup \{ 0 \}$ and assume that
$$
A = \bigcup_{i = 0}^M A_i \quad \mbox{and} \quad B = \bigcup_{j = 0}^{N} B_j,
$$
with
\begin{equation} \label{ABdefs}
\begin{aligned}
A_0 := (-\infty, a_0) \quad \mbox{and} \quad A_i := (a_{2 i - 1}, a_{2 i}) & \quad \mbox{for } i = 1, \ldots, M, \\
B_0 := (b_0, +\infty) \quad \mbox{and} \quad B_j := (b_{2 j}, b_{2 j - 1}) & \quad \mbox{for } j = 1, \ldots, N,
\end{aligned}
\end{equation}
where~$\{ a_i \}_{i = 0}^{2 M}, \{ b_j \}_{j = 0}^{2 N} \subseteq \R$ are two sets of points satisfying~$a_{i - 1} < a_{i}$ and~$b_j < b_{j - 1}$, for every~$i = 1, \ldots, 2 M$ and~$j = 1, \ldots, 2 N$. In this framework, inequality~\eqref{iotadecreases} takes the form
\begin{equation} \label{ID2tech0}
\sum_{\substack{i = 0, \ldots, M \\ j = 0, \ldots, N}} \int_{A_i} \int_{B_j} d\mu \ge \int_{A_*} \int_{B^*} d\mu.
\end{equation}

Clearly, when~$M = N = 0$ there is nothing to prove, as it holds~$A_* = A$ and~$B^* = B$. In case either~$M = 0$ or~$N = 0$, the verification of~\eqref{ID2tech0} is also simple. Indeed, suppose for instance that~$N = 0$ and~$M \ge 1$. Then,~$B^* = B = (b_0, +\infty)$ and~$A_* = (-\infty, a_*)$ for some~$a_* \in \R$. Up to a set of measure zero we may write~$A_*$ as the union of~$M + 1$ disjoint adjacent intervals~$\{ C_i \}_{i = 0}^M$ given by~$C_i = A_i - \bar{a}_i$, for some~$\bar{a}_i \ge 0$ and for every~$i$. Accordingly,
\begin{equation} \label{A*B*leABN=0M=1}
\int_{A_*} \int_{B^*} d\mu = \sum_{i = 0}^M \int_{C_i} \int_{b_0}^{+\infty} d\mu = \sum_{i = 0}^M \int_{A_i} \int_{b_0 + \bar{a}_i}^{+\infty} d\mu \le \sum_{i = 0}^M \int_{A_i} \int_{b_0}^{+\infty} d\mu = \int_A \int_B d\mu,
\end{equation}
that is~\eqref{ID2tech0}. Note that the second identity follows by adding to both variables of the double integral the same quantity~$\bar{a}_i$. That is, we applied the change of coordinates~$x = w - \bar{a}_i$,~$y = z - \bar{a}_i$ and got
$$
\int_{C_i} \int_{b_0}^{+\infty} d\mu = \int_{C_i} \int_{b_0}^{+\infty} K(x - y) \, dx \, dy = \int_{A_i} \int_{b_0 + \bar{a}_i}^{+\infty} K(w - z) \, dw \, dz = \int_{A_i} \int_{b_0 + \bar{a}_i}^{+\infty} d\mu.
$$

As the case~$M = 0$,~$N \ge 1$ is completely analogous, we can now address the validity of~\eqref{ID2tech0} when~$M, N \ge 1$. Recalling definitions~\eqref{decrrearr}-\eqref{incrrearr}, it is immediate to see that
$$
A_* = \left( -\infty, a_* \right), \quad \mbox{with } \, a_* = a_0 + \sum_{\ell = 1}^M |A_\ell| = a_0 + \sum_{\ell = 1}^M (a_{2 \ell} - a_{2 \ell - 1})
$$
and
$$
B^* = \left( b^*, +\infty \right), \quad \mbox{with } \, b^* = b_0 - \sum_{\ell = 1}^N |B_j| = b_0 - \sum_{\ell = 1}^N (b_{2 \ell - 1} - b_{2 \ell}).
$$
Set
\begin{align}
\label{Cidef} C_i & := A_i - \bar{a}_i, \quad \mbox{with } \, \bar{a}_i := \sum_{\ell = 0}^{i - 1} (a_{2 \ell + 1} - a_{2 \ell}) \, \mbox{ for } i = 1, \ldots, M \mbox{ and } \bar{a}_0 := 0, \\
\label{Djdef} D_j & := B_j + \bar{b}_j, \quad \mbox{with } \, \bar{b}_j := \sum_{\ell = 0}^{j - 1} (b_{2 \ell} - b_{2 \ell + 1}) \, \hspace{2pt} \mbox{ for } j = 1, \ldots, N \mbox{ and } \bar{b}_0 := 0.
\end{align}
The families~$\{ C_i \}_{i = 0}^M$ and~$\{ D_j \}_{j = 0}^{N}$ are both made up of consecutive open intervals. Moreover, up to sets of measure zero, we have
\begin{equation} \label{CDfill}
A_* = \bigcup_{i = 0}^M C_i \quad \mbox{and} \quad B^* = \bigcup_{j = 0}^{N} D_j.
\end{equation}
Consequently, we can equivalently express~\eqref{ID2tech0} as
\begin{equation} \label{ID2tech1}
\sum_{\substack{ i = 0, \ldots, M \\ j = 0, \ldots, N}} \int_{A_i} \int_{B_j} d\mu \ge \sum_{\substack{ i = 0, \ldots, M \\ j = 0, \ldots, N}} \int_{C_i} \int_{D_j} d\mu.
\end{equation}

Let~$j = 1, \ldots, N$ be fixed. We compute
$$
\int_{A_0} \int_{B_j} d\mu = \int_{C_0} \int_{D_j - \bar{b}_j} d\mu = \int_{C_0 + \bar{b}_j} \int_{D_j} d\mu = \int_{\left( C_0 + \bar{b}_j \right) \setminus C_0} \int_{D_j} d\mu + \int_{C_0} \int_{D_j} d\mu.
$$
Notice that the first identity follows from definitions~\eqref{Cidef}-\eqref{Djdef}, the second by applying to both variables of the double integral a shift of length~$\bar{b}_j$, and the third since~$C_0 \subseteq C_0 + \bar{b}_j$. Similarly,
$$
\int_{A_i} \int_{B_0} d\mu = \int_{C_i} \int_{\left( D_0 - \bar{a}_i \right) \setminus D_0} d\mu + \int_{C_i} \int_{D_0} d\mu
$$
for every~$i = 1, \ldots, M$. Furthermore, by a translation of size~$\bar{b}_j - \bar{a}_i$, we may also write
$$
\int_{A_i} \int_{B_j} d\mu = \int_{C_i + \bar{a}_i} \int_{D_j - \bar{b}_j} d\mu = \int_{C_i + \bar{b}_j} \int_{D_j - \bar{a}_i} d\mu
$$
for every~$i = 1, \ldots, M$ and~$j = 1, \ldots, N$. Finally, as~$A_0 = C_0$ and~$B_0 = D_0$, we have
$$
\int_{A_0} \int_{B_0} d\mu = \int_{C_0} \int_{D_0} d\mu.
$$
Applying the last four identities together with~\eqref{CDfill}, formula~\eqref{ID2tech1} becomes
\begin{equation} \label{ID2tech2}
\sum_{\substack{ i = 0, \ldots, M \\ j = 0, \ldots, N}} \int_{E_{i; j}} \int_{F_{j; i}} \, d\mu \ge \int_{a_0}^{a_*} \int_{b^*}^{b_0} d\mu,
\end{equation}
where we put
\begin{equation} \label{EFdef}
\begin{aligned}
E_{0; 0} & := \{ a_0 \}, & \, F_{0; 0} & := \{ b_0 \}, && \\
E_{i; 0} & := C_i, & \, F_{0; i} & := \left( D_0 - \bar{a}_i \right) \setminus D_0, \, && \, \mbox{for } i = 1, \ldots, M, \\
E_{0; j} & := \left( C_0 + \bar{b}_j \right) \setminus C_0, & \, F_{j; 0} & := D_j, \, && \, \mbox{for } j = 1, \ldots, N, \\
E_{i; j} & := C_i + \bar{b}_j, & \, F_{j; i} & := D_j - \bar{a}_i, \, && \, \mbox{for } i = 1, \ldots, M, \, j = 1, \ldots, N.
\end{aligned}
\end{equation}

We now claim that
\begin{equation} \label{ID2claim}
[a_0, a_*] \times [b^*, b_0] \subseteq \bigcup_{\substack{i = 0, \ldots, M \\ j = 0, \ldots, N}} \overline{E_{i; j}} \times \overline{F_{j; i}}.
\end{equation}
Observe that~\eqref{ID2claim} is stronger than~\eqref{ID2tech2}, and therefore that its validity would lead us to the conclusion of the proof.

\begin{figure}[h]
\centering
\includegraphics[width=0.95\textwidth]{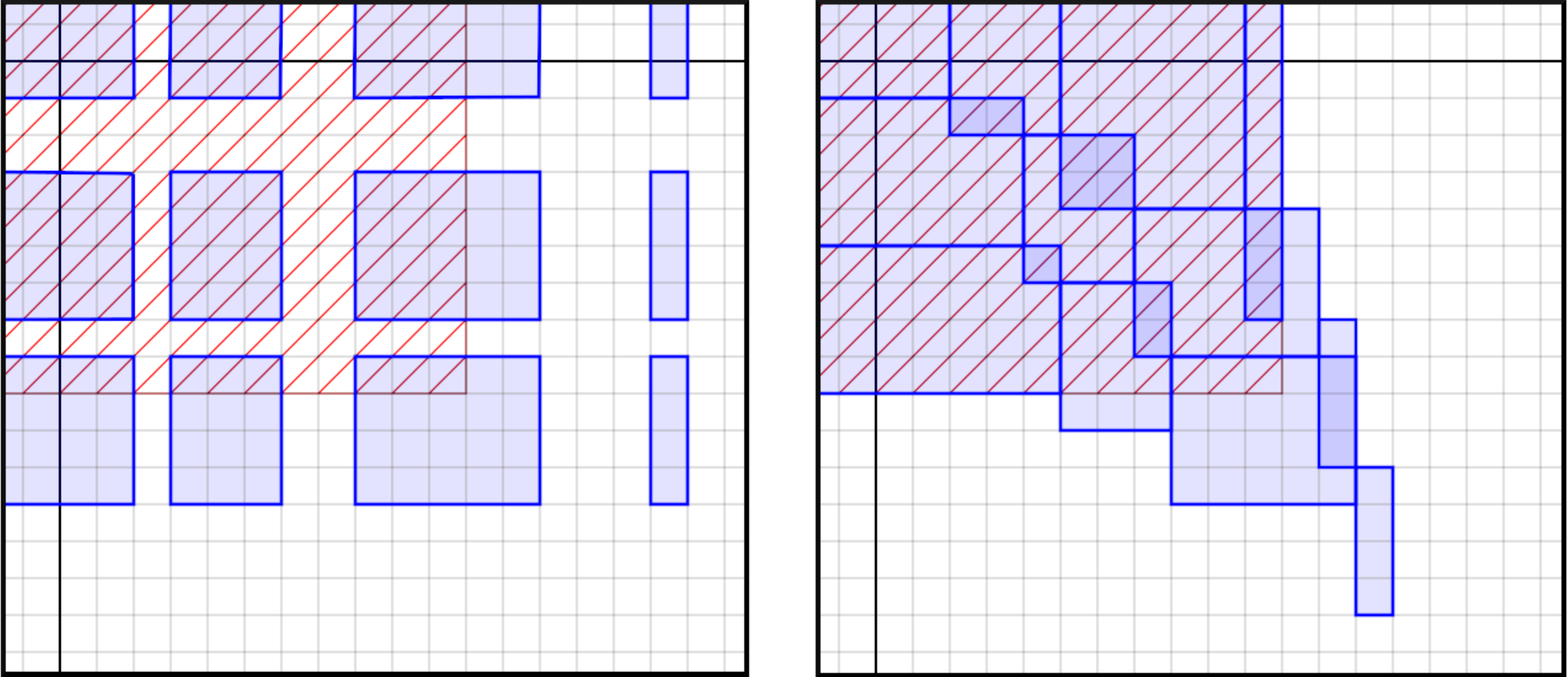}
\caption{An example illustrating the validity of~\eqref{ID2claim}. On the left, we have drawn in solid blue background the rectangles corresponding to the initial configuration given by the intervals~$A_0 = (-\infty, 2)$,~$A_1 = (3, 6)$,~$A_2 = (8, 13)$,~$A_3 = (16, 17)$---on the horizontal axis---and~$B_0 = (-1, +\infty)$,~$B_1 = (- 7, - 3)$,~$B_2 = (-12, -8)$---on the vertical axis. The rearranged set~$A_* \times B^* = (-\infty, 11) \times (-9, +\infty)$ is represented in red diagonal pattern. On the right, we translated the blue rectangles~$A_i \times B_j$ along the direction~$(1, 1)$, following the rules outlined above, to obtain the sets~$E_{i; j} \times F_{i; j}$. The new configuration covers completely the region shaded in red.}
\end{figure}

Before showing that~\eqref{ID2claim} is true, we make some considerations on the intervals~$E_{i; j}$'s and~$F_{j; i}$'s. Given a bounded non-empty interval~$I \subseteq \R$, we indicate with~$\ell(I)$ and~$r(I)$ its left and right endpoint, respectively. We have that
\begin{alignat}{2}
\label{rlE}
r(E_{i - 1; j}) & = \ell(E_{i; j}), && \qquad \mbox{for } i = 1, \ldots, M, \, j = 0, \ldots, N, \\
\label{rlF}
r(F_{j; i}) & = \ell(F_{j - 1; i}), && \qquad \mbox{for } i = 0, \ldots, M, \, j = 1, \ldots, N, \\
\label{Eends}
r(E_{M; j}) & \ge a_*, && \qquad \mbox{for } j = 0, \ldots, N, \\
\label{Fends}
\ell(F_{N; i}) & \le b^*, && \qquad \mbox{for } i = 0, \ldots, M.
\end{alignat}
To check~\eqref{rlE}, we recall definitions~\eqref{EFdef},~\eqref{Cidef},~\eqref{ABdefs}, and notice that
\begin{align*}
r(E_{i - 1; j}) & = r(A_{i - 1}) - \bar{a}_{i - 1} + \bar{b}_j = a_{2 i - 2} - \bar{a}_{i} + (a_{2 i - 1} - a_{2 i - 2}) + \bar{b}_j \\
& = \ell(A_{i}) - \bar{a}_{i} + \bar{b}_j = \ell(E_{i; j})
\end{align*}
for every~$i = 1, \ldots, M$ and~$j = 0, \ldots, N$. On the other hand, it holds
\begin{align*}
r(E_{M; j}) & = r(A_M) - \bar{a}_M + \bar{b}_j = a_{2 M} - \sum_{\ell = 0}^{M - 1} (a_{2 \ell + 1} - a_{2 \ell}) + \bar{b}_j \\
& = a_0 + \sum_{\ell = 1}^{M} (a_{2 \ell} - a_{2 \ell - 1}) + \bar{b}_j \ge a_*,
\end{align*}
which gives~\eqref{Eends}. Items~\eqref{rlF} and~\eqref{Fends} follow analogously.

In view of formulas~\eqref{rlE}-\eqref{Fends}, we immediately deduce that
\begin{equation} \label{Ecover}
[a_0, a_*] \subseteq \bigcup_{i = 0}^M \overline{E_{i ; j}} \quad \mbox{for every } j = 0, \ldots, N
\end{equation}
and
$$
\hspace{3pt} [b^*, b_0] \subseteq \bigcup_{j = 0}^N \overline{F_{j;i}} \quad \mbox{for every } i = 0, \ldots, M.
$$

On top of the previous facts, we also claim that
\begin{equation} \label{ID2claim2pre}
\ell(E_{i; j}) > \ell(E_{i; j - 1}) \quad \mbox{for every } i = 1, \ldots, M, \, j = 1, \ldots, N
\end{equation}
and
\begin{equation} \label{ID2claim2}
r(F_{j; i}) < r(F_{j; i - 1}) \quad \mbox{for every } i = 1, \ldots, M, \, j = 1, \ldots, N.
\end{equation}
Indeed, for~$i = 1, \ldots, M$ and~$j = 1, \ldots, N$ we have
$$
r(F_{j; i}) = r(D_j) - \bar{a}_i = r(D_j) - \bar{a}_{i - 1} - (a_{2 i - 1} - a_{2 i - 2}) < r(D_j) - \bar{a}_{i - 1} = r(F_{j; i - 1}).
$$
This proves~\eqref{ID2claim2}, while~\eqref{ID2claim2pre} can be checked in a similar fashion.

Thanks to the previous remarks, we can now address the proof of~\eqref{ID2claim}. Let
\begin{equation} \label{pdef}
p = (x, y) \in [a_0, a_*] \times [b^*, b_0]
\end{equation}
and suppose by contradiction that~$p$ does not belong to the right-hand side of~\eqref{ID2claim}. I.e.,
\begin{equation} \label{pcontradict}
p \notin \overline{E_{i; j}} \times \overline{F_{j; i}} \quad \mbox{for every } i = 0, \ldots, M \mbox{ and } j = 0, \ldots, N.
\end{equation}
By virtue of~\eqref{Ecover}, in correspondence to every~$j = 0, \ldots, N$ we can pick an~$i_j \in \{ 0, \ldots, M \}$ in such a way that
\begin{equation} \label{xinE}
x \in \overline{E_{i_j; j}}.
\end{equation}
We claim that
\begin{equation} \label{ijnonincr}
\{i_j\}_{j = 0}^N \mbox{ is non-increasing}.
\end{equation}
Indeed, suppose that we have constructed the (finite) sequence~$\{ i_\ell \}$ up to the index~$\ell = j - 1$, with~$j \in \{ 1, \ldots, N \}$. Of course, when~$i_{j - 1} = M$ we necessarily have~$i_{j} \le i_{j - 1}$. On the other hand, if~$i_{j - 1} \le M - 1$, using~\eqref{ID2claim2pre} and~\eqref{rlE}, we infer that
$$
\ell(E_{i_{j - 1} + 1; j}) > \ell(E_{i_{j - 1} + 1; j - 1}) = r(E_{i_{j - 1}; j - 1}) \ge x.
$$
Hence, also in this case~$i_j$ falls within the set~$\{ 0, \ldots, i_{j - 1} \}$ and~\eqref{ijnonincr} is established.

Next, by comparing~\eqref{xinE} and~\eqref{pcontradict}, we notice that~$y \notin \cup_{j = 0}^N \overline{F_{j; i_j}}$. This amounts to say that, for every index~$j = 0, \ldots, N$,
\begin{equation} \label{y<lory>r}
\mbox{ either } \, y < \ell(F_{j; i_j}) \, \mbox{ or } \, y > r(F_{j; i_j}).
\end{equation}
We now claim that the latter possibility cannot occur, i.e.,~that
\begin{equation} \label{ID2claimcontr}
y < \ell(F_{j; i_j})
\end{equation}
for every~$j = 0, \ldots, N$. Note that~\eqref{ID2claimcontr} would lead us to a contradiction. Indeed, by using it with~$j = N$ and in combination with~\eqref{pdef} and~\eqref{Fends}, we would get
$$
b^* \le y < \ell(F_{N; i_N}) \le b^*,
$$
which is clearly impossible. Therefore, to finish the proof we are only left to show that~\eqref{ID2claimcontr} holds true for every~$j = 0, \ldots, N$. To achieve this, we argue inductively. First, we check that~\eqref{ID2claimcontr} is verified for~$j = 0$. Indeed, by~\eqref{pdef} and~\eqref{EFdef},
$$
y \le b_0 = r(F_{0; i_0}),
$$
and thus~\eqref{y<lory>r} yields that~$y < \ell(F_{0; i_0})$---i.e.,~\eqref{ID2claimcontr} for~$j = 0$. Secondly, we pick any~$j \in \{ 1, \ldots, N\}$ and assume that~\eqref{ID2claimcontr} is valid with~$j - 1$ in place of~$j$. Then, recalling~\eqref{rlF},~\eqref{ijnonincr}, and possibly~\eqref{ID2claim2} (applied~$i_{j - 1} - i_j$ times), we get that
$$
y < \ell(F_{j - 1; i_{j - 1}}) = r(F_{j; i_{j - 1}}) \le r(F_{j; i_j}).
$$
By comparing this with~\eqref{y<lory>r}, we finally deduce that claim~\eqref{ID2claimcontr} holds true. Thus, the proof of~\eqref{iotadecreases} is complete.

We now assume the kernel~$K$ to satisfy~$\inf_I K > 0$ for every compact set~$I \subseteq \R$ and show that inequality~\eqref{iotadecreases} is strict when one between~$A$ and~$B$ coincides with its rearrangement and the other does not. Without loss of generality, we suppose that~$B = B^* = (b_0, +\infty)$ and~$d := |A \Delta A_*| > 0$. We claim that
\begin{equation} \label{strictclaim}
\delta(A, B) := \I_K(A, B) - \I_K(A_*, B^*) \ge c \, d^2,
\end{equation}
for some constant~$c > 0$ depending only on~$\underline{\alpha}$,~$\overline{\alpha}$,~$\underline{\beta}$,~$\overline{\beta}$, and~$K$.

Thanks to the same approximation procedure considered in the first part of the proof, it suffices to establish~\eqref{strictclaim} in the case when~$A$ can be written as
$$
A = \bigcup_{i = 0}^{M} A_i, \mbox{ with } A_0 = (- \infty, a_0) \mbox{ and } A_i = (a_{2i - 1}, a_{2 i}) \mbox{ for } i = 1, \ldots, M,
$$
for some~$M \in \N$ and~$\underline{\alpha} \le a_{i - 1} < a_i \le \overline{\alpha}$ for every~$i = 1, \ldots, 2 M$. Recalling~\eqref{A*B*leABN=0M=1}, we have that
\begin{equation} \label{stricttech1}
\delta(A, B) = \sum_{i = 1}^{M} \int_{a_{2 i - 1}}^{a_{2 i}} \int_{b_0}^{b_0 + \bar{a}_i} K(x - y) \, dx \, dy \ge \left( \inf_{\left[ \underline{\alpha} - \overline{\beta}, \, \overline{\alpha} - \underline{\beta} \right]} K \right) \sum_{i = 1}^{M} (a_{2 i} - a_{2i - 1}) \, \bar{a}_i,
\end{equation}
with~$\bar{a}_i$ as in~\eqref{Cidef}.

Consider the index~$m := \min \big\{ i \in \{ 1, \ldots, M \} : a_{2 i} \ge a_* \big\}$. A straightforward computation gives
$$
|A\setminus A_*| = \sum_{i = m}^{M} \big( a_{2 i} - a_{2 i - 1} \big) - \big( a_* - a_{2 m - 1} \big)_+  = \sum_{\ell = 0}^{m - 1} \big( a_{2 \ell + 1} - a_{2 \ell} \big) - \big( a_* - a_{2 m - 1} \big)_- = |A_* \setminus A|.
$$
Consequently, all these quantities are equal to~$d/2$ and therefore
\begin{align*}
\sum_{i = 1}^{M} (a_{2 i} - a_{2i - 1}) \, \bar{a}_i & = \sum_{i = 1}^{M} (a_{2 i} - a_{2i - 1}) \sum_{\ell = 0}^{i - 1} (a_{2 \ell + 1} - a_{2 \ell}) = \sum_{\ell = 0}^{M - 1} (a_{2 \ell + 1} - a_{2 \ell}) \sum_{i = \ell + 1}^{M} (a_{2 i} - a_{2i - 1}) \\
& \ge \sum_{\ell = 0}^{m - 1} (a_{2 \ell + 1} - a_{2 \ell}) \sum_{i = m}^{M} (a_{2 i} - a_{2i - 1}) \ge |A_* \setminus A| |A\setminus A_*| = \frac{d^2}{4}.
\end{align*}
Claim~\eqref{strictclaim} follows from this and~\eqref{stricttech1}. This concludes the proof of Proposition~\ref{iotadecreasesprop}.
\end{proof}

\subsection{Vertical rearrangements and the~$s$-perimeter} \label{rearrangsub}

We now take advantage of Proposition~\ref{iotadecreasesprop} to show that~$\Per_s$ decreases under vertical rearrangements---that is, we prove Theorem~\ref{Persdecreases}.

\begin{proof}[Proof of Theorem \ref{Persdecreases}]
Given~$\varepsilon > 0$, consider the nonlocal perimeter~$\Per_s^\varepsilon$ corresponding to the truncated kernel~$K^\varepsilon(r) := \max \{ r, \varepsilon \}^{- n - 1 - s}$ for~$r = |Z|$ and~$Z \in \R^{n + 1}$. That is, given a measurable~$F \subseteq \R^{n + 1}$ and~$\Op \subseteq \R^{n + 1}$ open, we define
$$
\Per_s^\varepsilon(F, \Op) := \Ll_s^\varepsilon(F\cap\Op,\Co F\cap\Op) + \Ll_s^\varepsilon(F \cap\Op,\Co F \setminus\Op) + \Ll_s^\varepsilon(F \setminus\Op,\Co F \cap\Op),
$$
where, for any two measurable sets~$A, B \subseteq \R^{n + 1}$,
$$
\Ll_s^\varepsilon(A,B):=\int_A\int_B K^\varepsilon(|X - Y|) \, dX\,dY = \int_A\int_B \frac{dX\,dY}{\max \{ |X - Y|, \varepsilon \}^{n + 1 + s}}.
$$

Let~$E$ be as in the statement of the theorem,~$w_E$ be the function introduced in~\eqref{rearr_func_def}, and write~$E_\star := \S_{w_E}$ for the vertical rearrangement of~$E$. Denote with~$F$ either the set~$E$ or its rearrangement~$E_\star$. Observe that, outside of~$\Omega^\infty$, both sets~$E$ and~$E_\star$ coincide with the subgraph of the same function~$v: \Co \Omega \to \R$. Hence,
\begin{equation} \label{GoutOmega}
F \setminus \Omega^\infty = \Big\{ (x, t) \in \left( \Co \Omega \right) \times \R : t < v(x) \Big\}.
\end{equation}
It is also clear that~$E_\star$ satisfies~\eqref{EboundedinOmega}. Accordingly,
\begin{equation} \label{GboundedinOmega}
\Omega \times (-\infty, -M) \subseteq F \cap \Omega^\infty \subseteq \Omega \times (-\infty, M).
\end{equation}

We compute
\begin{align*}
\Per_s^\varepsilon(F, \Omega^M)
& = \Ll_s^\varepsilon(F \cap \Omega^M,\Co F \cap \Omega^M) \\
& \quad + \Ll_s^\varepsilon(F \cap \Omega^M, \Co F \cap (\Omega^\infty \setminus \Omega^M)) + \Ll_s^\varepsilon(F \cap \Omega^M, \Co F \setminus \Omega^\infty) \\
& \quad + \Ll_s^\varepsilon(F \cap (\Omega^\infty \setminus \Omega^M), \Co F \cap \Omega^M) + \Ll_s^\varepsilon(F \setminus \Omega^\infty, \Co F \cap \Omega^M) \\
& = \Ll_s^\varepsilon(F \cap \Omega^\infty,\Co F \cap \Omega^\infty) - \Ll_s^\varepsilon(F \cap (\Omega^\infty \setminus \Omega^M), \Co F \cap (\Omega^\infty \setminus \Omega^M)) \\
& \quad + \Ll_s^\varepsilon(F \cap \Omega^M, \Co F \setminus \Omega^\infty) + \Ll_s^\varepsilon(F \setminus \Omega^\infty, \Co F \cap \Omega^M).
\end{align*}
Observe that all the above terms are finite, thanks to the boundedness of both~$\Omega$ and~$K_\varepsilon$, the decay of~$K_\varepsilon$ at infinity, and property~\eqref{GboundedinOmega}. By this identity and again~\eqref{GboundedinOmega},
\begin{equation} \label{PersEstar-PersE}
\begin{aligned}
\Per_s^\varepsilon(E_\star, \Omega^M) - \Per_s^\varepsilon(E, \Omega^M) & = \Ll_s^\varepsilon(E_\star \cap \Omega^\infty,\Co E_\star \cap \Omega^\infty) - \Ll_s^\varepsilon(E \cap \Omega^\infty,\Co E \cap \Omega^\infty) \\
& \quad + \Ll_s^\varepsilon(E_\star \cap \Omega^M, \Co E_\star \setminus \Omega^\infty) - \Ll_s^\varepsilon(E \cap \Omega^M, \Co E \setminus \Omega^\infty) \\
& \quad + \Ll_s^\varepsilon(E_\star \setminus \Omega^\infty, \Co E_\star \cap \Omega^M) - \Ll_s^\varepsilon(E \setminus \Omega^\infty, \Co E \cap \Omega^M).
\end{aligned}
\end{equation}
Set
$$
F(x) := \Big\{ t \in \R : (x, t) \in F \Big\} \quad \mbox{for } x \in \R^n
$$
and
$$
K_{a}^\varepsilon(t) := K^\varepsilon \Big( \! \sqrt{a^2 + t^2} \, \Big) \quad \mbox{for } a, t \in \R.
$$
Using the notation of~\eqref{IKdef}, by~\eqref{GoutOmega},~\eqref{GboundedinOmega}, and Fubini's theorem, identity~\eqref{PersEstar-PersE} becomes
\begin{align*}
\Per_s^\varepsilon(E_\star, \Omega^M) - \Per_s^\varepsilon(E, \Omega^M) & = \int_{\Omega} \int_{\Omega} \left\{ \I_{K_{|x - y|}^\varepsilon} \! \left( E_\star(x), \Co E_\star(y) \right) - \I_{K_{|x - y|}^\varepsilon} \! \left( E(x), \Co E(y) \right) \right\} dx \, dy \\
& \quad + \int_{\Omega} \int_{\C \Omega} \left\{ \I_{K_{|x - y|}^\varepsilon} \! \left( E_\star(x), \Co E_\star(y) \right) - \I_{K_{|x - y|}^\varepsilon} \! \left( E(x), \Co E(y) \right)  \right\} dx \, dy \\
& \quad + \int_{\C \Omega} \int_{\Omega} \left\{ \I_{K_{|x - y|}^\varepsilon} \! \left( E_\star(x), \Co E_\star(y) \right) - \I_{K_{|x - y|}^\varepsilon} \! \left( E(x), \Co E(y) \right)  \right\} dx \, dy.
\end{align*}
Recalling the definition of decreasing rearrangement of a subset of the real line introduced in~\eqref{decrrearr}, we observe that~$E(x)_* = (-\infty, w_E(x)) = E_\star(x)$ for all~$x \in \R^n$. Also note that, for every~$\alpha, \beta \in \R$, we have~$\I_{K_a^\varepsilon} \! \left( (-\infty, \alpha), (\beta, +\infty) \right) < \infty$. Hence, we can apply Proposition~\ref{iotadecreasesprop} and deduce that
$$
\I_{K_{|x - y|}^\varepsilon} \! \left( E_\star(x), \Co E_\star(y) \right) - \I_{K_{|x - y|}^\varepsilon} \! \left( E(x), \Co E(y) \right) \le 0 \quad \mbox{for a.e.~} x, y \in \R^n,
$$
where we also took advantage of property~\eqref{decrincr}. Using this inequality in the previous identity, we get that~$\Per_s^\varepsilon(E_\star, \Omega^M) \le \Per_s^\varepsilon(E, \Omega^M)$. Letting~$\varepsilon \searrow 0$, we conclude that~\eqref{Persdecreases} holds true.

To finish the proof, we are left to show that, if~$\Per_s(E, \Omega^M) < \infty$, then the inequality in~\eqref{Persdecreases} is strict unless~$E = E_\star$ up to a negligible set. Indeed, suppose that~\eqref{Persdecreases} holds as an identity. By letting~$\varepsilon \searrow 0$ in the last two formulas, it is easy to see that
$$
\I_{K_{|x - y|}} \! \left( E_\star(x), \Co E_\star(y) \right) - \I_{K_{|x - y|}} \! \left( E(x), \Co E(y) \right) = 0 \quad \mbox{for a.e.~} x, y \in \R^n,
$$
where~$K_a(t) := (a^2 + t^2)^{- \frac{n + 1 + s}{2}}$. But then, since~$K_a$ is positive and continuous for every~$a > 0$, the second part of the statement of Proposition~\ref{iotadecreasesprop} yields that~$|E(x) \Delta E_\star(x)| = 0$ for a.e.~$x \in \Omega$---note that we also exploited the fact that~$\Co E(y) = \Co E_\star(y)$ for a.e.~$y \in \Co \Omega$, thanks to~\eqref{GoutOmega}. From this, it follows that~$|E \Delta E_\star| = 0$. The proof of Theorem~\ref{Persdecreases} is thus complete.
\end{proof}

\section{Proof of Theorem~\ref{equiv_intro}}\label{Equiv_proof_Sect}

\noindent
We begin by showing the equivalence of~\ref{equiv:uvisc}-\ref{equiv:upwise}, assuming~$\Omega$ to be merely an open set.

Implication~\ref{equiv:uvisc}~$\Rightarrow$~\ref{equiv:ulocweak} is an immediate consequence of the first part of Theorem~\ref{Gen_viscweak}.

Next,~\ref{equiv:ulocweak}~$\Rightarrow$~\ref{equiv:ulocmin} can be easily deduced from Corollary~\ref{weak_implies_min_lemma}.

As for~\ref{equiv:ulocmin}~$\Rightarrow$~\ref{equiv:Sgumin}, by Proposition~\ref{Linftylocprop} we know that~$u\in L^\infty_{\loc}(\Omega)$. Let~$\{ \Omega_k \}$ be a sequence of open subsets of~$\Omega$ with Lipschitz boundary, such that~$\Omega_k \Subset \Omega_{k+1}$ for all~$k$ and~$\bigcup_{k\in\mathbb N}\Omega_k=\Omega$. Let~$\{ M_k \}$ be a diverging sequence for which
\eqlab{\label{cylind_approx}
M_k > \|u\|_{L^\infty(\Omega_k)},
}
and consider the cylinders~$\Op^k:=\Omega_k\times(-M_k,M_k)$. We claim that~$\S_u$ is~$s$-minimal in each~$\Op^k$. Since~$\Op^k\nearrow\Omega^\infty$, this would readily give that~$\S_u$ is locally $s$-minimal in~$\Omega^\infty$, as desired.

Let~$E\subseteq\R^{n+1}$ be such that~$E\setminus\Op^k=\S_u\setminus\Op^k$ and let~$w_E$ be the function defined in~\eqref{rearr_func_def}. We can suppose that~$\Per_s(E,\Op^k)<\infty$, otherwise there is nothing to prove. By~\eqref{cylind_approx}, we know that~$E$ satisfies~\eqref{EboundedinOmega} and hence Theorem~\ref{Persdecreases} yields that
\eqlab{\label{equiv_proof_eqn1}
	\Per_s(\S_{w_E},\Op^k)\le\Per_s(E,\Op^k).
}
Notice that, thanks to Proposition~\ref{per_of_subgraph_prop}, we know that~$w_E\in\B_{M_k}\W^s_u(\Omega_k)$---recall the terminology introduced at the beginning of Section~\ref{Linftysec}.
By this, the fact that~$u\in\B_{M_k}\W^s(\Omega_k)$ minimizes~$\F_s$ in~$\Omega_k$, identity~\eqref{per_of_subgraph} (with~$\Omega = \Omega_k$ and~$M = M_k$), and~\eqref{equiv_proof_eqn1},
we get
\bgs{
	\Per_s(\S_u,\Op^k)\le\Per_s(\S_{w_E},\Op^k)\le\Per_s(E,\Op^k).
}
The arbitrariness of the set~$E$ implies that~$\S_u$ is $s$-minimal in~$\Op^k$, as claimed.

We now prove that~\ref{equiv:Sgumin}~$\Rightarrow$~\ref{equiv:upwise}. First, by Proposition~\ref{LinftyforPers} we have that~$u \in L^\infty_{\loc}(\Omega)$. Then,~\cite[Theorem~1.1]{CaCo} actually yields that~$u\in C^\infty(\Omega)$. Hence, given any~$x\in\Omega$, we can find both an interior and an exterior tangent ball to~$\S_u$ at~$(x,u(x))\in\partial\S_u\cap\Omega^\infty$. The Euler-Lagrange equation satisfied by $s$-minimal sets---see~\cite[Theorem~5.1]{CRS10}---and~\eqref{Curv_subgr_funct_u_id} then imply that~$\h_s u(x)=H_s[\S_u](x,u(x))=0$.

Finally, implication~\ref{equiv:upwise}~$\Rightarrow$~\ref{equiv:uvisc} holds thanks to Definition~\ref{visc_sol_def} and Remark~\ref{mah}.

Assume now~$\Omega$ to be a bounded open set with Lipschitz boundary. Under this assumption, Corollary~\ref{weak_implies_min_lemma} ensures that~\ref{equiv:uglobweak}~$\Leftrightarrow$~\ref{equiv:uglobmin}. Also,~\ref{equiv:uglobmin}~$\Rightarrow$~\ref{equiv:ulocmin} is always trivially verified. To conclude, observe for instance that, when~$u \in L^\infty(\Omega)$, the implication~\ref{equiv:uvisc}~$\Rightarrow$~\ref{equiv:uglobweak} easily follows from Theorem~\ref{Gen_viscweak}.

\section{Proof of Theorem~\ref{last_Theorem}}\label{Unif_Cont_Sect}

\noindent
In this brief section, we establish the validity of Theorem~\ref{last_Theorem}. This will be a consequence of the next two propositions.

First, we address the existence and uniqueness of~$s$-minimal graphs. Before heading to our statement, we make the following observation.

Let~$\Omega$ be a bounded open set with~$C^2$ boundary and~$\varphi: \R^n \to \R$ be a measurable function, bounded in~$B_R \setminus \Omega$ for some~$R > 0$ and such that~$\varphi = 0$ a.e.~in~$\Co \Omega$. In~\cite{graph} it is proved that there exists a radius~$\tilde{R} > 0$, depending only on~$n$,~$s$, and~$\Omega$, such that if~$R \ge \tilde{R}$ and~$E$ is a locally $s$-minimal set in~$\Omega^\infty$ such that~$E \setminus \Omega^\infty = \S_{\varphi} \setminus \Omega^\infty$, then
\begin{equation} \label{Eisbounded}
	\Omega \times (-\infty, -M_0) \subseteq E \cap \Omega^\infty \subseteq \Omega \times (-\infty, M_0),
\end{equation}
with~$M_0 = C \left( R + \| \varphi \|_{L^\infty(B_R)} \right)$ for some numerical constant~$C > 0$. Roughly speaking, this is a global~``$L^\infty$ estimate'' for nonlocal minimal surfaces (not necessarily graphs) in terms of their (graphical) exterior data, and can be thought of as a geometric counterpart of our
Theorem~\ref{minareboundedthm}. Its validity can be inferred from a careful inspection of the proof of~\cite[Lemma~3.2]{graph}.

With this in hand, we can easily establish the following result.

\begin{prop}
Let~$\Omega\subseteq\R^n$ be an open set with boundary of class~$C^2$ and such that~$\Omega \subseteq B_{R_0}$ for some~$R_0 > 0$. There exists a radius~$R > R_0$, depending only on~$n$,~$s$, and~$\Omega$, such that the following holds true. If~$\varphi:\Co \Omega \to\R$ is a measurable function, bounded in~$B_R \setminus \Omega$, then there exists a unique locally~$s$-minimal set~$E$ in~$\Omega^\infty$ which coincides with the subgraph of~$\varphi$ outside of~$\Omega^\infty$. The set~$E$ is the subgraph~$\S_u$ of a function~$u: \R^n \to \R$ with~$u|_\Omega \in L^\infty(\Omega) \cap C^\infty(\Omega)$. 
\end{prop}
\begin{proof}
Let~$R$ be larger than the radius~$\tilde{R}$ considered earlier and such that~$\Omega_{\Theta \diam(\Omega)} \subseteq B_R$, with~$\Theta$ being the maximum between the two constants found in Theorems~\ref{Dirichlet} and~\ref{minareboundedthm}. Note that, thanks to Lemma~\ref{tail_equiv_cond_Lemma}\ref{tail_equiv_ii}, we know that condition~\eqref{TailinL1} holds true. Consequently, Theorem~\ref{Dirichlet} yields the existence of a unique minimizer~$u$ of~$\F_s$ in~$\Omega$ such that~$u = \varphi$ a.e.~in~$\Co \Omega$. By Theorem~\ref{minareboundedthm} we have that~$u \in L^\infty(\Omega)$, while Theorem~\ref{equiv_intro} gives that~$u$ is smooth inside~$\Omega$ and that its subgraph~$\S_u$ is locally~$s$-minimal in~$\Omega^\infty$.

Let now~$E \subseteq \R^{n + 1}$ be a locally~$s$-minimal set in~$\Omega^\infty$ such that~$E \setminus \Omega^\infty = \S_{\varphi} \setminus \Omega^\infty$. In view of our previous remark,~$E$ satisfies~\eqref{Eisbounded} for some~$M_0 > 0$. Consequently, we may apply to it Theorem~\ref{Persdecreases} and infer that~$E$ is the subgraph of a function~$v \in \B_{M_0} \W^s_\varphi(\Omega)$. Since, by Theorem~\ref{equiv_intro},~$v$ is a minimizer of~$\F_s$ in~$\Omega$, we conclude that~$u = v$ a.e.~in~$\R^n$. The proof is thus complete.
\end{proof}

To conclude the proof of Theorem~\ref{last_Theorem}, we are only left to deal with the uniform continuity of the minimizer~$u$ in~$\Omega$.

\begin{prop} \label{unifcontprop}
Let~$\Omega \subseteq \R^n$ be a bounded open set with boundary of class~$C^2$ and~$u$ be a minimizer of~$\F_s$ in~$\Omega$. If~$u = \varphi$ in~$\Co \Omega$, with~$\varphi: \Co \Omega \to \R$ such that~$\varphi \in C(\Omega_r \setminus \Omega)$ for some~$r > 0$, then~$u|_{\Omega}$ can be extended to a function~$\bar{u} \in C(\overline{\Omega})$.
\end{prop}

\begin{proof}
First of all, since~$u$ is a minimizer of~$\F_s$ in~$\Omega$, Theorem~\ref{equiv_intro}
ensures that~$u\in C^\infty(\Omega)$. To obtain that~$u$ is continuous up to the boundary of~$\Omega$, we thus only need to show that
\begin{equation} \label{bdary_reg_claim_proof}
\mbox{for every } x \in \partial \Omega, \mbox{ the limit } 
 \ell(x) := \lim_{\Omega \ni y\to x} u(y) \mbox{ exists and is finite}.
\end{equation}
Indeed, if this is the case, then it is easy to see that~$\ell \in C(\partial \Omega)$ and thus that the extension of~$u|_\Omega$ by~$\ell$ defines a continuous function in the whole~$\overline{\Omega}$.

To prove~\eqref{bdary_reg_claim_proof}, we first observe that~$u\in L^\infty(\Omega)$, thanks to Propositions~\ref{Linftylocprop} and~\ref{Bdary_Bdedness_prop}. Hence,
\bgs{
\ell^-(x):=\liminf_{\Omega \ni y \to x} u(y)\in\R
\quad\mbox{and}\quad
\ell^+(x):=\limsup_{\Omega \ni y \to x} u(y)\in\R,
}
for every~$x\in\partial\Omega$. Claim~\eqref{bdary_reg_claim_proof} boils down to showing that~$\ell^-(x)=\ell^+(x)$.

We argue by contradiction and suppose that~$\ell^-(x_0) < \ell^+(x_0)$ at some~$x_0\in\partial\Omega$. Then, at least one between~$\ell^-(x_0)$ and~$\ell^+(x_0)$
is different from~$\varphi(x_0)$. Without loss of generality, we assume that~$\ell^-(x_0)<\varphi(x_0)$. By this and the continuity of~$\varphi$, there exists~$\delta>0$ such that~$\ell^-(x_0)<\varphi(x)$ for every~$x\in B_\delta(x_0)\setminus \Omega$. Thus, setting~$X_0:=(x_0,\ell^-(x_0))$, we have that
$$
\mathcal B_\varrho(X_0)\setminus \Omega^{\infty}\subseteq\S_u,
$$
for a small~$\varrho>0$. Also observe that, as a consequence of the definition of~$\ell^-(x_0)$, we have~$X_0\in\partial\S_u$. Therefore, we can apply~\cite[Theorem~5.1]{graph}, which gives that~$\partial\S_u$ is
of class~$C^{1,\frac{1+s}{2}}$ in~$\mathcal B_\varrho(X_0)$, up to taking a smaller~$\varrho$.

Write~$X_t := X_0 + t e_{n + 1}$. We claim that
\begin{equation} \label{vertsegclaim}
X_t \in \partial \S_u \mbox{ and } H_s[\S_u](X_t) = 0 \mbox{ for every } t \in \left[ 0, \frac{\varrho}{2} \right].
\end{equation}
To see that~$X_t \in \partial \S_u$, it suffices to observe that for every~$\ell \in [\ell^-(x_0), \ell^+(x_0)]$, there exists a sequence of points~$\{ y_k \} \subseteq \Omega$ converging to~$x_0$ and such that~$u(y_k) = \ell$ for all~$k \in \N$. This last fact can be easily deduced from the continuity of~$u$ inside~$\Omega$ and the regularity of~$\partial \Omega$. That~$H_s[\S_u](X_t) = 0$ also follows from this and the~$C^{1, \frac{1 + s}{2}}$ regularity of~$\partial \S_u$ in~$\mathcal{B}_\varrho(X_0)$, thanks to~\cite[Lemma~3.4]{graph}.

Claim~\eqref{bdary_reg_claim_proof} is now a consequence of the strong comparison principle. Indeed, using~\eqref{vertsegclaim} and a suitable change of variables, we get that
\begin{align*}
0 & = H_s[\S_u](X_t) - H_s[\S_u](X_0) \\
& = \PV \int_{\R^{n + 1}} \frac{\chi_{\Co \S_u}(X_t + Z) - \chi_{\S_u}(X_t + Z) - \chi_{\Co \S_u}(X_0 + Z) + \chi_{\S_u}(X_0 + Z)}{|Z|^{n + 1 + s}} \, dZ \\
& = 2 \, \PV \int_{\R^{n + 1}} \frac{\chi_{\S_u \setminus \S_{u - t}}(X_0 + Z)}{|Z|^{n + 1 + s}} \, dZ,
\end{align*}
for all~$t \in [0, \varrho/2]$. Since this is impossible, we conclude that~\eqref{bdary_reg_claim_proof} must hold true and the proof is thus complete.
\end{proof}

\appendix

\section{A density result and a Hardy-type inequality} \label{app}

\noindent
We include here a few auxiliary results that have been used throughout the previous sections. First, we have the following known result about the density of smooth functions in fractional Sobolev spaces.

\begin{prop}\label{smooth_cpt_dense}
Let~$s \in (0, 1)$ and~$p \ge 1$ be such that~$s p < 1$. Let~$\Omega\subseteq\Rn$ be a bounded open set with Lipschitz boundary and~$u \in W^{s, p}(\Omega)$. Then, there exists a sequence~$\{ u_k \} \subseteq C^\infty_c(\Omega)$ which converges to~$u$ in~$W^{s, p}(\Omega)$. Furthermore, if~$a \le u \le b$ a.e.~in~$\Omega$ for some~$- \infty \le a \le 0 \le b \le +\infty$, then we can choose the~$u_k$'s in such a way that also~$a \le u_k \le b$ in~$\Omega$ for every~$k \in \N$.
\end{prop}
\begin{proof}
First of all, we observe that it suffices to consider the case of a bounded~$u$, as the statement in its generality can then be proved easily via truncations. Hence, we assume that
\begin{equation} \label{aleuleb}
a \le u \le b \mbox{ a.e.~in~} \Omega, \mbox{ for some } a \in (- \infty, 0] \mbox{ and } b \in [0, +\infty).
\end{equation}

Secondly, we may further restrict to~$u$'s with support compactly contained in~$\Omega$. Indeed, suppose that the result holds true for all such functions. Then, given any general~$u \in W^{s, p}(\Omega)$, we take~$\delta > 0$ small and define~$v_\delta := u \chi_{\Omega_{- \delta}}$---recall~\eqref{Omegarhodef}. Clearly,~$\supp(v_\delta) \subseteq \Omega_{- \delta} \Subset \Omega$, and thus we can find~$u_\delta \in C^\infty_c(\Omega)$ such that~$a \le u_\delta \le b$ in~$\Omega$ and~$\| v_\delta - u_\delta \|_{W^{s, p}(\Omega)} \le \delta$. We claim that~$\lim_{\delta \rightarrow 0} \| u - v_\delta \|_{W^{s, p}(\Omega)} = 0$. To see this, notice that~$\| u - v_\delta \|_{L^p(\Omega)} \le \| u \|_{L^\infty(\Omega)} |\Omega \setminus \Omega_{- \delta}| \rightarrow 0$ as~$\delta \searrow 0$, thanks to the boundedness~$\Omega$ and the Lipschitz regularity of its boundary. On the other hand,
\begin{align*}
[ u - v_\delta ]_{W^{s, p}(\Omega)}^p & = \int_{\Omega \setminus \Omega_{- \delta}} \int_{\Omega \setminus \Omega_{- \delta}} \frac{|u(x) - u(y)|^p}{|x - y|^{n + s p}} \, dx \, dy + 2 \int_{\Omega \setminus \Omega_{- \delta}} |u(x)|^p \left( \int_{\Omega_{- \delta}} \frac{dy}{|x - y|^{n + s p}} \right) dx \\
& \le \int_{\Omega \setminus \Omega_{- \delta}} \int_{\Omega \setminus \Omega_{- \delta}} \frac{|u(x) - u(y)|^p}{|x - y|^{n + s p}} \, dx \, dy + 2 \| u \|_{L^\infty(\Omega)}^p \int_{\Omega \setminus \Omega_{- \delta}} \int_{\Omega_{- \delta}} \frac{dx \, dy}{|x - y|^{n + s p}}.
\end{align*}
Both summands converge to zero as~$\delta\searrow0$: the first by the continuity of the Lebesgue integral, the second by~\cite[Lemma~2.7(i)]{Cyl} and the fact that~$s p < 1$.

We thus take~$u \in W^{s, p}(\Omega)$ satisfying~\eqref{aleuleb} and~$\supp(u) \subseteq \Omega'$ for some open set~$\Omega' \Subset \Omega$. We also call~$u$ its extension to~$0$ outside of~$\Omega$---it is easy to see that~$u \in W^{s, p}(\R^n)$. Let~$\eta$ be a standard mollifier, i.e., a non-negative function~$\eta \in C_c^\infty(\R^n)$ such that~$\supp(\eta) \subseteq B_1$ and~$\| \eta \|_{L^1(\R^n)} = 1$. For~$\varepsilon > 0$, define~$\eta_\varepsilon(x) := \varepsilon^{-n} \eta(x/\varepsilon)$ and~$u_\varepsilon(x) := (u \ast \eta_\varepsilon)(x)$ for every~$x \in \R^n$. We have that~$u_\varepsilon \in C_c^\infty(\R^n)$,~$a \le u_\varepsilon \le b$ in~$\R^n$, and~$\supp(u_\varepsilon) \Subset \Omega$, provided~$\varepsilon \le \dist(\Omega', \partial \Omega) / 2$. It is well-known that~$u_\varepsilon \rightarrow u$ in~$L^p(\Omega)$ as~$\varepsilon \searrow 0$, and the convergence is actually in~$W^{s, p}(\Omega)$. This is probably well-known too. Nevertheless, we reproduce here the argument of~\cite[Lemma~11]{FSV15} for the convenience of the reader. By H\"older's inequality, we estimate
\begin{align*}
[u - u_\varepsilon]^p_{W^{s, p}(\Omega)} & \le \int_{\R^n} \int_{\R^n} \frac{|u(x) - u_\varepsilon(x) - u(y) + u_\varepsilon(y)|^p}{|x - y|^{n + s p}} \, dx \, dy \\
& \le \int_{\R^n} \int_{\R^n} \left( \int_{B_1} \left| u(x) - u(y) - u(x - \varepsilon z) + u(y - \varepsilon z) \right| \eta(z) \, dz \right)^p \frac{dx \, dy}{|x - y|^{n + s p}} \\
& \le |B_1|^{p - 1} \int_{B_1} \eta(z)^p \psi_\varepsilon(z) \, dz,
\end{align*}
with
$$
\psi_\varepsilon(z) := \left\| \tau_{\varepsilon z} V - V \right\|_{L^p(\R^n \times \R^n)}^p, \quad V(x, y) := \frac{u(x) - u(y)}{|x - y|^{\frac{n + s p}{p}}}, \quad \mbox{and} \quad \tau_w V(x, y) := V(x - w, y - w).
$$
Since~$V \in L^p(\R^n \times \R^n)$, by the continuity of translations in~$L^p$ we have that~$\psi_\varepsilon(z) \rightarrow 0$ as~$\varepsilon \searrow 0$, for every~$z \in B_1$. As~$|\psi_\varepsilon| \le 2^p [u]_{W^{s, p}(\R^n)}$ in~$B_1$, using Lebesgue's dominated convergence theorem we conclude that~$u_\varepsilon \rightarrow u$ in~$W^{s, p}(\Omega)$.
\end{proof}

Next, we have the following fractional Hardy-type inequality. This inequality is probably well-known to the expert reader---it is stated for instance in~\cite{Dy04}, see formula~(17) there. However, since its proof does not seem easily accessible in the literature, we provide a simple argument based on the fractional Hardy inequality on half-spaces established in~\cite{FS10}.

\begin{prop}\label{FHI}
	Let~$s \in (0, 1)$ and~$p \ge 1$ be such that~$s p < 1$. Let~$\Omega \subseteq \R^n$ be a bounded open set with Lipschitz boundary. Then, there exists a constant~$C > 0$, depending only on~$n$,~$s$,~$p$, and~$\Omega$, such that
	\begin{equation} \label{CH:ApP:hardyine}
	\int_\Omega \frac{|u(x)|^p}{\dist(x, \partial \Omega)^{s p}} \, dx \le C \| u \|_{W^{s, p}(\Omega)}^p
	\end{equation}
	for every~$u \in W^{s, p}(\Omega)$.
\end{prop}

\begin{proof}
	In light of Proposition~\ref{smooth_cpt_dense}, we can restrict ourselves to consider~$u\in C^\infty_c(\Omega)$.
	
	Let~$\{ B^{(j)} \}_{j = 1}^N$ be a sequence of balls of the form~$B^{(j)} = B_{r}(x^{(j)})$, with~$N \in \N$,~$x^{(j)} \in \partial \Omega$, and~$r > 0$, for which there exist bi-Lipschitz homeomorphisms
	$$
	T_j : B'_2 \times (-2, 2) \longrightarrow 2 B^{(j)} := B_{2 r}(x^{(j)})
	$$
	satisfying
	\begin{align*}
	T_j(U_2) & = 2 B^{(j)}, && \hspace{-50pt} \mbox{with } U_2 := B'_2 \times (-2, 2),\\
	T_j(U^+_2) & = \Omega \cap 2 B^{(j)}, && \hspace{-50pt} \mbox{with } U^+_2 := B_2' \times (0, 2),\\
	T_j(U^0_2) & = \partial \Omega \cap 2 B^{(j)}, && \hspace{-50pt} \mbox{with } U^0_2 := B_2' \times \{ 0 \},
	\end{align*}
	and such that~$\partial \Omega \subseteq \cup_{j = 1}^N B^{(j)}$. Here, for~$\varrho > 0$ we write~$B'_\varrho := \left\{ x' \in \R^{n - 1} : |x'| < \varrho \right\}$.
	
	Let~$\varepsilon > 0$ be such that~$\Omega \setminus \cup_{j = 1}^N B^{(j)} \Subset \Omega_{- \varepsilon}$ and set~$B^{(0)} := \Omega_{- \varepsilon}$. Clearly,
	\begin{equation} \label{CH:ApP:B0covered}
	\int_{B^{(0)}} \frac{|u(x)|^p}{d_{\partial\Omega}(x)^{s p}} \, dx \le \varepsilon^{- s p} \int_{B^{(0)}} |u(x)|^p \, dx \le C \| u \|_{L^p(\Omega)}^p,
	\end{equation}
	where~$d_{\partial\Omega}(x) := \dist(x, \partial \Omega)$ for every~$x \in \Omega$ and, from now on,~$C$ denotes any constant larger than~$1$, whose value depends at most on~$n$,~$s$,~$p$, and~$\Omega$.

	Notice that~$\{ B^{(j)} \}_{j = 0}^N$ is an oper cover of~$\Omega$ and let~$\{ \eta_j \}_{j = 0}^N$ be a smooth partition of unity on~$\Omega$ subordinate to~$\{ B^{(j)} \}_{j = 0}^N$. For~$j =1, \ldots, N$, we define~$v_j := \eta_j u \in C^\infty_c(\Omega\cap B^{(j)})$. Changing variables through~$T_j$, we have
	$$
	\int_{\Omega \cap B^{(j)}} \frac{|v_j(x)|^p}{d_{\partial\Omega}(x)^{s p}} \, dx = \int_{T_j^{-1}(\Omega \cap B^{(j)})} \frac{|v_j(T_j(\bar{x}))|^p}{d_{\partial\Omega}(T_j(\bar{x}))^{s p}} \left| \det D T_j(\bar{x}) \right| d\bar{x}.
	$$
	Notice that for every~$x \in \Omega \cap B^{(j)}$ there exists~$D_j(x) \in \partial \Omega \cap 2 B^{(j)}$ such that~$d_{\partial\Omega}(x) = \left| x - D_j(x) \right|$. Since~$T_j$ is bi-Lipschitz and~$T_j^{-1}(D_j(x)) \in B_2' \times \{ 0 \}$, we have
	$$
	\begin{aligned}
	d_{\partial\Omega}(T_j(\bar{x})) & = |T_j(\bar{x}) - D_j(T_j(\bar{x}))| = |T_j(\bar{x}) - T_j(T_j^{-1}(D_j(T_j(\bar{x}))))| \\
	& \ge C^{-1} |\bar{x} - T_j^{-1}(D_j(T_j(\bar{x})))| \ge C^{-1} \bar{x}_n
	\end{aligned}
	$$
	for every~$\bar{x} \in T_j^{-1}(\Omega \cap B^{(j)})$. Accordingly, writing~$w_j := v_j \circ T_j$ we get
	$$
	\int_{\Omega \cap B^{(j)}} \frac{|v_j(x)|^p}{d_{\partial\Omega}(x)^{s p}} \, dx \le C \int_{U^+_2} \frac{|w_j(\bar{x}))|^p}{\bar{x}_n^{s p}} \, d\bar{x}.
	$$
	
	Let us observe that $w_j$ is supported inside~$T_j^{-1}(\Omega \cap B^{(j)})$.
	We now employ the fractional Hardy inequality on half-spaces of~\cite[Theorem~1.1]{FS10} and deduce that
	\begin{equation} \label{CH:ApP:hardytech1}
	\int_{\Omega \cap B^{(j)}} \frac{|v_j(x)|^p}{d_{\partial\Omega}(x)^{s p}} \, dx \le C \int_{\R^n_+} \int_{\R^n_+} \frac{|w_j(\bar{x}) - w_j(\bar{y})|^p}{|\bar{x} - \bar{y}|^{n + s p}} \, d\bar{x}\, d\bar{y},
	\end{equation}
	where~$\R^n_+ = \{ z \in \R^n \,|\, z_n > 0 \}$ and it is understood that $w_j$
	is extended by 0 in $\R^n_+\setminus U_2^+$. We point out that---since $T_j^{-1}(B^{(j)}) \Subset U_2$ and $T_j^{-1}(\Omega \cap B^{(j)})\subseteq U_2^+$---we have
	\[
	\mbox{dist}\big(T_j^{-1}(\Omega \cap B^{(j)}),\R^n_+ \setminus U_2^+\big)>0.
	\]
	Thus, using that~$w_j$ is supported inside~$T_j^{-1}(\Omega \cap B^{(j)})$,
	we estimate
	\begin{equation} \label{CH:ApP:hardytech2}
	\begin{aligned}
	\int_{\R^n_+} \int_{\R^n_+} \frac{|w_j(\bar{x}) - w_j(\bar{y})|^p}{|\bar{x} - \bar{y}|^{n + s p}} \, d\bar{x}\, d\bar{y} & \le \int_{U_2^+} \int_{U_2^+} \frac{|w_j(\bar{x}) - w_j(\bar{y})|^p}{|\bar{x} - \bar{y}|^{n + s p}} \, d\bar{x} \,d\bar{y} \\
	& \quad + 2 \int_{T_j^{-1}(\Omega \cap B^{(j)})} \left( \int_{\R^n_+ \setminus U_2^+} \frac{ |w_j(\bar{x})|^p}{|\bar{x} - \bar{y}|^{n + s p}} \, d\bar{y} \right) d\bar{x} \\
	& \le \int_{U_2^+} \int_{U_2^+} \frac{|w_j(\bar{x}) - w_j(\bar{y})|^p}{|\bar{x} - \bar{y}|^{n + s p}} \, d\bar{x}\, d\bar{y} + C \| w_j \|_{L^p(U_2^+)}^p.
	\end{aligned}
	\end{equation}
	
	By combining~\eqref{CH:ApP:hardytech1} with~\eqref{CH:ApP:hardytech2} and switching back to the variables in~$\Omega$, we easily find that
	$$
	\int_{\Omega \cap B^{(j)}} \frac{|v_j(x)|^p}{d_{\partial\Omega}(x)^{s p}} \, dx \le C \left( \int_{\Omega \cap 2 B^{(j)}} \int_{\Omega \cap 2 B^{(j)}} \frac{|v_j(x) - v_j(y)|^p}{|x - y|^{n + s p}} \, dx \,dy + \| v_j \|_{L^p(\Omega \cap 2 B^{(j)})}^p \right).
	$$
	Recalling that~$v_j = \eta_j u$ and~$\eta_j$ is Lipschitz, a simple computation then leads us to
	$$
	\int_{\Omega \cap B^{(j)}} \frac{|v_j(x)|^p}{d_{\partial\Omega}(x)^{s p}} \, dx \le C \| u \|_{W^{s, p}(\Omega)}^p \quad \mbox{for all } j = 1, \ldots, N.
	$$
	Estimate~\eqref{CH:ApP:hardyine} follows by putting together this with~\eqref{CH:ApP:B0covered} and using that~$\{ \eta_j \}$ is a partition of unity.
\end{proof}

A simple consequence of the previous Hardy inequality is the following estimate, which actually gives that~$\| \cdot \|_{W^{s, 1}(\R^n)}$ and~$\| \cdot \|_{W^{s, 1}(\Omega)}$ are equivalent norms for the space~$\W^{s}_0(\Omega)$ introduced in~\eqref{domain_w_data}.

\begin{corollary}\label{FHI_corollary}
	Let~$s \in (0, 1)$ and~$p \ge 1$ be such that~$s p < 1$. Let~$\Omega\subseteq\Rn$ be a bounded open set with Lipschitz boundary. Then, there exists a constant~$C > 0$, depending only on~$n$,~$s$,~$p$, and~$\Omega$, such that
	\bgs{
		\int_\Omega\left(|u(x)|^p\int_{\Co\Omega}\frac{dy}{|x-y|^{n+sp}} \right)dx\le
		C \|u\|^p_{W^{s,p}(\Omega)},
	}
	for every $u\in W^{s,p}(\Omega)$.
\end{corollary}
\begin{proof}
	The inequality follows immediately from the estimate
	$$
	\int_{\Co\Omega}\frac{dy}{|x-y|^{n+sp}} \le \int_{\Co B_{\dist(x, \partial \Omega)}}\frac{dz}{|z|^{n+sp}} = \frac{\Ha^{n - 1}(\S^{n - 1})}{s p} \hspace{1pt} \dist(x, \partial \Omega)^{- s p},
	$$
	which holds for every~$x \in \Omega$, and Proposition~\ref{FHI}.
\end{proof}

\end{document}